\documentclass[oneside,12pt,reqno]{amsart}

\usepackage{amssymb}
\usepackage{amsmath,amsthm,amscd}
\usepackage{mathtools}
\usepackage{verbatim}
\usepackage{graphicx,hyperref}
\usepackage[all]{xy}

\usepackage[usenames,dvipsnames]{xcolor}
\usepackage{amsthm, amsfonts, graphicx, pinlabel, tikz}
\usepackage[all]{xy}
\usepackage{hyperref} 
\usepackage{enumitem}

\newcommand{\HH}{\mathbf{H}}

\newcommand{\kk}{\mathbf{k}}
\newcommand{\C}{\mathbb{C}}
\newcommand{\CP}{\mathbb{C}P}
\newcommand{\RP}{\mathbb{R}P}
\newcommand{\R}{\mathbb{R}}
\newcommand{\Q}{\mathbb{Q}}
\newcommand{\Z}{\mathbb{Z}}

\newcommand{\id}{\mathrm{Id}}
\newcommand{\im}{\mathrm{im}}

\newcommand{\OP}{\operatorname}

\renewcommand{\Re}[1]{\mathfrak{Re}\,#1}

\renewcommand{\Re}{\mathfrak{Re}}

\theoremstyle{plain}

\newtheorem{thm}{Theorem}[section]
\newtheorem{cor}[thm]{Corollary}
\newtheorem{lem}[thm]{Lemma}

\newtheorem{prop}[thm]{Proposition}

\theoremstyle{definition}
\newtheorem{defn}[thm]{Definition}

\theoremstyle{remark}
\newtheorem{rem}[thm]{Remark}
\newtheorem{ex}[thm]{Example}

\numberwithin{equation}{section}
 
\title[The persistence of a relative Rabinowitz-Floer complex]{The persistence of a relative Rabinowitz-Floer complex}

\author{Georgios Dimitroglou Rizell}
\address{Department of Mathematics\\
Uppsala University\\
Box 480\\
SE-751 06 Uppsala\\
Sweden}
\email{georgios.dimitroglou@math.uu.se}

\author{Michael G. Sullivan}
\address{Department of Mathematics and Statistics\\
University of Massachusetts\\
Amherst\\
MA 01003\\
USA}
\email{sullivan@math.umass.edu}

\thanks{The first author is supported by the Knut and Alice Wallenberg Foundation under the grant KAW 2016.0198, and by the Swedish Research Council under the grant 2020-04426. The second author is supported by the Simons Foundation grant 708337. The authors thank useful feedback on earlier versions by Yasin Karacan and Stefan Nemirovski.}

\begin{document}

\begin{abstract}
We give a quantitative refinement of the invariance of the Legendrian contact homology algebra in general contact manifolds. We show that in this general case, the Lagrangian cobordism trace of a Legendrian isotopy defines a DGA stable tame isomorphism which is similar to a bifurcation invariance-proof for a contactization contact manifold. We use this result to construct a relative version of the Rabinowitz-Floer complex defined for Legendrians that also satisfies a quantitative invariance, and study its persistent homology barcodes. We apply these barcodes to prove several results, including: displacement energy bounds for Legendrian submanifolds in terms of the oscillatory norms of the contact Hamiltonians; a proof of Rosen and Zhang's non-degeneracy conjecture for the Shelukhin--Chekanov--Hofer metric on Legendrian submanifolds; and, the non-displaceability of the standard Legendrian real-projective space inside the contact real-projective space.
\end{abstract}

\maketitle

\section{Introduction}
\label{sec:Introduction}

Let $(Y^{2n+1}, \xi)$ be a $(2n+1)$-dimensional contact manifold with contact form $\alpha,$ and $\Lambda \subset Y$ be a (closed) $n$-dimensional Legendrian submanifold.
A {\bf Reeb chord} (or {\bf{$\alpha$-Reeb chord}}) of $\Lambda$ is a non-trivial flow starting and ending on $\Lambda,$ of the Reeb vector field $R_\alpha \in \Gamma(TY),$ which is defined by $\alpha({R_\alpha})=1$ and $d\alpha({R_\alpha}, \cdot)=0.$
We are interested in estimating the number of Reeb chords from a given Legendrian (closed submanifold) $\Lambda$ to its image under a contact isotopy with compact support. If there are no such Reeb chords, we say that the contact isotopy {\bf displaces} $\Lambda$ for that given $\alpha.$ 
This is the contact analogue of a Hamiltonian isotopy displacing a Lagrangian submanifold \cite{ChekanovDisplacement}.
Our Main Theorem \ref{thm:Main} has more information about Reeb chords than the known analogous results for Lagrangian intersection points, because we not only give a single lower bound on how long some fixed number of chords persist, but rather, a sequence of lower bounds depending on the number of chords required to persist.

Our main tool is a filtered {\bf{Legendrian contact homology differential graded algebra}} (also called the {\bf{Chekanov-Eliashberg DGA}}).
Let $\mathcal{A}(\Lambda)$ be the free non-commutative unital algebra over the field (or ring)  $\kk$ freely generated by $\alpha$-Reeb chords of $\Lambda.$ If the moduli spaces of these disks can be oriented in a coherent way, e.g.~by the choice of a spin structure on the Legendrian as in \cite{EES05c,Karlsson}, then $\kk$ is $\Z$ or $\Z_p.$ Otherwise $\kk = \Z_2.$

The grading is induced by the Conley--Zehnder index of Reeb chords;  see Appendix \ref{sec:puregrading}.

The differential $\partial$ has degree $-1$ and counts $J$-holomorphic disks in the symplectization $(\R_\tau \times Y, d(e^\tau \alpha))$ with Lagrangian boundary condition $\R_\tau \times \Lambda.$ 

%

Each Reeb chord $c$ has a {\bf{length}} (or {\bf{action}}) $ \ell(c) := \int_c \alpha>0.$
For $0 < l \le \infty,$
let $\mathcal{A}^l(\Lambda)$ be the unital sub-algebra generated by those generators $c$ length bounded from above by $\int_c \alpha <l.$
The action-decreasing property of the differential, which is a direct consequence of the positivity of $d\alpha$-area and Stokes' theorem, implies that $\mathcal{A}^l(\Lambda) \subset \mathcal{A}(\Lambda)$ is a unital sub-DGA.


An {\bf{augmentation}} for the DGA $\mathcal{A}^l$  is a (graded) DGA-morphism to the ground field, $\varepsilon: (\mathcal{A}^l, \partial) \rightarrow (\kk,\partial_\kk := 0).$
Because $\ell(1)=0,$ for any Chekanov-Eliashberg DGA $\mathcal{A},$ there exists $l>0$ such that $\mathcal{A}^l$ has an augmentation.

The {\bf{oscillation}} of a contact ($\alpha$-)Hamiltonian $H_t:Y \times \R_\tau \rightarrow \R$ is
\[
\|H_t\|_{\OP{osc}} 
:= \int_0^1 \left(\max_{y \in Y} H_t - \min_{ y \in Y} H_t\right)dt
\]
This oscillation defines the Hofer norm of the corresponding contact Hamiltonian isotopy $\phi^t_{\alpha,H_t}$



In order to circumvent the analytical difficulties of establishing invariance of the Legendrian contact homology algebra for general contact forms, we will make certain technical assumptions on the contractible periodic Reeb orbits of $(Y,\alpha)$, at least below some fixed length. To this purpose, we introduce the following definition.

\begin{defn}
  \label{acs}Consider a contractible and non-degenerate periodic Reeb orbit $\gamma$ of $(Y,\alpha)$. We let  $|\gamma| \in \Z \cup \{-\infty\}$ denote the \emph{minimum} of the expected dimensions of the moduli spaces of unparameterized  pseudoholomorphic planes inside the symplectization  $\R \times Y$ that are asymptotic to the Reeb orbit $\gamma$ at the convex end, where the symplectization has been equipped with a cylindrical almost complex structure. In the case when the first Chern class vanishes, this expected dimension does not depend on the chosen plane, and $|\gamma| \in \Z$.
\end{defn}
Note that, in the aforementioned moduli space, we do not identify planes that differ by a translation of the symplectization coordinate. See Appendix \ref{sec:gradingperiodic} for more details. 
\begin{ex} 
{{The expected dimension $|\gamma|$ of a contractible Reeb orbit is at least two}}
for suitable non-degenerate perturbations of the round contact sphere
  $$\left(S^{2n+1},\alpha_{st}=\frac{1}{2}\sum_i(x_idy_i-y_idx_i) \right), \:\: n \ge 1,$$
as well as for the high dimensional ``lens-spaces'' given as the quotients $S^{2n+1}/\Z_k$ for a subgroup $\Z_k \subset S^1$. We will here consider the case $\RP^{2n+1}=S^{2n+1}/\Z_2$;  see Proposition \ref{prop:planedegree} for the relevant index computation in the case of $\RP^{2n+1}$. The computations for the sphere and lens spaces are analogous.
  \end{ex}
The assumption $|\gamma| >1$ for all contractible Reeb orbits, allows us to define the Legendrian contact homology algebra without involving the contact homology algebra for the periodic Reeb orbits; see \cite[Section 3.3.3]{GDRSurgery}. (Note that the contact homology algebra of periodic Reeb orbits has a canonical augmentation in this case.) The assumption $|\gamma|>1$ also eliminates the need for considerations of the periodic Reeb orbits in the proof of the invariance result from \cite{Ekholm} for the Legendrian contact homology algebra under a Legendrian isotopy (while fixing the ambient contact form). 

Before we state our results in detail, we give a quick summary, noting how general the set-up is in light of the above technical discussion on the Legendrian contact homology algebra. 
\begin{itemize}
\item
Theorem \ref{thm:Main} proves lower bounds for the number of Reeb chords between a Legendrian and its image under a contact Hamiltonian in terms of the oscillation norm. This is for an arbitrary Legendrian in an arbitrary contact manifold, but the bounds incorporate the technical condition on 
$|\gamma|$ mentioned above.
\item
Theorem \ref{thm:NonDegen} proves that the Shelukhin--Chekanov--Hofer metric of a Legendrian orbit space is non-degenerate. This is for an arbitrary Legendrian in an arbitrary contact manifold.
\item
Corollary \ref{cor:Nakumara} proves that the $C^0$-limit of a sequence of Legendrians is again Legendrian. Here we assume the contact manifold is geometrically bounded and there exists some lower bound on the length of Reeb chords in the sequence. But there is no assumption on the closed Reeb orbits.
\item
Theorem \ref{thm:Interlink} generalizes the ``inter-linkedness" of an ordered pair of Legendrians \cite[Theorem 1.5]{EntovPolterovich2021a}. We make several assumptions here: the two Legendrians are individually augmentable; $|\gamma| >1$ holds for all closed Reeb orbits; and the resulting well-defined Rabinowitz Floer complex is not acyclic.
\item
Theorem \ref{thm:RP} proves the non-displaceability of the Legendrian equator in standard contact $\RP^{2n+1}$ equipped with a small perturbation of the standard round $S^1$-invariant contact form. We will show that the $|\gamma| >1$ assumption is satisfied because the perturbation is small.
\item Proposition \ref{prop:Bifurcation} proves roughly that the Legendrian contact homology algebra, below any fixed action level, of an arbitrary Legendrian undergoing a small generic isotopy in an arbitrary contact manifold is invariant under stable-tame isomorphism. Here we assume $|\gamma| >1$ holds for those closed orbits below this action level bound.
\end{itemize}

\begin{thm}[Main Theorem]
\label{thm:Main}
Fix a generic closed Legendrian $\Lambda \subset (Y,\alpha)$ of a contact manifold with a fixed contact form.  Generically, we can order the Reeb chords by action
$$0< \ell(c_1) < \ell(c_2) < \ldots.$$
Also, write $\hbar \in (0,+\infty]$ for the minimal length of a contractible periodic Reeb orbit $\gamma$ that satisfies $|\gamma| \le 1$. Suppose that $\mathcal{A}^l(\Lambda)$ with $0 < l \le \infty$ admits an augmentation to the field $\kk$ where $l  \le  \hbar$.
Fix $k$ and consider any compactly supported contact Hamiltonian $H_t \colon Y \to \R$ such that  $\|H_t\|_{\OP{osc}} < \min\left\{l, \ell(c_k)\right\}$
Then there exists at least 
$$
\sum_{i=0}^n \dim\left(H_{i}(\Lambda; \kk)\right)-2(k-1)
$$
many Reeb chords with one endpoint on $\Lambda$ and the other endpoint on $\phi^1_{\alpha,H_t}(\Lambda).$
\end{thm}



\begin{rem}
\label{rem:ClosedOrbits}
The assumption on the expected dimension of the pseudoholomorphic planes inside the symplectization should not be strictly necessary, but possible to replace by a condition on the existence of an augmentation of the contact homology algebra of periodic Reeb orbits below the action level $l > 0$. However, the setup and analysis become significantly simplified under our stronger assumption. More precisely, without these assumptions, one would have to define the Chekanov--Eliashberg algebra with coefficients in the periodic orbit contact homology algebra in order to define the differential; or, in the case when $(Y,\alpha)$ admits an exact symplectic filling, use anchored disk as in \cite{Ekholm19}. More importantly, our additional hypothesis allows us to avoid the technical gluing and transversality results for pseudoholomorphic planes asymptotic to periodic orbits that require the use of virtual perturbations as done by Pardon \cite{Pardon}, Bao--Honda \cite{BaoHonda}, and the Polyfold theory of Hofer--Wysocki--Zehnder applied in the SFT setting \cite{Polyfolds} by Fish--Hofer. With the additional hypothesis $ l \le  \hbar $, we only need SFT compactness for closed Reeb orbits \cite{BEHWZ}, and the gluing/transversality results for pseudoholomorphic discs with a single positive Reeb chord asymptotic.

Finally note that, even when replacing the assumption on the expected dimension of the pseudoholomorphic half-planes by the existence of an augmentation, the contact form remains fixed throughout our Legendrian isotopies. Hence, the difficult analytical issues that arise when proving the invariance under deformations of the contact form would not be present even in the more general case.
\end{rem}

In \cite{DRS2},  we studied this problem when $Y = P \times \R_z$ and $\alpha = dz - \theta$ where $(P, d\theta)$ is an exact symplectic manifold with a certain bounded geometry at infinity. In both cases we used a persistence homology theory (via barcodes) defined by Reeb chords. The main difference between our setup in that article and the typical setup, is that we considered complexes with a finite action filtration.

The new aspects of this article are: we translate known Floer-continuation results to new (but needed) Floer-bifurcations results; we prove an invariance for a newly-defined limited-action-window Rabinowitz Floer homology theory, which allows for chords of zero-length to appear. 

We now describe several new applications of the Main Theorem \ref{thm:Main}.

Fix a subset $N \subset Y$ and let $\mathcal{L}(N)$ be its orbit space under the identity component of the contactomorphism group, $\mbox{C}_0(Y, \xi).$ 
Given a contact isotopy $\phi^t$  and contact form $\alpha,$ let $H^{\alpha,\phi}_t$ be the contact Hamiltonian for $\phi^t.$
Following Rosen and Zhang's \cite[Definition 1.7]{RosenZhang}, define the pseudo-metric on $\mathcal{L}(N)$
$$\delta_\alpha(N, N') = 
\inf \left\{ \int_0^1\max |H_t|dt; \,\, \phi^1_{\alpha,H_t}(N) = N' \right\}. 
$$
Shelukhin shows that this defines a right-invariant non-degenerate norm on $\mbox{C}_0(Y, \xi)$
\cite{Shelukhin17JSG}.
 If $N$ is $n$-dimensional and somewhere not Legendrian (or more generally, if $N$ is not contact coisotropic) then this pseudo-metric identically vanishes \cite[Proposition 7.4]{RosenZhang}. 
We answer \cite[Conjecture 1.10]{RosenZhang}.

\begin{thm}
\label{thm:NonDegen}
The Shelukhin--Chekanov--Hofer distance $\delta_\alpha$  on a contact manifold that is assumed to be either closed, or which admits a codimension-zero contact embedding into a closed contact manifold,  is non-degenerate when restricted to closed Legendrian submanifolds.
\end{thm}
\begin{rem}
  \begin{enumerate}
\item The condition on admitting an embedding into a closed contact manifold holds, for example, for contactizations and prequantizations of Liouville domains;
\item  The non-degeneracy should hold for more general open contact manifolds which are of bounded geometry. For the proof the following features are crucial: the constant $\hbar \ge 0$ must be strictly positive for some contact form, and we must be in a setting in which the SFT compactness theorem for pseudoholomorphic curves is satisfied (we need e.g.~convexity assumptions at infinity).
  \end{enumerate}
  \end{rem}

Two special cases of this were already proved: when the Legendrian is hypertight (no contractible Reeb orbits or chords) \cite{Usher20} and when the Legendrian is orderable (no positive loops of Legendrians) \cite{Hedicke21}.

Consider  a sequence of closed Legendrians $\Lambda_i$ which $C^0$-converges to a smooth (not necessarily Legendrian) embedding $\Lambda_\infty.$ Assume that there exists a $\delta$ such that each $\Lambda_i$ has no Reeb chord $\gamma$ of $\alpha$-length $\ell(\gamma)< \delta.$
Nakamura proves that $\Lambda_\infty$ is again Legendrian assuming two conditions \cite[Theorem 3.4]{Nakamura}
\begin{enumerate}
\item
The Reeb vector field $R_\alpha$ is nowhere tangent to $\Lambda_\infty$ and certain geometric boundedness conditions hold for $M$.
\item
$(Y, \alpha) = (P \times \R_z, dz - \theta).$ 
\end{enumerate}
The only need for the second hypothesis in Nakamura's proof is to use \cite{DRS2} to show the Reeb chords persist under a contact isotopy of small oscillatory norm.
 Since our Theorem \ref{thm:Main} generalizes the persistence of Reeb chords proven in  \cite{DRS2} to more general contact manifolds, we get an easy corollary.
\begin{cor}
\label{cor:Nakumara}
Nakamura's convergence result holds only assuming the initial hypotheses: the Reeb chord lengths cannot approach 0, $\Lambda_i$ is closed, and the ambient contact manifold $(M,\alpha)$ is  either closed, or admits a codimension-zero contact embedding into a closed contact manifold.
\end{cor}

Entov and Polterovich define an ordered pair of  Legendrians $(\Lambda_0, \Lambda_1)$ in  $(Y, \alpha)$ to be 
{\bf{interlinked}} if for some $\mu, c>0$ and every contact Hamiltonian $H$ satisfying 
$H \ge c,$ there is a Hamiltonian orbit $\gamma$ from $\Lambda_0$ to $\Lambda_1$ such that 
$\ell(\gamma) \le \mu/c.$ Using a persistence homology theory for Reeb chords, they proved if $\Lambda_0$ is the $0$-section in 
$(Y = J^1Q, \alpha = dz-pdq),$  $\Lambda_1 \subset J^1Q$ has an augmentation, and there exists a unique non-degenerate Reeb chord from $\Lambda_0$ to  $\Lambda_1,$ then $(\Lambda_0, \Lambda_1)$ are interlinked \cite[Theorem 1.5]{EntovPolterovich2021a}.
We can generalize this.

\begin{thm}
\label{thm:Interlink}
Suppose $\Lambda_0, \Lambda_1 \subset (Y, \alpha)$ are a generic pair of Legendrians with generic contact form $\alpha$ for which $|\gamma| \ge 2$ is satisfied for all contractible  periodic Reeb orbits $\gamma$ (c.f. Remark \ref{rem:ClosedOrbits}). Assume $\Lambda_0, \Lambda_1$ have augmentations. 

Assume that the Rabinowitz Floer complex $RFC^{[0,+\infty)}(\Lambda_0,\Lambda_1)=LCH(\Lambda_0,\Lambda_1),$ which is well-defined, see Section \ref{sec:Banana}, is not acyclic. (This is automatically the case e.g.~when the chords of positive action cannot be partitioned into pairs $(c,d)$ with $\mathfrak{a}(d)>\mathfrak{a}(c)$ and index difference $|d|-|c|=1$.)
Then $(\Lambda_0, \Lambda_1)$ are interlinked.
\end{thm}

Some hypothesis on the mixed chords is needed. 
Suppose $\Lambda_0,\Lambda_1$ are both in two Darboux charts in $(J^1Q, dz-pdq),$ separated by a large $z$-distance. Isotope $\Lambda_0$ in the $(p,q)$ direction such that the $T^*Q$-projections of $\Lambda_0$ and $\Lambda_1$ no longer intersect. During the isotopy we see (for example, via the $T^*Q$- or  $J^0(Y)$-projection) the mixed chords canceling in pairs  with strips connecting them of expected dimension 1. In this case  neither 
$(\Lambda_0, \Lambda_1)$ nor $(\Lambda_1, \Lambda_0)$ are interlinked.

The standard Legendrian $(n+1)$-sphere inside the standard contact sphere is given by
$$\Lambda_0 := S^{2n+1} \cap \mathfrak{Re} \C^{n+1},$$ 
i.e.~the intersection of the conical Lagrangian in $\C^{n+1}$ which is given by the real-part and the unit sphere. Quotienting by the $\Z_2$-antipodal map $z \mapsto -z \in \C^{n+1},$ $\Lambda_0/\Z_2$ is the standard Legendrian embedding of $\RP^{n}$ into the standard contact $\RP^{2n+1}=S^{2n+1}/\Z_2$.
It is easy to show that $\Lambda_0$ can be displaced from itself in $S^{2n+1}.$
By computing a Rabinowitz Floer theory (Definition \ref{defn:MCC}), we prove this is not true for this real projective plane. 

\begin{thm}
\label{thm:RP}
Consider the standard contact $\RP^{2n+1}$ equipped with a small non-degenerate  perturbation of the standard round $S^1$-invariant contact form as described in Proposition \ref{prop:RCF}. The standard Legendrian $\RP^n=\Lambda_0/\Z_2$, has a Chekanov--Eliashberg algebra that admits an augmentation, and a well-defined and invariant Rabinowitz Floer homology that is equal to
$$RFH(\RP^{n},\RP^{n})=\bigoplus_{i\in\Z} \Z_2[i].$$
In particular, the standard  Legendrian  $\RP^n$ cannot be displaced from itself for these contact forms.
\end{thm}

There are two standard methods to prove the invariance of holomorphic-curve-based theories, such as the one underlying Theorem \ref{thm:Main}.

Given a generic Legendrian isotopy $\{\Lambda_t\}_{t\in[-,+]}$ between two generic Legendrian submanifolds $\Lambda_-$ and $\Lambda_+,$
there is a filtered unital DGA-morphism 
$$\Phi: \mathcal{A}(\Lambda_+) \rightarrow \mathcal{A}(\Lambda_-)$$ 
given by counting certain holomorphic disks in the symplectization $\R \times Y$ with boundary lying on an exact Lagrangian constructed from the trace of the isotopy with $\Lambda_+$ (resp. $\Lambda_-$) at the positive (resp. negative) end.
See \cite{Chantraine, ChantraineColinDR, EkholmHondaKalman} for details on these Lagrangian trace constructions.
The filtration is defined by the length.
$\Phi$ has a DGA-homotopy inverse, and so induces an isomorphism on homology \cite{Ekholm}. This invariance approach is known as the 
{\bf{continuation method}}.

The {\bf{bifurcation method}} studies how the DGA $\mathcal{A}(\Lambda_t)$ varies with $t \in [-,+].$
Generically, there are two events that occur at isolated $t:$ a pair of Reeb chords can appear or disappear ({\bf{birth/deaths}}); an isolated holomorphic curve can appear ({\bf{handle-slide disk}}), and via one-parameter Gromov compactness, can change the moduli spaces which define the differential.
Each birth/death induces on the DGA a stabilization/destabilization combined with possible filtered tame automorphisms, while each handle-slide disk induces a filtered tame automorphism \cite{EES05b, Ch}. So the isotopy overall induces a sequence of (filtered) {\bf{stable-tame isomorphisms}}.

A stable-tame isomorphism is conjecturally stronger than a homotopically-invertible DGA-morphism.
For example, when making the analogous comparison for Morse chain complex invariance, the bifurcation method preserves the stable-Morse number and certain torsion-based invariances \cite{Damian02, Sullivan02} which the continuation method, a priori, does not. However, due to sensitive gluing analysis, for studying Chekanov-Eliashberg invariance the bifurcation method has only been proved if $(Y, \alpha)=(P \times \R_z, dz - \theta)$ \cite{EES05b, Ch}. This (and the proof of Theorem \ref{thm:Main}) motivates Proposition \ref{prop:Bifurcation}.
The statement is somewhat unwieldy because  a general contact manifold may have infinite Reeb chords with actions arbitrarily close, while we can only analyze events involving finite numbers of birth/deaths and handle-slides. However, it can be roughly formulated in the following (imprecise) way: the Chekanov--Eliashberg algebra in a general contact manifold is invariant under stable-tame isomorphism below any fixed action level.

\begin{prop}[Bifurcation analysis for concordance maps]
\label{prop:Bifurcation}
Let $\{\Lambda_t\}_{t\in[-,+]}$ be a generic Legendrian isotopy between two generic Legendrian submanifolds {{$\Lambda_-$}} and $\Lambda_+$ and denote by $\Phi_{[a,b]}$ the unital DGA morphism induced by the Lagrangian trace of the isotopy $\{\Lambda_t\}_{t \in [a,b]},$ with $[a,b] \subset [-,+]$. Recall that this is an exact Lagrangian cobordism diffeomorphic to a cylinder; see e.g.~Appendix \ref{sec:trace} for its construction.

Fix some number $l>0$ and assume that all birth/deaths in the Legendrian isotopy $\Lambda_t$ are generic, so that none occur precisely at action $l$, while there are finitely many that occur at action strictly less than $l$ at distinct times.

For any  sufficiently fine generic subdivision
$$-=t_1<t_2<\ldots<t_N=+$$
for which each $(t_i,t_{i+1})$ contains at most one birth/death below action $l,$
 the following holds. The restricted DGA-morphism $\Phi_{[t_i,t_{i+1}]}|_{\mathcal{A}^l(\Lambda_{t_{i+1}})}$ can be conjugated to the  algebra map that is defined by mapping to zero any generator involved in a death moment and  canonically  identifying all remaining Reeb chord generators,  where the conjugation is by filtered tame automorphisms of the domain and target.
This means, in particular, that $\Phi_{[t_i,t_{i+1}]}|_{\mathcal{A}^l(\Lambda_{t_{i+1}})}$ is a stable-tame isomorphism of DGAs. 
\end{prop}
\begin{rem}
\label{rem:ClosedOrbits2}
Of course the sub-DGA $\mathcal{A}^l$ is itself not invariant; the restriction $\Phi_{[t_i,t_{i+1}]}|_{\mathcal{A}^l(\Lambda_{t_{i+1}})}$ is e.g.~not necessarily contained inside $\mathcal{A}^l(\Lambda_{t_{i}})$.
\end{rem}
\begin{rem}
Since we do not discuss the virtual perturbation schemes for defining the contact homology algebra for periodic Reeb orbits, we again need to assume that $|\gamma| \ge 2$ holds for any contractible Reeb orbit of length at most $l$ in the above proposition; c.f.~Remark \ref{rem:ClosedOrbits}.
\end{rem}

In \cite{DRS2}, we exploited the bifurcation-type invariance in \cite{EES05b,EES05c} to show that the barcode from persistent homology induced by the action-filtration is continuous with respect to the oscillatory norm in the case of contactizations. The invariance by DG-homotopies from \cite{Ekholm} that holds in general contact manifold also satisfies filtration-preserving properties, but these are a priori not continuously depending on the oscillation alone; see the notion of length introduced in \cite{SabloffTraynor15}. Proposition \ref{prop:Bifurcation} allows us to reprove the results from \cite{DRS2} in the general setting.

In Section \ref{sec:Algebra} we review some algebra and combinatorics of DGAs, mapping cones, and barcodes.
In Section \ref{sec:Geometry}, we prove Proposition \ref{prop:Bifurcation}.
In Section \ref{sec:Rabinowitz}, we introduce a Rabinowitz Floer complex generated by 
Reeb chords between two Legendrians, and study how a certain mapping cone 
of this complex changes when one of the Legendrians isotopes (possibly through the other one). This version of the complex was previously defined by Legout \cite{Legout} in the case of a contactization, and is also related to the Floer homology for Lagrangian cobordisms defined in \cite{Dimitroglou:Cthulhu}. In Section \ref{sec:Proofs}, we use the changing barcodes of these mapping cone complexes to prove Theorem \ref{thm:Main}, Theorem \ref{thm:NonDegen}, and Theorem \ref{thm:Interlink}.
In Section \ref{sec:RP}, we compute the example of Theorem \ref{thm:RP}; see Proposition \ref{prop:RCF}.

\begin{rem}
\label{rem:Oh}
We learned that Oh has posted before us by a couple of weeks a related result \cite{Oh21}: if the Hamiltonian oscillation is less than the length of the shortest Reeb chord between, then the number of Reeb chords is bounded below by the sum of the betti numbers.
This improves one of our earlier results in which we require the upper bound on the oscillation to be less than the length of the shortest Reeb chord, multiplied by the conformal factor of the contact form \cite{DRS1}.
Whereas our prior results (\cite{DRS1, DRS2}) and this paper use versions of Floer theory which have already been established by others, Oh's approach is different in that he establishes the analytical framework for a new theory called Hamiltonian perturbed contact-instantons.
 With this new theory established, his result follows from Chekanov's original argument for Lagrangians \cite{ChekanovDisplacement}. 
\end{rem}

\section{Algebraic preliminaries}
\label{sec:Algebra}
This section is purely algebraic, with no mention of the geometric applications.

\subsection{Filtered DGAs, augmentations and stable tame isomorphisms}
\label{sec:STI}

Let $(\mathcal{A},\partial)$ be either a noncommutative or a graded-commutative
 semifree
  unital DGA over the ground field $\kk.$ 
 Assume $(\mathcal{A},\partial)$ has an action-filtration $\ell \colon \mathcal{A} \to \{-\infty\} \cup \R$, where $\ell(\partial(x)) < \ell(x)$; we say that $\partial$ is (strictly) filtration-decreasing, or action-decreasing. See \cite{DRS2} for more details.  
Suppose $\mathcal{A}$ admits a $\Z_2$-graded augmentation $\varepsilon \colon \mathcal{A} \to \kk$. 
There is an induced  filtration-preserving unital  algebra-automorphism
$ \Psi_{\varepsilon} \colon \mathcal{A} \to \mathcal{A} $
defined by
$$\Psi_{\varepsilon}(a)=a-\varepsilon(a)$$
on the generators whose inverse is defined by
$$\Psi_{\varepsilon}^{-1}(a)=a+\varepsilon(a).$$
In particular, the inverse is also  filtration-preserving,  and these automorphisms conjugate the differential to
$$\partial^{\varepsilon} := \Psi_{\varepsilon} \circ \partial \circ \Psi_{\varepsilon}^{-1}$$
which preserves both word length and action.
In particular, if we write $\partial^{\varepsilon}(a) =  \sum_{i=0} (\partial^{\varepsilon}(a))_i$
as a sum of monomials, then $(\partial^{\varepsilon})_1$  defines  a strictly filtration-decreasing  differential on the graded vector space generated by the DGA generators which is of degree-$(-1)$.

%

A filtered tame automorphism $\Phi: (\mathcal{A}, \partial) \rightarrow (\mathcal{A}, \partial')$ 
 of a semifree DGA $\mathcal{A}$ with preferred basis 
is defined on the generators by
$$\Phi(x) = k x+ \delta_x^y w$$ where 
$k \in \kk$ is a unit, $\delta^y_x$ is the Kronecker-delta, and $w \in \mathcal{A}$ is a word such that $\ell(w) < \ell(y).$

A canonical identification between semifree filtered DGAs with preferred bases is a DGA isomorphism induced by an identification of the generators which preserves the grading and differential, but not necessarily the filtration.

The stabilization {$(\mathcal{B},\partial_{\mathcal{B}})$ of $(\mathcal{A},\partial_{\mathcal{A}})$ is constructed by adding to $\mathcal{A}$ two generators $x,y$ such that $\partial_{\mathcal{B}}(y) = x$ and $\partial_{\mathcal{B}}|_{\mathcal{A}} = \partial_{\mathcal{A}}.$
Note that there is a canonical DGA inclusion $\mathcal{A} \rightarrow \mathcal{B}$ and DGA quotient $\mathcal{B} \rightarrow \mathcal{A}.$

A (filtered) {\bf{stable-tame isomorphism}} (STI) $\Phi \colon (\mathcal{A},\partial_{\mathcal{A}}) \to (\mathcal{B},\partial_{\mathcal{B}})$
is a finite composition of (filtered) tame automorphisms,  (possibly non-filtered) canonical identifications,
stabilizations, and inverse stabilizations.
%
%

Let $\Phi \colon (\mathcal{A},\partial_{\mathcal{A}}) \to (\mathcal{B}, \partial_{\mathcal{B}})$ be a DGA-morphism.
For any augmentation $\varepsilon$ of $\mathcal{B}$ we can define the conjugation
$$ \Phi^\varepsilon :=  \Psi_{\varepsilon} \circ \Phi \circ \Psi_{\varepsilon  \circ \Phi}^{-1} \colon (\mathcal{A},\partial_{\mathcal{A}}^{\varepsilon \circ \Phi}) \to (\mathcal{B},\partial_{\mathcal{B}}^\varepsilon)$$
which satisfies 
$$(\Phi^\varepsilon)_1 (\partial^{\varepsilon \circ \Phi  }_{\mathcal{A}})_1
 = (\partial^{\varepsilon}_{\mathcal{B}})_1(\Phi^\varepsilon)_1 .
$$
If $\Phi$ is a (filtered) STI, then $(\Phi^\varepsilon)_1$ is a finite composition of (filtered) handle-slides, and birth/deaths at the chain level (see Section \ref{ssec:Barcodes}).

\subsection{Mapping cones with filtrations and invariance}

\label{sec:AlgebraInvariance}

Let $(C_{01},d_{01})$ and $(C_{10},d_{10})$ be filtered graded complexes with action filtrations $\ell \colon C \to \R \cup \{-\infty\}$, and strictly 
{{filtration-decreasing}}
 differentials. (The grading subscript is suppressed while the ``01" and ``10" subscripts will be justified in Section \ref{sec:Rabinowitz}.) Moreover, assume that the generators of $C_{01}$ all have actions above some fixed $\gamma \in \R$ while the generators of $C_{10}$ all have actions below $\gamma.$ We write
$$C^{< a} \coloneqq \ell^{-1}(-\infty,a)  \subset C$$
for the sub-complex consisting of chains of action less than $a \in \R$, and denote the quotient complex consisting of chains in the action window $[a,b)$ by $C^{[a,b)} \coloneqq C^{< b}/C^{<a}$.

Let $B \colon C_{01} \to C_{10}$ be a chain map,  which automatically is strictly filtration-decreasing by the above assumptions. For this reason,  we get an induced action filtration on the chain complex given by the mapping cone $(\OP{Cone}(B),d_B)$ which we represent by
$$\left(C_{10} \oplus C_{01},\partial_{\OP{Cone}}=\begin{pmatrix}-\partial_{10} & B \\ 0 & \partial_{01}\end{pmatrix}\right),$$
i.e.~$\partial_{\OP{Cone}}$ is again strictly filtration-decreasing.

Suppose $B' \colon C'_{01} \to C'_{10}$ is another  chain map between filtered complexes. Again we assume that all generators of the domain have action greater than the action of the generators in the target, which means that $B'$ is automatically strictly filtration-decreasing.  A map $f: C \rightarrow C'$ between filtered chain complexes with filtrations $\ell, \ell'$ is said to have {\bf{degree}} $\epsilon \in \R$ if $\ell'(f(x)) \le \ell(x) + \epsilon$ for all $x \in C.$
Assume the following.
\begin{enumerate}[label={(A\arabic*)}, ref=(A\arabic*)]
\item \label{A1} There exist chain maps $\phi_{01} \colon C_{01} \to C'_{01}$ and $\psi_{10} \colon C'_{10} \to C_{10}$ with homotopy inverses $\psi_{01}$ and $\phi_{10}$ with chain homotopies $h_{ij} \colon \psi_{ij}\circ \phi_{ij} \sim \id_{C_{ij}}$ and $k_{ij} \colon \phi_{ij}\circ \psi_{ij} \sim \id_{C'_{ij}}$, where all maps above are of degree $\epsilon$.
\item \label{A2} The map $B$ is chain homotopic to $\psi_{10} B' \phi_{01}$ via a chain homotopy $H \colon B \sim \psi_{10} B' \phi_{01}.$ Note that this homotopy automatically has negative degree, i.e.~it is strictly filtration-decreasing.
\end{enumerate}
{(Note in \ref{A1} that the homotopies have degree $\epsilon$ instead of $2\epsilon$ as is common in Hamiltonian Floer theory literature, for example \cite[Definition 8.1]{UsherZhang16}.)}
Then we have a homotopy commutative square
$$\xymatrix{
C_{01} \ar[r]^{B} \ar[d]_{\phi_{01}}  \ar@{-->}[dr]_{h} & C_{10} \ar[d]^{\phi_{10}} \\
C'_{01} \ar[r]_{B'} &  C'_{10} 
}
$$
where $h=\phi_{10}H+k_{10} B'\phi_{01}$ is the induced chain homotopy between $\phi_{10} B$ and $B'\phi_{01}$, which thus is a map of degree $\epsilon$. Similarly, there is a chain homotopy $h' \colon B \psi_{01} \sim  \psi_{10}B'$ given by $h'=H\psi_{01}+\psi_{01}B'k_{01}$ that makes the square in the following diagram
$$\xymatrix{
C'_{01} \ar[r]^{B'} \ar[d]_{\psi_{01}}  \ar@{-->}[dr]_{h'} & C'_{10} \ar[d]^{\psi_{10}} \\
C_{01} \ar[r]_{B} &  C_{10} 
}
$$
commute up to homotopy. Both $h$ and $h'$ are maps of degree 
{{$\epsilon$.}}

It follows that there are induced chain maps of the cones
$$\xymatrix{
 C_{10} \ar[r] \ar[d] & \OP{Cone}(B) \ar@{-->}[d]_f \ar[r] & C_{01} \ar[d]  \\
 C'_{10} \ar[r] \ar[d] & \OP{Cone}(B') \ar@{-->}[d]_g \ar[r] & C'_{01} \ar[d] \\
  C_{10} \ar[r]  & \OP{Cone}(B) \ar[r] & C_{01}   \\
}
$$
given by $f=\begin{pmatrix}\phi_{10} & h \\ 0 & \phi_{01}\end{pmatrix}$ and
$g=\begin{pmatrix}\psi_{10} & h' \\ 0 & \psi_{01}\end{pmatrix}$
of degree $\epsilon$. By the five-lemma, the induced homology maps are isomorphisms. 

\begin{lem}
  \label{lem:smalldeg}
  The maps $f \circ g$ and $g\circ f$ are each chain homotopic to  automorphisms of filtered chain complexes  (i.e.~degree zero chain maps with degree zero inverses), via chain homotopies that are of degree $\epsilon$.
\end{lem}
\begin{proof}
  We show the statement for $g\circ f$; the argument for $f \circ g$ is completely analogous.

There is a chain homotopy $\begin{pmatrix}h_{10}& 0 \\ 0 & h_{01} \end{pmatrix}$ of degree $\epsilon$ from $f\circ g$ to a chain automorphism of the form
  $$ \begin{pmatrix}\id_{C_{10}}& K \\ 0 & \id_{C_{01}} \end{pmatrix}.$$
Since the entry $K \colon C_{01} \to C_{10}$ is of non-positive degree, the matrix is of degree zero. The matrix is a chain map since it is chain homotopic to $f \circ g.$ This chain map-property translates to the fact that {{$K$}} is a chain map (that performs a grading shift by $+1$). Note that the inverse map is
  $$ \begin{pmatrix}\id_{C_{10}}& -K \\ 0 & \id_{C_{01}} \end{pmatrix}.$$
  which hence also is of degree zero.

  \end{proof}

\subsection{Simple equivalences from small degree homotopy equivalences}
\label{sec:homotopyinvariance}

Next we relate the homotopy equivalences of small degree as above with the invariance of bifurcation-type that our previous work \cite{DRS2} was based on.

Recall that a basis $\{\mathbf{e}_i\}$ of the filtered complex $C$ is {\bf compatible} with the filtration if there is an action function $\ell(\mathbf{e}_i) \in \R$ defined on the basis, so that $c \in C^{<a}$ if and only if $c$ can be written as a sum of basis elements of action less than $a$; see \cite[Section 2.1]{DRS2}). We say that a complex is {\bf $\delta$-gapped} if two different basis elements in a compatible basis have action values that  either coincide, or differ by at least $\delta>0$.

  \begin{lem}
    \label{lem:se}
    Consider two filtered complexes $C$ and $C'$, where $C$ is $\delta$-gapped and satisfies the property that each action level has at most finitely many generators, and for which the following is satisfied:
    \begin{enumerate}
  \item $C$ and $C'$ admit bases compatible with the filtration whose elements are in a bijection $x\mapsto x'$ under which the action satisfies $\ell(x)-\ell'(x') \le \epsilon$; and
  \item There are chain maps $\phi \colon C \to C'$ and {{$\psi \colon C' \to C$}} which both are of degree $\epsilon>0$, where $\psi\phi$ and $\phi\psi$ are homotopic to  automorphisms of filtered chain complexes  via chain homotopies of degree $\epsilon$.
\end{enumerate}
If $\delta>4\epsilon>0$, then $\phi$ is a isomorphism with inverse $\psi$.

If we endow the complexes with bases that are compatible with the filtrations, ordered in decreasing action, with the additional assumption that two different elements have distinct action values, then $\phi$ is upper-triangular with units on the diagonal.
\end{lem}
\begin{proof}
By the assumptions, we can write
  $$\psi\phi=\Phi+\partial K+K\partial$$
  where $\Phi$ is  an automorphism of filtered chain complexes.  By the assumption that $C$ is $\delta$-gapped, and since $K$ is of degree degree $\epsilon<\delta/4$, we conclude that $K$ is in fact filtration preserving. Hence $\partial K +K\partial$ is strictly filtration-decreasing. It follows by a standard fact that $\psi\phi$ itself is an automorphism of filtered chain complexes.

  The map $\phi$ is injective by the above. 
 It thus suffices to show that $\phi$ is surjective since, in that case, $\psi=\phi^{-1}$.

  Note that, for any basis element $x \in C$ in a basis compatible with the
 filtration,  the map induced by quotient and restriction 
$$\phi \colon C^{[\ell(x)-3\epsilon,\ell(x)+3\epsilon)} \to C'^{[\ell(x)-3\epsilon,\ell(x)+3\epsilon)}$$
 is a map between equidimensional vector-spaces 
  by the assumption on the action spectrum of the involved complexes. Since $\psi\phi$ is an automorphism of filtered complexes, together with the assumption that $C$ is $\delta$-gapped, we deduce that the above map on the quotient is injective as well, and thus surjective. Consequently the map $\phi$ itself is surjective, as sought.
  \end{proof}

  \begin{lem}[Characterization of a birth/death]
        \label{lem:bd}
Consider two filtered complexes $C$ and $C'$, and assume that $C^{[a+\delta,a+3\delta)}=C^{[a,a+4\delta)}$ are both two-dimensional, while $C'^{[a,a+4\delta)}=0$. Assume that there exist chain maps $\phi \colon C \to C'$ and $\psi \colon C' \to C$  which both are of degree $\epsilon>0$, where $\psi\phi$ and $\phi\psi$ are homotopic to  automorphisms of filtered chains complexes  via chain homotopies of degree $\epsilon$. If $\delta>\epsilon>0$, then $C^{[a+\delta,a+3\delta)}$ 
is a complex generated by $x,y$ with $\partial x = k y$ for some unit $k.$
\end{lem}
\begin{proof}
Since the map $\phi \colon C^{[a,a+4\delta)} \to C'^{[a,a+4\delta)}=0$ induced by quotient and surjection  is a homotopy equivalence, 
$C^{[a,a+4\delta)}$ is acyclic. The only two-dimensional acyclic complexes are the ones described above.
\end{proof}

\subsection{Barcodes}
\label{ssec:Barcodes}

We sketch without details the modified barcode theory done in \cite[Section 2]{DRS2}.

For $t \in \R,$ $C(t)^{b_t}_{a_t}$ be a filtered complex with filtration action $\ell$ taking values in 
$[a_t, b_t) \subset \R$ and $-\infty.$
A {\bf{piecewise continuous}} (PWC) {\bf{family}} of such filtered complexes, parameterized by $t,$ is characterized by the following properties.
\begin{itemize}
\item The endpoints of the action window $[a_t,b_t)$ vary continuously with $t.$
\item There exists a discrete set of $t_1 < \cdots < t_N$ such that during any component $I \subset \R \setminus \{t_1, \ldots, t_N\}$ there are canonically identified (for different $t \in I$) generators of the complexes which are compatible with the action.
\item The action of each such generator is continuous and almost everywhere differentiable with respect to $t \in I.$
\item For each $t \in I,$ the differential strictly decreases the action.
\item For each $t \in \{t_1, \ldots, t_N\}$ the chain complex undergoes at most one of the following possible ``simple bifurcations" 
\begin{itemize}
\item The algebraic equivalent of a Morse handle-slide
\item The algebraic equivalent of a Morse birth/death
\item An entrance (resp. exit) of one generator into (resp. from) at either the top or bottom  of the action window.
\end{itemize}
\end{itemize}

We continue to use \cite[Section 2]{DRS2} in defining barcodes, but for the equivalent definition based on normal forms, see \cite[Section 2.1]{PolterovichRosenSamvelyanZhang} and \cite[Definition 6.2]{UsherZhang16}.
A barcode is a finite collection of ``bars" $[s,e)$ where the endpoint $e$ might be $+\infty.$
Let $\phi_{c_0,c_1} \colon H(C(t)^{c_0}_{a_t}) \to H(C(t)^{c_1}_{a_t})$ be induced from the inclusion $C(t)^{c_0}_{a_t} \hookrightarrow C_*(t)^{c_1}_{a_t}$ where $a_t \le c_0 \le c_1 \le b_t.$
The {\bf barcode of the complex} $(C(t)^{b_t}_{a_t},\partial_t)$ is the barcode uniquely characterized by the following properties.
\begin{itemize}
\item The number of bars with starting point $s$ is equal to the dimension of the quotient
$$ \OP{coker}( \phi_{s,s+\epsilon}) = H(C(t)^{s+\epsilon}_{a_t},\partial_t)/\im\,\phi_{s,s+\epsilon}$$
where $\epsilon>0$ is any sufficiently small number.
\item The number of bars with starting point $s$ that persist at action level $l\ge s$ is equal to the dimension of the subspace
$$ [\phi_{s+\epsilon,l+\epsilon}](\OP{coker}( \phi_{s,s+\epsilon})) \subset H(C(t)^{l+\epsilon}_{a_t},\partial_t)/\im\,\phi_{s,l+\epsilon}$$
where $\epsilon>0$ is any sufficiently small number and where the map
$$[\phi_{s+\epsilon,l+\epsilon}] \colon \OP{coker}( \phi_{s,s+\epsilon}) \to H(C(t)^{l+\epsilon}_{a_t},\partial_t)/\im\,\phi_{s,l+\epsilon} $$
is induced by descending $\phi_{s+\epsilon,l+\epsilon}$ to the quotients.
\end{itemize}

\begin{prop}[Proposition 2.7 of \cite{DRS2}]
\label{prop:Barcode}
Consider a PWC family of filtered complexes with action window, $C(t)^{b_t}_{a_t}.$ 
When the complex undergoes no such bifurcation, the barcode undergoes a continuous change of action levels for its starting and endpoints.
At the bifurcations the barcode undergoes the following corresponding changes.
\begin{itemize}
\item {\bf Handle-slide: } The barcode is unaffected.
\item {\bf Birth/death: } When two generators $x,y$ undergo a birth/death, then a bar connecting $\ell(x)$ to $\ell(y)$ is added to/removed from the barcode. (The bar is not present at the exact time of the birth/death, but immediately after/before it is visible and of arbitrarily short length.)
\item {\bf Exit below: } A generator slides below the action level $a_t$ at time $t.$ If the uniquely determined 
bar which starts at the action level of that generator is of finite length, then that bar gets replaced with a bar of infinite length whose starting point is located at the same action level as the endpoint of the original bar. If the bar has infinite length, then it simply disappears from the barcode.
\item {\bf Entry below: } This is the same as a exit below but in backwards time.
\item {\bf Exit above: } A generator slides beyond the action level $b_t$ at time $t.$ There is a uniquely determined bar which either ends or starts at the action level of that generator. In the first case, the bar gets replaced with one that has the same starting point but which is of infinite length. In the second case, when the bar necessarily is infinite, then that bar simply disappears from the barcode.
\item {\bf Entry above: } This is the same as an exit above, but in backwards time.
\end{itemize}
\end{prop}

\subsection{A piecewise continuous family of complexes from small-degree homotopy equivalences}

In order to investigate the continuous dependence of the barcodes in relation to the invariance properties established in Section \ref{sec:AlgebraInvariance}, we need to relate invariance under small degree homotopy equivalence as in Lemma \ref{lem:smalldeg} to the bifurcation-type invariance that Proposition \ref{prop:Barcode} above is based upon.

Let $C(t)$, $t \in [0,1]$ be a family of finite-dimensional filtered complexes with choices of compatible bases elements. We assume that all basis elements vary smoothly with $t$ except that there are finitely times $t_1<\ldots<t_N$ at which there is a unique birth/death moment. Roughly speaking,  at these moments precisely two basis elements of the same action either appear or disappear.  The next paragraph gives the precise characterization of a birth/death.

Since we require the differential to be strictly filtration-decreasing, the two basis elements that undergo a birth/death at $t_i$ are necessarily missing from the complex $C(t_i)$.  However, in the case of a birth (resp. death) two generators for $t>t_i$ (resp $t<t_i$) are assumed to  have actions that extend continuously to $t=t_i$, such that the extensions moreover attain the same action value at $t=t_i$. (Note that, in particular, the action of any basis element is bounded in the family, and there is a global bound on the dimension of the complexes $C(t)$ in the family.)
For simplicity we make the additional assumptions that, at each birth/death moments $t=t_i$, all elements in the compatible basis have distinct action values and, moreover, their action values different from the action of the (continuous extension of the) birth/death pair. In addition, we assume that there is a finite set of times when the action values of a compatible basis are not distinct.

 In order to simplify the notation, we will now assume that the finite number of times when the action spectrum is not injective, as well as the birth/death moments, all occur at rational times $t \in \Q \cap [0,1]$.

It is worth to stress that, at this moment, we have not yet made any assumptions on how the differential of the complexes $C(t)$ varies; we are simply prescribing how their compatible bases depend on $t$. Under the further assumptions of the next result, Proposition \ref{prop:PWC}, we establish an invariance result for this family of complexes; this is in fact one of the main goals of this section.

\begin{prop}
\label{prop:PWC}
Let $C(t)$ be a family of complexes as above that satisfies the following additional requirement. 
For all $\epsilon>0$, all sufficiently large $N \gg 0$ and all $i \in \{0,\ldots,N-1\},$ there are chain maps
  \begin{gather*}
    \phi_i \colon C(i/N)  \to C((i+1)/N) \:\:\text{and}\:\:  \psi_i \colon C((i+1)/N)  \to C(i/N)
  \end{gather*}
of degree $\epsilon>0$,  where $\psi_i\phi_i$ and $\phi_i\psi_i$ are homotopic to  automorphisms of filtered chain complexes  via chain homotopies of degree $\epsilon$. Then, for  $N \gg 0$ is sufficiently large, there exists a piecewise continuous family of complexes $D(t)$, that admit isomorphisms $C(t) \cong D(t)$ of filtered vector spaces, that moreover are chain maps for all $t=i/N$, $i =0,1\ldots,N$.
\end{prop}
\begin{proof}

  We start by prescribing
  $$D(i/N) \coloneqq C(i/N)$$
  for all $i =0,1\ldots,N.$
  
  First we consider the special case when $C(t)$ has no birth/deaths, and the action values of the basis elements are all distinct for all times. We then construct the PWC family as follows. The complexes $D(t)$ for $t \in [i/N,(i+1)/N)$ are constructed by setting $D(t) \coloneqq D(i/N)$ as a complex, and then simply letting the action values of a compatible basis vary accordingly with $t \in [i/N,(i+1)/N)$. (I.e.~the differential remains unchanged.) It is immediate that $D(t) \cong C(t)$ holds on the level of \emph{filtered vector spaces.}

  The family of complexes $D(t)$ for $t \in [i/N,(i+1)/N)$ extends by continuity to also $t=(i+1)/N$; denote the limit filtered complex by $\tilde D((i+1)/N)$.  What remains is to construct an isomorphism of filtered complexes $\tilde D((i+1)/N)\cong D((i+1)/N)$.

  The assumptions of Lemma \ref{lem:se} are satisfied for all maps $\phi_i,\psi_i$ whenever $N \gg 0$. In particular, all complexes $D(i/N)$ can be assumed to be $\delta$-gapped for some fixed $\delta>0$. Hence $\phi_i \colon D(i/N) \to D((i+1)/N)$ is a chain isomorphism of degree $\epsilon$. Since $\tilde D((i+1)/N)$ is canonically identified with $D(i/N)$ (only the action values of the compatible basis have changed slightly), it follows that the induced chain isomorphism $\phi_i \colon \tilde D((i+1)/N) \to D((i+1)/N)$ is of degree zero. This implies that we have a PWC family, as sought.

In the case when the family $C(t)$ has birth/deaths or compatible basis elements of the same action value, then we need to take care at those moments separately. This we do in the subsequent paragraph. After having constructed the PWC family in that a small neighborhood of these points in time, the family for the remaining times can be constructed as above. 

In the case when two action values for a compatible basis coincide at some $t=i/N$, then we can again use Lemma \ref{lem:se} as above to construct the family $D(t)$ for $t \in [i/N,(i+1)/N]$ and $[(i-1)/N,i/N]$. In the case when there is a birth/death at $t=i/N$, the same can be done by alluding to Lemma \ref{lem:bd}. Once having taking care of the construction of $D(t)$ for these times, we simply invoke the construction in the first simple case, i.e.~the case when action values are distinct, and when there are no birth/death moments.



 \end{proof}

\section{Proof of Proposition \ref{prop:Bifurcation}}
\label{sec:Geometry}

Consider in the symplectization $(\R_\tau \times Y, d(e^\tau \alpha)),$ the {\bf{Lagrangian trace}} of a Legendrian isotopy $\Lambda_t$ with $t$ ranging from $-$ to $+.$
This (exact embedded) Lagrangian concordance coincides with the cylinder $\R \times \Lambda_\pm$ for $\pm \tau \ge k$ for some $k \gg 0$, and after reparameterizing the Legendrian isotopy, the $\tau$-level set of the trace is close to $\{\tau\} \times \Lambda_\tau$
For example, see \cite[Theorem 1.2]{Chantraine}, the proof of \cite[Lemma A.1]{Ekholm}, the proof  of \cite[Lemma 6.1]{EkholmHondaKalman}, or \cite[Definition 2.10]{PanRutherford20} on constructing this trace.  We recall the version of the construction from \cite[Theorem 1.2]{Chantraine} by Chantraine in Appendix \ref{sec:trace}.

The {\bf length} $\delta \in \R_{\ge 0}$ of the cobordism, as defined by Sabloff--Traynor in \cite{SabloffTraynor15}, is the shortest $\delta$ such that $\{\tau \in [\tau_0,\tau_0+\delta]\} \subset \R \times Y$ contains the non-cylindrical portions of the cobordism and almost complex structure. Proposition \ref{prop:trace} in the Appendix shows how the length of the trace cobordism constructed by \cite{Chantraine} depends on the conformal factor of the contact isotopy; we also consider the length of the ``inverse cobordism''. The length is crucial for analyzing the filtered invariance result, since the chain maps produced by the trace cobordism have filtration properties that depend on this.

We denote by $x^t$ a continuous family of chords of $\Lambda_t.$ Recall that the family of Legendrians is assumed to be generic, which means that a chord $x^\pm$ of $\Lambda_\pm$ that does not undergo a death-type bifurcation in the family corresponds to a unique chord $x^\mp$ of $\Lambda_\mp$.

The Lagrangian together with an appropriately compatible almost complex structure $J$, induce a DGA-morphism $$\Phi: (\mathcal{A}_+, \partial_+) \rightarrow (\mathcal{A}_-, \partial_-)$$ between the DGAs $(\mathcal{A}_\pm, \partial_\pm)$ of $\Lambda_\pm$ \cite{Ekholm}. More precisely, we require as in \cite{Ekholm} that the almost complex structure is {\bf{adjusted}}, which means that it is compatible with the symplectic form and, {\bf{cylindrical}} outside of a compact subset; the latter means that the almost complex structure preserves the contact planes lifted to $\R_\tau \times Y,$ is invariant under the $\R_\tau$-translation, and sends $\partial_\tau$ to the lifted Reeb flow.
 
Suppose the Lagrangian trace of an isotopy $\Lambda_t$  induces the map $\Phi$ and the inverse trace induces the map $\Psi.$ 
Construct a generic 1-family of Lagrangians connecting the trivial cobordism $\Lambda_+\times \R$ with induced DGA map $\id: (\mathcal{A}_+, \partial_+) \rightarrow (\mathcal{A}_+, \partial_+),$ to the trace concatenated with its inverse.
Let $G$ count index $-1$ punctured pseudoholomorphic curves that occur at isolated moments in this family of cobordism, as defined in \cite{Ekholm}. Then 
\begin{eqnarray}
\label{eq:homotopy1}
\id & = & \Psi\Phi - (\partial_+ G - G \partial_+).
\end{eqnarray}
The below  result follows from the  filtration-preserving  properties of the involved DGA-morphisms and chain homotopy.



\begin{lem}
\label{lem:ActionBound}
\begin{enumerate}
\item
  Consider the Legendrian isotopy $\Lambda_t$ generated by a time-dependent contact Hamiltonian $H_t$. For any $\delta>0$ the Legendrian isotopies
  $$\{\Lambda_t\}_{t \in [i/N,(i+1)/N]} \quad \text{and} \quad \{\Lambda_{-t}\}_{t\in[-(i+1)/N,-i/N]}$$
  may all be assumed to have a trace cobordism of length less than $\delta$ whenever $N \gg 0$.  In addition, both concatenations of these two trace cobordisms may be assumed to be compactly supported Hamiltonian isotopic to the trivial cylinder through cobordisms of length at most $2\delta$.

\item Let $\Phi$ and $\Psi$ above be induced by Lagrangian cobordisms of length $\delta>0$, where the concatenation of the cobordisms are Hamiltonian isotopic to trivial cobordism through cobordisms of length at most $2\delta$. (E.g.~the trace cobordisms from Part (1).)  Then
  $$ \ell(F(x^+)) < \ell(x^+) e^{2\delta}$$ holds for $F \in \{\Psi, \Phi\}$,
  while any word $\mathbf{x}^+$ that consists of letters that satisfy $\ell(x^+) \le a$ has the property that $G(\mathbf{x}^+)$ consists of words of letters that all satisfy $\ell(x^-) \le e^{2\delta} a$. (Notation as in Equation (\ref{eq:homotopy1}).) 

%
\end{enumerate}
\end{lem}

\begin{proof}
 For item (1), we start by fixing a global contact isotopy that generates the Legendrian isotopy $\Lambda_t$. For $N \gg 0$, we may assume by continuity that the restriction of the contact isotopy that generates $\{\Lambda_t\}_{t\in [i/N,(i+1)/N]}$ has a conformal factor that is bounded by $\delta>0$. The construction of the trace cobordisms with the sought properties can now be deduced from Proposition \ref{prop:trace}.

  

For item (2): The statement is clear for any map $F$ that is defined by a count of finite energy pseudoholomorphic disks in $\R_\tau \times Y$ with boundary on a Lagrangian cobordism of length at most $2\delta$ and a single positive puncture, when the almost complex structure is adjusted and, moreover, cylindrical where the cobordism is cylindrical (e.g.~outside of $[-\delta,\delta]\times Y$). Namely, \cite[Lemma 3.4 and Proposition 3.5 (9)]{Dimitroglou:Cthulhu} explicitly bound the Hofer energy of such curves used to define $F(x^+)$ from below by 0 and from above by the quantity $\ell(x^+) e^{\delta} -\ell(F(x^+))e^{-\delta}.$ To match with their convention, we ignore their distinction of pure and mixed chords, and we center the concordance around $\tau_0 = 0$ (from our definition of length above).

Since the chain maps $F=\Phi,\Psi$ are defined by pseudoholomorphic disks of the type mentioned above, the statement now directly follows in these cases.

The chain homotopy $G$ has a more complicated construction, which was carried out in \cite[Appendix B]{Ekholm}. Each term in $G(x^+_1\cdot\ldots\cdot x^+_k)$ corresponds to a count of disconnected pseudoholomorphic \emph{buildings} (see \cite{BEHWZ}), where each component of the building has the topological type of a broken disk with a single positive puncture at $x^+_i$ for $i=1,\ldots,k$. In addition, each component satisfies the following:
\begin{itemize}
\item There is a single level consisting of a number of pseudoholomorphic disks of index $-1$ and $0$ inside $\R \times Y$, each with boundary on one of the Lagrangian cobordisms in the family that interpolates between the concatination and the trivial cylinder (these are all of length at most $2\delta$), and which all have a single positive puncture. As in the first case, we can again assume that the almost complex structure is cylindrical in the subset where the family of cobordisms are cylindrical (e.g.~outside of $[-\delta,\delta]\times Y$).
\item All other levels consist of punctured disks of index $1$ and trivial strips of index 0 inside $\R \times Y$, with boundary on either $\R \times \Lambda_{i/N}$ or $\R \times \Lambda_{(i+1)/N}$ and a single positive puncture, which are pseudoholomorphic for a cylindrical almost complex structure.
\end{itemize}
If one consider these terms as a composition of operations, the fact that the disks of index 1 in the second bullet point define the differential (which is strictly filtration-decreasing), the statement finally follows by an energy estimate similarly to the first case.

%


\end{proof}

%
%
%
%

The remainder of the proof of Proposition \ref{prop:Bifurcation} is similar to the proof of Proposition \ref{prop:PWC}; roughly, we decompose the isotopy into small steps that then are shown to induce homotopy equivalences of small degree.  Lemma \ref{lem:ActionBound}.1 implies that, for a sufficiently fine decomposition of $[-,+]$, each map $\Phi_{[-_i,+_i]}$ has an arbitrarily small cobordism length. To prove Proposition \ref{prop:Bifurcation}, we thus restrict to a single interval of a sufficiently fine generic decomposition (this single small interval we continue to label $[-,+]$) and show that it is  a finite composition of (de)stabilizations and tame automorphisms.

Note that the Reeb chord lengths vary continuously with the parameter $t$. For a very fine decomposition we may thus assume that these lengths are almost constant in the interval $[-,+]$. Together with the Proposition's hypothesis of genericity, and since $\delta>0$ can be assumed to be arbitrarily small, we get three cases listed below. For all cases, we assume that no chord has action less than $e^{2\delta}.$ Recall that we only consider the Chekanov--Eliashberg algebra $\mathcal{A}^l$ generated by chords with actions less than $l.$ Below we thus disregard all chords of action greater than $l'$ for some suitable action level $l' \gg l.$  Moreover, after further shrinking $\delta>0$, we can assume that no Reeb chords on the family $\Lambda_t$ of Legendrians has length contained in the interval $[e^{-\delta}l',e^{\delta}l']$.

\begin{enumerate}
\item[Case (1):] There are no births/deaths in $[-,+]$. Further, any $x^+$ satisfies
  $$  \ell(x^\mp) \in  [e^{-\delta}\ell(x^\pm), e^{\delta}\ell(x^\pm)],$$
  while any $y^+$ different from $x^+$ satisfies
  $$[e^{-\delta}\ell(y^-),e^{\delta}\ell(y^-)] \cap [e^{-\delta}\ell(x^+),e^{\delta}\ell(x^+)]=\emptyset.$$
  (In particular, any two Reeb chords have distinct lengths.)
\item[Case (2):]  The chords whose lengths are contained in $[e^{-2\delta}\ell_0,e^{2\delta}\ell_0]$ are precisely two, and undergo a birth/death at $0 \in [-,+]$; i.e.~there are precisely two chords $x^+,y^+$ of lengths
  $$\ell(x^t),\ell(y^t) \in [e^{-\delta}\ell_0,e^{\delta}\ell_0]$$
for $t > 0$ (resp. $t <0$) while there are no such chords for $t<0$ (resp. $t >0$). Furthermore, $\ell(x^+) > \ell(y^+)$ (resp. $\ell(x^-) > \ell(y^-)$). The chords $z^t$ of length
  $$\ell(z^t) \notin [e^{-2\delta}\ell_0,e^{2\delta}\ell_0]$$
  satisfy  the assumptions of Case (1).

\item[Case (3):] There are no births/deaths in $[-,+]$.  The chords whose lengths are contained in $[e^{-2\delta}\ell_0,e^{2\delta}\ell_0]$ are precisely two, which moreover have lengths contained inside $[e^{-\delta}\ell_0,e^{\delta}\ell_0]$, satisfy $\pm\ell(x^\pm) > \pm\ell(y^\pm)$ while $\ell(x^0) =\ell(y^0).$ The chords $z^t$ of length
  $$\ell(z^t) \notin [e^{-2\delta}\ell_0,e^{2\delta}\ell_0]$$
  satisfy  the assumptions of Case (1).

\end{enumerate}

%

\smallskip

In the case when there are no births/deaths, the invariance under DG-homotopy together with Lemma \ref{lem:ActionBound}.2 now implies that
\begin{eqnarray}
  \label{eq:homotopy}
  x^+ & = & (\Psi\Phi - (\partial_+ G - G \partial_+))x^+ = \\
\notag &=&  k_\Psi k_\Phi x^+ + \sum_j  v^+_j - (\partial_+ G - G \partial_+)x^+ 
=  k_\Psi k_\Phi x^+ + \sum_j  v^+_j + \sum_k  u^+_k
\end{eqnarray}
is satisfied where $k_\Phi \in \kk$ (resp. $k_\Psi \in \kk$) are the coefficients $\langle \Phi(x^+),x^-\rangle$ and $\langle \Psi(x^-),x^+\rangle$, and $v^+_j, u^+_k$ are monomials of chords of $\Lambda_+$ that satisfy
$$e^{2\delta} \ell(x^+) \ge \ell(v^+_j), \ell(u^+_k).$$
%

\smallskip

\noindent
Case(1):

 Looking at the last two terms in equation (\ref{eq:homotopy}), Lemma \ref{lem:ActionBound}.2 implies two things.
First, if $v_j^+ \in (\kk\setminus\{0\}) x^+$ then $v_j^+$ is in the image of  $\Psi$ of an element of action at most $e^{\delta}\ell(x^+),$ which by definition is not contained in $\kk x^-$.
Second, if $u_k^+ \in (\kk\setminus\{0\}) x^+$, then either $x^+$ appears as a term in the differential of a word $G(x^+)$ whose all letters all are of action at most $e^{2\delta}\ell(x^+)$, or as a term in $G(w^+)$ for a word $w^+$ of action $\ell(w^+) < \ell(x^+)$. Moreover, in the latter case, all letters in $G(w^+)$, and thus in particular $x^+$ itself, have action bounded by $e^{2\delta}\ell(w^+)$.
The hypotheses in Case (1) imply that $\ell(v^+_j), \ell(u^+_k) <  \ell(x^+).$ Thus $k_\Phi= k_\Psi^{-1}$ in (\ref{eq:homotopy}) is a unit. 

\smallskip

\noindent
Cases (2) and (3):

 Case (1) handles all chords without other chords of approximately the same action. So we have reduced to studying the maps at $x$ and $y$ which have approximately the same action.
Consider the $2 \times 2$ matrix 
$ [F]:=  
\begin{pmatrix} F_{xx}  & F_{xy} \\ F_{yx} & F_{yy} \end{pmatrix}$
in the $x,y$ basis, for the map $F \in \{\Phi, \Psi, G, \partial_+, \id\}.$
The bound from below of $\ell(z)$ implies there is no additional (non-linear) term involving $x$ or $y$ in either $F(x)$ or $F(y).$ 
Since $\ell(x^+) > \ell(y^+),$
$[ \partial_+  ]  =\begin{pmatrix} 0  & 0 \\  (\partial_+)_{yx} & 0 \end{pmatrix}$ for some $(\partial_+)_{yx} \in \kk.$
 We also consider the corresponding $2 \times 2$ matrix version of (\ref{eq:homotopy}).

\smallskip

\noindent
Case (2):

Since $(x^-,y^-)$ do not exist, $$
[\Psi  ] 
=\begin{pmatrix} 0  & 0 \\ 0 & 0 \end{pmatrix} = 
[\Phi  ] .$$
So (\ref{eq:homotopy}) implies $1 = \id_{xx} =  G_{xy}(\partial_+)_{yx}  .$
In particular, after the tame automorphism of scaling $y$ by $\left((\partial_+)_{yx}\right)^{-1} =  G_{xy} \in \kk,$
$\partial_+(x^+) = y^+ + \sum_i w_i$ with $\ell(w_i) < \ell(y^+).$
The result now follows from Chekanov's algebraic treatment of birth-deaths 
\cite[Sections 8.4-8.5]{Ch}.

\smallskip

\noindent
Case (3):

Assume at  $t=0$ that $J$ is generic so that the DGA differential $\partial_0$ is well defined.
Let $\Phi^0: (\mathcal{A}, \partial_+) \rightarrow  (\mathcal{A}, \partial_0)$ and
$\Psi^0: (\mathcal{A}, \partial_0) \rightarrow  (\mathcal{A}, \partial_+)$ be the DGA morphisms
induced by the trace and its inverse over the subinterval $[0,+] \subset [-,+].$
Let $G^0$ be the homotopy relating $\Psi^0 \Phi^0$ and $\id_0.$
As above, define the $2 \times 2$ matrix 
$\begin{pmatrix} F_{xx}  & F_{xy} \\ F_{yx} & F_{yy} \end{pmatrix}$
in the $x^0,y^0$ basis, for the map $F \in \{\Phi^0, \Psi^0, G^0, \partial_0, \id_0\}.$
Stokes Theorem implies that $[ \partial_0 ] =0$ as a $2 \times 2$ matrix.
From the $2 \times 2$ matrix equations
$$
[ \id_0   ] 
= 
[ G^0 \partial_0  ] 
+ 
[ \partial_0 G^0 ] 
+ 
[ \Psi^0 \Phi^0 ]  = [\Psi^0 \Phi^0  ] ,
\quad 
[ \partial_+ \Phi^0 ]  = 
[ \Phi^0 \partial_0 ]  = 0,
\quad
[ \Psi^0 \partial_+ ]  = 
[ \partial_0 \Psi^0 ]  = 0,$$
we get 
$$
1 = (\id_0)_{xx} = \Psi^0_{xx} \Phi^0_{xx} + \Psi^0_{xy} \Phi^0_{yx},
\quad
(\partial_+)_{yx} \Phi^0_{xx}=0,
\quad
\Psi^0_{xy} (\partial_+)_{yx}=0,
$$
which imply $(\partial_+)_{yx}=0.$ 

Since $[ \partial_+ ]  =0$ as a $2\times 2$ matrix, (\ref{eq:homotopy}) implies $[\Psi  ] , [\Phi  ]  \in 
{{GL(2,\Z).}}$
\cite[Corollary 2.6 ]{Conrad} proves that 
{{$SL(2,\Z),$}} 
which is an index 2 subgroup of 
{{$GL(2,\Z),$}}
 is generated by the two tame automorphisms
$$
\begin{pmatrix} 1  & 1 \\ 0 & 1 \end{pmatrix}, 
\quad
\begin{pmatrix} 1 & 0 \\ 1 & 1 \end{pmatrix}.
$$
But we also allow the map
$$\begin{pmatrix} -1  & 0 \\ 0 & 1 \end{pmatrix} \in 
{{GL(2,\Z) \setminus SL(2,\Z).}}
$$
Thus $\Psi, \Phi$ are compositions of our allowable tame automorphisms.

\section{A Rabinowitz-Floer theory for Legendrians}
\label{sec:Rabinowitz}

Rabinowitz-Floer homology in the case of a contact type hypersurface was originally defined by Cieliebak, Frauenfelder and Oancea \cite{CieliebakFrauenfelderOancea}. We present a version of the theory in the relative case, RFH, which previously has been considered in the Hamiltonian setting by Merry \cite{Merry} and Cieliebak and Oancea \cite{CieliebakOancea}, and in the SFT setting by Legout \cite{Legout}. Our construction of the complex is the direct generalization of the construction from \cite{Legout} to the case of an arbitrary contact manifold, while allowing augmentations that are only defined under some action level.

In Section \ref{sec:ModuliSpace}, we compactify the moduli spaces used in Sections \ref{sec:Banana} and \ref{sec:BananaInvariance}.
In Section \ref{sec:Banana}, we introduce a Rabinowitz Floer complex (RFC) as a mapping cone complex generated by Reeb chords. 
In Section \ref{sec:BananaInvariance}, we prove the invariance of this mapping cone complex. Compared to the invariance result from \cite{Legout}, we here need to take extra care for controlling the filtration-preserving properties, in order to establish invariance by a PWC family of complexes.

\subsection{Compactification of certain moduli spaces}
\label{sec:ModuliSpace}

Let $\Lambda_0^+, \Lambda_1^+ \subset Y$ be two Legendrians isotopic to $\Lambda_0^-,\Lambda_1^-.$  Assume $\bar{\Lambda}^\pm := \Lambda_0^\pm \cup \Lambda_1^\pm$ is embedded.
Let $ - \le t \le +$ parameterize this isotopy. (We use $\pm$ instead $\pm 1$ to avoid notational overuse of ``1.")
Further, let $L_0, L_1 \subset ( \R_\tau \times Y, d(e^\tau \alpha))$ be the exact Lagrangian concordance arising from the trace of the isotopy, with $L_i \cap  \{\tau\} \times Y = {\Lambda}_i^\pm$ for $\pm \tau \gg 1.$
Assume that $L := L_0 \cup L_1$ has at most one transverse double point $q.$

 There exists primitives of $e^\tau\alpha|_{TL_i}$ by the exactness. Since this primitive is necessarily locally constant wherever $L_i$ is cylindrical, we can fix a unique primitive that vanishes on the negative ends of $L_i$. After a small perturbation of $L_i$, we may assume that there is a non-zero difference of primitives at the unique intersection point $\{q\}=L_0 \cap L_1$.

We  consider several types of asymptotic behaviors for our holomorphic disks.
\begin{itemize}
\item {\bf{Mixed $\alpha^\pm$-chords.}}
Such a chord starts on $\Lambda_0^\pm$ (resp. $\Lambda_1^\pm$) and ends on $\Lambda_1^\pm$ (resp. $\Lambda_0^\pm$). 
\item {\bf{Pure $\alpha^\pm$-chords.}}
Such a chord both starts and ends on $\Lambda_0^\pm$ (resp. $\Lambda_1^\pm$). 
\item {\bf{Lagrangian intersection point $q.$}}
\end{itemize} 
Let $\Gamma$ be a non-empty cyclically ordered set of the above, each endowed with a sign.
Repetition is allowed.
%
Let $\mbox{Bd}\in \{L, \bar{\Lambda}^\pm\}.$
Let $\mathcal{M}^{d}(\Gamma; \mbox{Bd})$
denote the moduli space of $J$-holomorphic disks $u:D \rightarrow \R \times Y,$ with  boundary  marked points, satisfying the following conditions.
\begin{itemize}
\item
The boundary of the disk maps to $L$ if $\mbox{Bd}= L$ and to $\R \times \bar{\Lambda}^\pm $ if $\mbox{Bd}= \bar{\Lambda}^\pm.$
\item
The (formal) dimension of the component is $d.$
\item 
Each marked point maps to an element of $\Gamma.$ The cyclic ordering of marked points induced by the boundary orientation matches the cyclic ordering of the chords/double points in $\Gamma.$
\item
If a marked point maps to the double point (i.e. when $\mbox{Bd}= L$), then the puncture is {\bf positive} (resp. {\bf negative}) if the primitive of $e^\tau\alpha|_{L}$ evaluated along the boundary of the disk makes a jump to a lower (resp. higher) value at the puncture, while traversing the boundary in the positive direction.

If a marked point maps to a chord, 
the endowed sign $\pm$ indicates that it is an asymptotic limit at the $\pm \infty$ end of the symplectization boundary. 
\end{itemize}
\begin{rem}
  Consider the Legendrian lift of $L_0 \cup L_1$ to the contactization $(\R_\tau \times Y \times \R_Z,dZ+e^\tau\alpha)$ that is uniquely determined by the requirement that its $Z$-coordinate vanishes at $\tau=-\infty$. The sign of the puncture at a double point has the following direct reformulation in terms of the Reeb chord on this Legendrian. A disk has a positive (resp. negative) puncture at $q$ if and only if the value of the $Z$-coordinate along the boundary of the disk, as specified by the Legendrian lift, jumps to a higher (resp. lower) value at the puncture when traversing the boundary according to the orientation. This will be important later, when we describe the cobordism by the front projection of its Legendrian lift.
  \end{rem}

Note that $\mathcal{M}^{d}(\Gamma; \bar{\Lambda}^\pm)$ has an $\R$-translation in the range.
We denote the quotient of this translation by $\widehat{\mathcal{M}}^{d-1}(\Gamma; \bar{\Lambda}^\pm) = \mathcal{M}^{d}(\Gamma; \bar{\Lambda}^\pm)/\R.$
We next list different types of boundary conditions for the moduli spaces.
In all cases below, $\Gamma$ may have some additional negative pure chords.
\begin{enumerate}
\item[($1_\pm$):]  $\mbox{Bd} =  \bar{\Lambda}^\pm;$ $\Gamma$ contains two positive mixed chords.
\item[($2_\pm$):] $\mbox{Bd} =  \bar{\Lambda}^\pm;$ $\Gamma$ contains one positive and one negative mixed chord.
\item[($3_\pm$):] $\mbox{Bd} =  \bar{\Lambda}^\pm;$ $\Gamma$  contains one positive pure chord  and no mixed chords.
\item[(4):]   $\mbox{Bd} =L;$ $\Gamma$ contains two mixed chords (both positive).
\item[(5):]   $\mbox{Bd} =L;$ $\Gamma$ contains two mixed chords (one positive, one negative).
\item[(6):]   $\mbox{Bd} =L;$ $\Gamma$ contains one mixed chord (positive) and $q$ (positive or negative).
\item[(7):]   $\mbox{Bd} =L;$ $\Gamma$ contains one mixed chord (negative) and $q$ (positive or negative).
\end{enumerate}
Note that the sign of the puncture at $q$ in cases (6) and (7) can be recovered by the following data: the component of the starting point (or endpoint) of the mixed Reeb chord asymptotic of the disk, together with the two action values of the potentials at $q$ for each of the two sheets of $L$.

We now describe the boundary $\partial$  in the sense of the SFT-compactification \cite{BEHWZ}, also called Gromov--Hofer compactification,  of certain 1-dimensional moduli spaces, modding out by the $\R$-translation when one can. 
We illustrate the notation with some examples. The broken curve ``$(2_+ \times 1_+)$" has boundaries in two copies of $(\R \times Y, \R \times \Lambda^+ ).$
In the lower copy there is a curve of index 1 (rigid after $\R$-translation) of type ($1_+$).
In the upper copy there is one curve of index 1 (rigid after $\R$-translation) of type ($2_+$)  and 
one ``trivial strip" of index 0 (which is a curve of the form $(\R \times \mbox{chord}) $). 
(We omit listing the trivial strips.)
The broken curve ($6 \times 6$) has two index 0 curves of type ($6$) in the same copy of $(\R \times Y, L),$ one with a positive puncture at $q$ and the other with a negative puncture at $q.$

Figures \ref{fig:breaking1}, \ref{fig:breaking2}, \ref{fig:breaking3}, \ref{fig:breaking4}, \ref{fig:breaking5}, \ref{fig:breaking6}, depict the broken configurations corresponding to the boundary strata of the moduli spaces of the corresponding type. Note that, in these figures, any trivial strip inside a cylindrical level has been omitted. For every broken configuration in which there is a non-empty level with boundary on $\overline{\Lambda}^+$, as well as a middle level with boundary on $L$, there might be non-cylindrical components in the middle level with only pure punctures. Such components are not exhibited in the aforementioned figures; see Figure \ref{fig:breaking5.2} for an example where levels of this type arise in the boundary of the moduli space of type ($5$).

The figures depict the most general case that we will need, i.e.~the special case when $L_0=\R \times \Lambda$ is a trivial cylinder, which in particular means that $\Lambda_0^\pm=\Lambda$, while $L_1$ is a Lagrangian cylinder from $\Lambda_1^-$ to $\Lambda_1^+$.

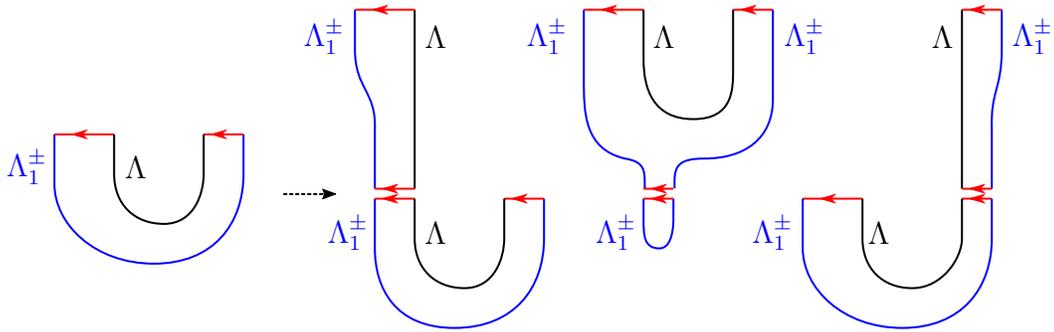
\begin{figure}[htp]
\vspace{5mm}
\centering
\labellist
\pinlabel $\color{blue}\Lambda^\pm_1$ at -10 61
\pinlabel $\color{blue}\Lambda^\pm_1$ at 186 110
\pinlabel $\Lambda$ at 31 61
\pinlabel $\Lambda$ at 230 110
\pinlabel $\color{blue}\Lambda^\pm_1$ at 111 36
\pinlabel $\Lambda$ at 144 36
\pinlabel $\color{blue}\Lambda^\pm_1$ at 102 110
\pinlabel $\Lambda$ at 144 110
\pinlabel $\color{blue}\Lambda^\pm_1$ at 283 110
\pinlabel $\Lambda$ at 335 110
\pinlabel $\color{blue}\Lambda^\pm_1$ at 369 110
\pinlabel $\color{blue}\Lambda^\pm_1$ at 212 36
\pinlabel $\Lambda$ at 311 36
\pinlabel $\color{blue}\Lambda^\pm_1$ at 271 36
\endlabellist
\includegraphics[scale=1]{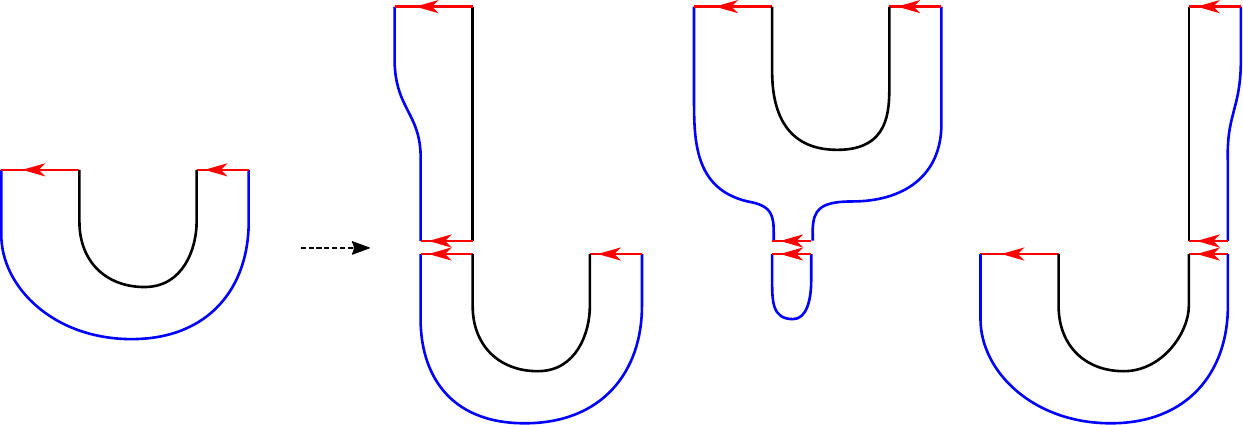}
\caption{Boundary points of the compactification of the moduli space $(1_\pm)$ of punctured disks with boundary on the cylinders over $\Lambda \cup \Lambda^\pm_1$ with two positive mixed Reeb chord asymptotics. Note that we have omitted any trivial strip inside the symplectization level from the figure. The broken configurations on the left and right belong to $(2_\pm \times 1_\pm)$ while the middle configuration belongs to $(1_\pm \times 3_\pm)$.}
\label{fig:breaking1}
\vspace{3mm}
\end{figure}

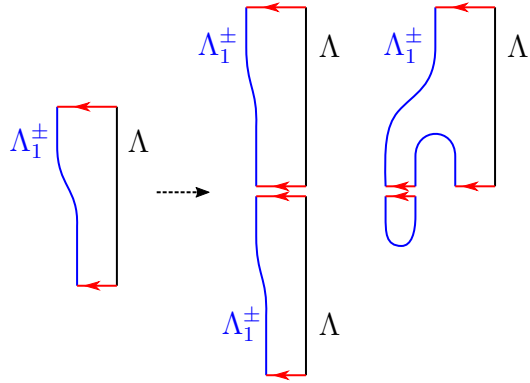
\begin{figure}[htp]
\vspace{5mm}
\centering
\labellist
\pinlabel $\color{blue}\Lambda^\pm_1$ at -10 90
\pinlabel $\Lambda$ at 31 90
\pinlabel $\Lambda$ at 103 21
\pinlabel $\color{blue}\Lambda^\pm_1$ at 70 21
\pinlabel $\Lambda$ at 103 126
\pinlabel $\color{blue}\Lambda^\pm_1$ at 61 126
\pinlabel $\color{blue}\Lambda^\pm_1$ at 131 126
\pinlabel $\Lambda$ at 174 126
\endlabellist
\includegraphics[scale=1]{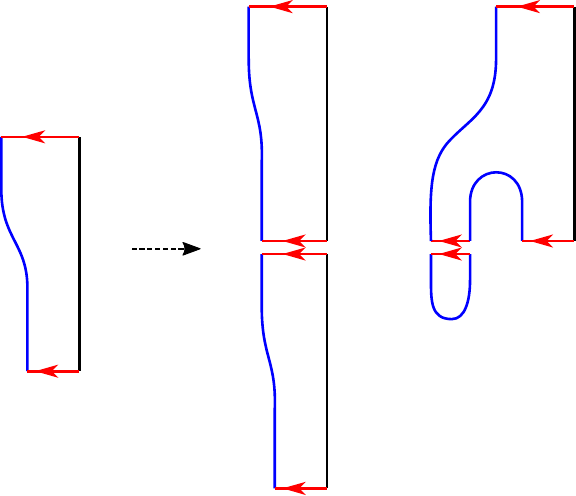}
\caption{Boundary points of the compactification of the moduli space $(2_\pm)$ of punctured disks with boundary on the cylinders over $\Lambda \cup \Lambda^\pm_1$. Note that we have omitted any trivial strip inside the symplectization level from the figure. The broken configuration on the left belongs to $(2_\pm \times 2_\pm)$ while the one one the right belongs to $(2_\pm \times 3_\pm)$. There are similar breakings for the disks whose mixed Reeb chords start on $\Lambda^\pm_1$ and end on $\Lambda$.}
\label{fig:breaking2}
\vspace{3mm}
\end{figure}

\begin{figure}[htp]
\vspace{5mm}
\centering
\labellist
\pinlabel $\Lambda$ at -8 90
\pinlabel $\Lambda$ at 38 91
\pinlabel $\Lambda$ at 27 57
\endlabellist
\includegraphics[scale=1]{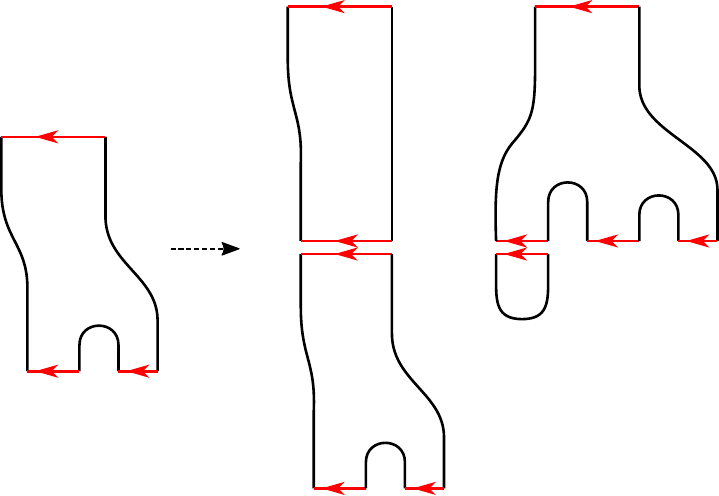}
\caption{Boundary points of the compactification of the moduli space $(3_\pm)$ of punctured disks with all boundary components on either the cylinder over $\Lambda$, $\Lambda_1^-$, or $\Lambda_1^+$. In other words, these are the curves that are used in the definition of the Chekanov--Eliashberg algebra of either Legendrian. All broken curves belong to $(3_\pm \times 3_\pm)$. Note that we have omitted any trivial strip inside the symplectization level from the figure.}
\label{fig:breaking3}
\vspace{3mm}
\end{figure}

\begin{figure}[htp]
\vspace{5mm}
\centering
\labellist
\pinlabel $\color{blue}L_1$ at -7 315
\pinlabel $\R \times \Lambda$ at 40 315
\pinlabel $\color{blue}L_1$ at 0 178
\pinlabel $\color{blue}\Lambda^+_1$ at -10 244
\pinlabel $\color{blue}\Lambda^+_1$ at 175 244
\pinlabel $\Lambda$ at 214 244
\pinlabel $\Lambda$ at 30 244
\pinlabel $\Lambda$ at 143 244
\pinlabel $\R \times \Lambda$ at 40 178
\pinlabel $\R \times \Lambda$ at 132 178
\pinlabel $\color{blue}L_1$ at 170 178
\pinlabel $\color{blue}\Lambda^-_1$ at 87 39
\pinlabel $\Lambda$ at 121 39
\pinlabel $\color{blue}L_1$ at 82 108
\pinlabel $\color{blue}L_1$ at -8 108
\pinlabel $\R\times \Lambda$ at 40 108
\pinlabel $\R \times \Lambda$ at 132 108
\pinlabel $\R \times \Lambda$ at 224 108
\pinlabel $\color{blue}L_1$ at 176 108
\pinlabel $\color{blue}\Lambda^-_1$ at 197 39
\pinlabel $\scriptstyle{\pm}$ at 32 78
\pinlabel $\scriptstyle{\mp}$ at 52 81
\endlabellist
\includegraphics[scale=1]{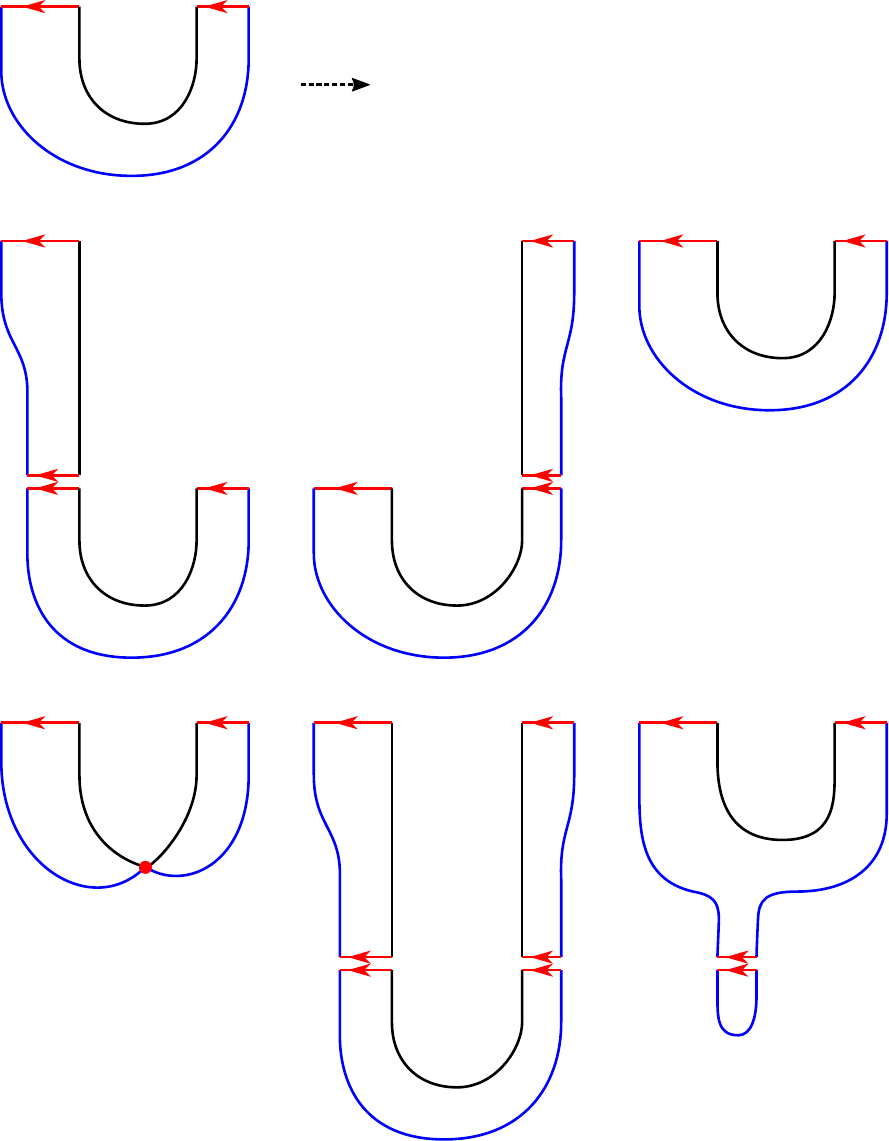}
\caption{Boundary points of the compactification of the moduli space $(4)$ of punctured disks with boundary on $L = L_0 \cup L_1$ with two positive mixed Reeb chords. Note that we have omitted any trivial strip inside the symplectization level from the figure. The broken configurations shown are as follows: The top row the left and middle configurations are in $(2_+ \times 4)$, while the one on the right is in $(1_+)$. On the bottom row from left to right, the configurations shown are in $(6\times 6)$, $(5 \times 5 \times 1_-)$, and $(4 \times 3_-)$.}
\label{fig:breaking4}
\vspace{3mm}
\end{figure}
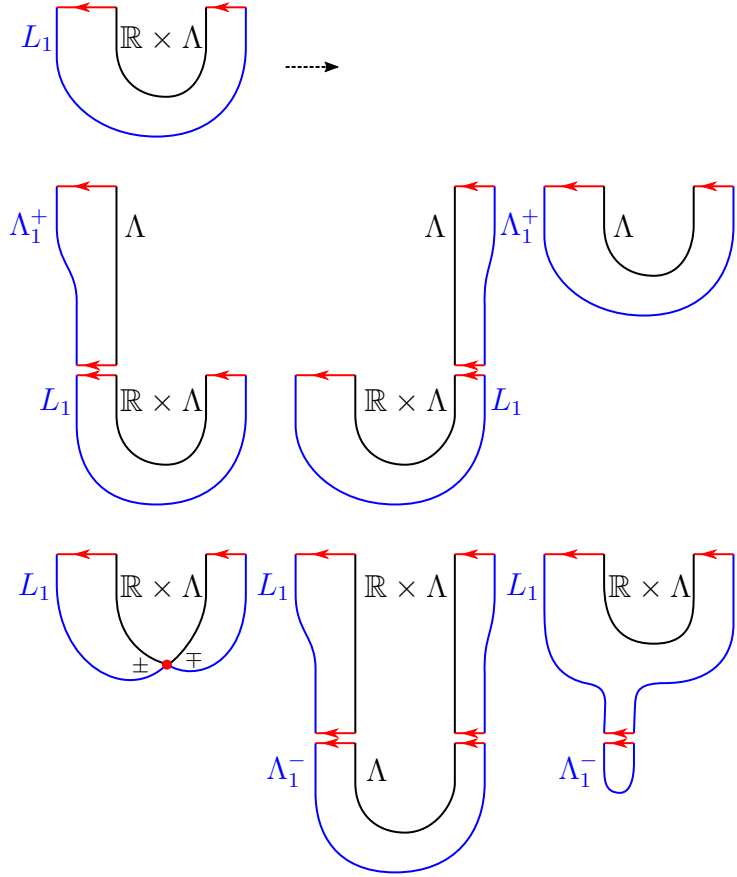

\begin{figure}[htp]
\vspace{5mm}
\centering
\labellist
\pinlabel $\color{blue}L_1$ at -8 125
\pinlabel $\color{blue}L_1$ at 124 125
\pinlabel $\color{blue}L_1$ at 180 125
\pinlabel $\R \times \Lambda$ at 41 125
\pinlabel $\color{blue}L_1$ at 68 125
\pinlabel $\R \times \Lambda$ at 111 83
\pinlabel $\color{blue}\Lambda^+_1$ at 63 199
\pinlabel $\Lambda$ at 101 199
\pinlabel $\R \times \Lambda$ at 228 125
\pinlabel $\Lambda$ at 162 51
\pinlabel $\color{blue}\Lambda^-_1$ at 125 51
\pinlabel $\R \times\Lambda$ at 172 83
\pinlabel $\R \times \Lambda$ at 307 125
\pinlabel $\color{blue}L_1$ at 258 125
\pinlabel $\color{blue}\Lambda^-_1$ at 240 51
\pinlabel $\scriptstyle{\pm}$ at 202 115
\pinlabel $\scriptstyle{\mp}$ at 204 98
\endlabellist
\includegraphics[scale=1]{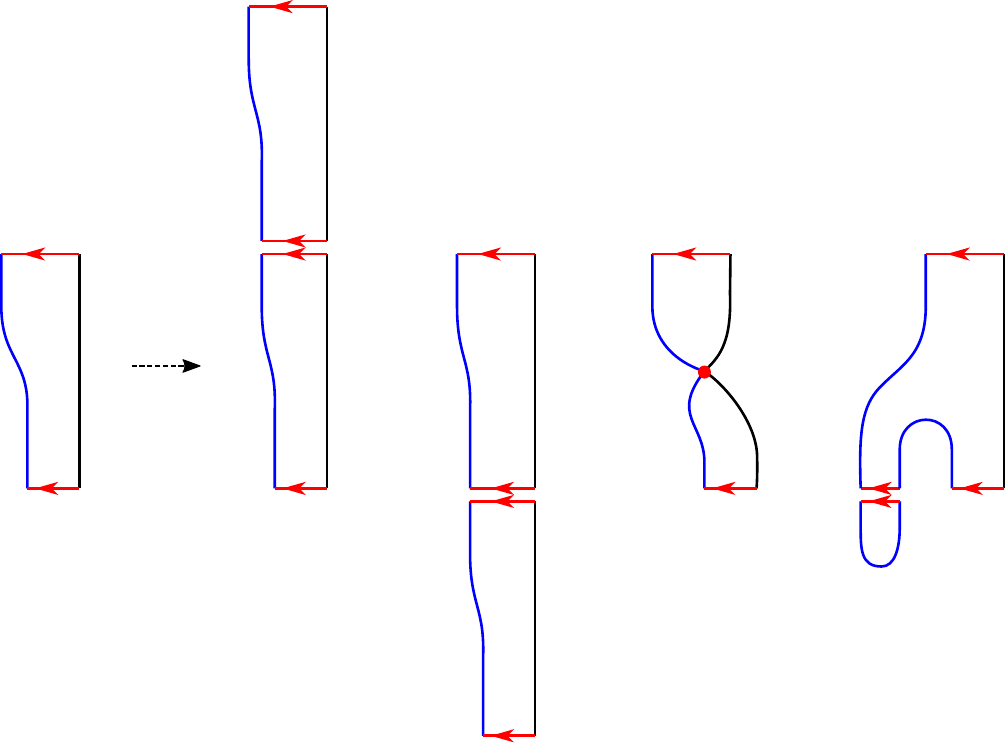}
\caption{Boundary points of the compactification of the moduli space $(5)$ of punctured disks with boundary on $L$ with one positive and one negative mixed Reeb chord puncture. Note that we have omitted any trivial strip inside the symplectization level from the figure. The broken configurations shown are in $(5 \times 2_-)$, $(5 \times 2_+)$, $(6 \times 7)$, and $(5 \times 3_-)$, going from left to right. There are analogous configurations when the mixed Reeb chords start on $L_1$ and end on $L_0.$ }
\label{fig:breaking5}
\vspace{3mm}
\end{figure}
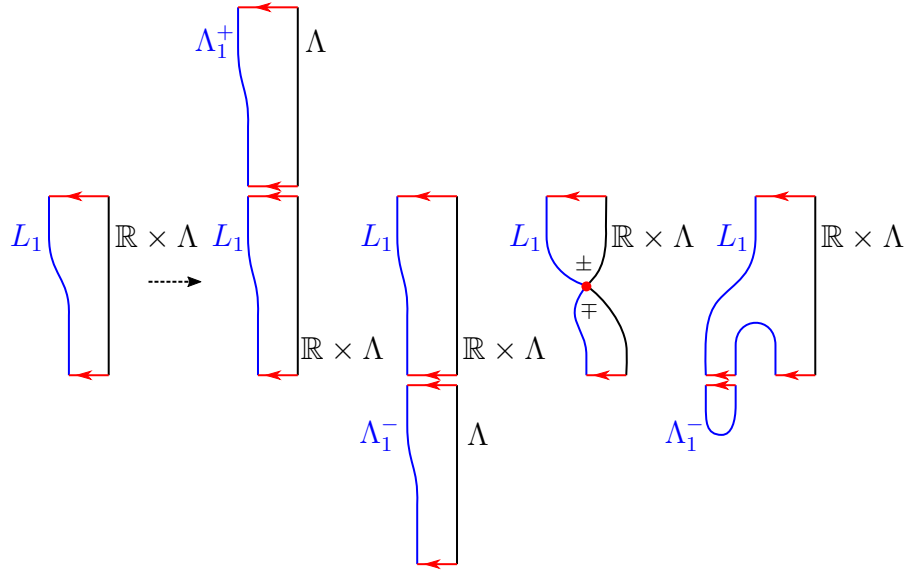

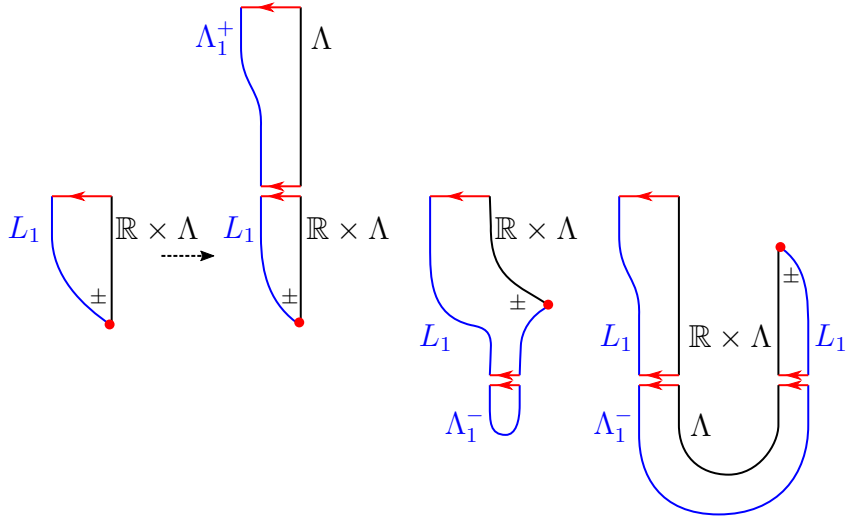
\begin{figure}[htp]
\vspace{5mm}
\centering
\labellist
\pinlabel $\color{blue}L_1$ at -10 108
\pinlabel $\color{blue}\Lambda^+_1$ at 62 180
\pinlabel $\R \times \Lambda$ at 40 108
\pinlabel $\R \times \Lambda$ at 112 108
\pinlabel $\Lambda$ at 102 180
\pinlabel $\color{blue}L_1$ at 145 67
\pinlabel $\color{blue}\Lambda^-_1$ at 156 34
\pinlabel $\color{blue}L_1$ at 71 108
\pinlabel $\R\times\Lambda$ at 183 108
\pinlabel $\color{blue}L_1$ at 213 67
\pinlabel $\color{blue}L_1$ at 294 67
\pinlabel $\R\times\Lambda$ at 256 67
\pinlabel $\color{blue}\Lambda^-_1$ at 212 34
\pinlabel $\Lambda$ at 245 34
\pinlabel $\scriptstyle{\pm}$ at 18 82
\pinlabel $\scriptstyle{\pm}$ at 90 82
\pinlabel $\scriptstyle{\pm}$ at 176 80
\pinlabel $\scriptstyle{\pm}$ at 279 90
\endlabellist
\includegraphics[scale=1]{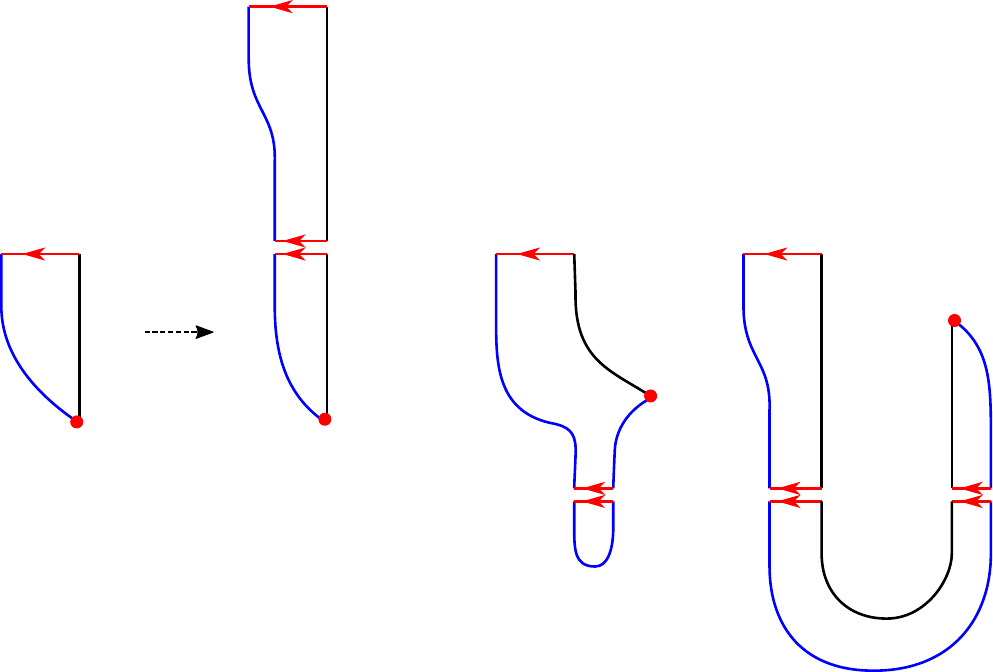}
\caption{Boundary points of the compactification of the moduli space $(6)$ of pseudoholomorphic curves with one puncture asymptotic to a positive mixed Reeb chord from $\R \times \Lambda$ to $L_1$, and one positive puncture at the unique interesection point. The broken configurations shown are in $(2_+ \times 6)$, $(6 \times 3_-)$, and $(2 \times 7 \times 1_-)$, going from left to right. There are analogous configurations in the case when the positive mixed Reeb chord starts on $L_1$ and ends on $\R \times \Lambda$.}
\label{fig:breaking6}
\vspace{3mm}
\end{figure}

\begin{prop}
\label{prop:MasterEQ}
For a generic almost complex structure,  the boundary of a one-dimensional moduli space is made of the following configurations of rigid moduli spaces. 
\begin{enumerate}
\item[($1_\pm$):]
 $(1_\pm \times 3_\pm), (2_\pm \times 1_\pm).$
\item[($2_\pm$):]
 $(2_\pm \times 2_\pm), (2_\pm \times 3_\pm)$
\item[($3_\pm$):]
$(3_\pm \times 3_\pm).$
\item[($4$):]
$(2_+ \times 4), (5 \times 5 \times 1_-),(1_+),(6 \times 6),(4 \times 3_-).$ 
\item[($5$):]
$(2_+ \times 5), (5 \times 2_-), (6 \times 7), (5 \times 3_-).$ 
\item[($6$):]
$(2_+ \times 6), (5 \times 7 \times 1_-), (6 \times 3_-).$ 
 \end{enumerate}
\end{prop}
The Lagrangian cobordisms we study in Section \ref{sec:Rabinowitz} will not have any 1-dimensional moduli space of type ($7$).

\begin{proof}

The mixed chords (or double points) each appears exactly once in the word of chords defining a given moduli space. If a disk has only pure chords, then Stokes' theorem implies it has a unique positive pure chord. Therefore, transversality for these spaces can be achieved by either perturbing $J$ near this distinguished positive chord as in \cite[Proposition 3.13]{GDRSurgery}, or by perturbing the Lagrangian boundary condition near $p$ as in \cite{EES05b}.

These configurations are the only ones that can appear because: (1) a dimension argument implies exactly one (non-trivial) moduli space in a symplectization; (2) each moduli space has at least one positive puncture.

\end{proof}

\subsection{The Rabinowitz complex as a mapping cone}
\label{sec:Banana}
In the following we assume that $\Lambda_0=\Lambda$ is fixed, and that $\Lambda_1^t=\phi^t_{\alpha, H_t}(\Lambda_1)$ is a Legendrian isotopy of $\Lambda_1^-=\Lambda_1$ for $- \le t \le +.$ We write $\bar{\Lambda}^t \coloneqq \Lambda_0 \cup \Lambda_1^t$, while $\bar{\Lambda}=\Lambda_0 \cup \Lambda_1$ as before. 

Fix $t \in [-,+]$ as in Section \ref{sec:ModuliSpace}. Choose a cylindrical almost complex structure that is regular for the moduli spaces involved (i.e.~that consist of disks for which there is a distinguished asymptotic that only appears once, and where the moduli space is of expected dimension at most two before taking the quotient by the action of translation). Recall that this is possible by \cite[Proposition 3.13]{GDRSurgery}.  Assume there exists augmentations  $\varepsilon^i$ defined for $\mathcal{A}^{l}(\Lambda_i)$, $i=0,1,$ which can be identified with an augmentation $\varepsilon_0$ of the DGA $\mathcal{A}^l_{\OP{pure}}(\bar{\Lambda}^0)$ that is generated by the pure chords.  In \cite[Lemma 3.4]{DRS2}, we describe how to define an augmentation $\varepsilon_t$ for
$\mathcal{A}^{l-l(t)}_{\OP{pure}}(\bar{\Lambda}^t)$
where 
$$
l(t) = \int_0^t \left(\max_{y \in Y} H_\tau(y) - \min_{y \in Y} H_\tau(y) \right) d\tau.
$$
The setting for \cite{DRS2} was limited to $(Y,\alpha) = (P \times \R_z, dz- \theta).$
However, \cite[Lemma 3.4]{DRS2} used only that the DGA underwent a stable-tame isomorphism.
Thus, Proposition \ref{prop:Bifurcation} implies that  \cite[Lemma 3.4]{DRS2} applies to our more current general setting.

Since the stable tame isomorphism given by Proposition \ref{prop:Bifurcation} are the DGA-quasi isomorphism induced by the Lagrangian trace cobordisms, we get

\begin{lem}
\label{lem:PullBackAug}
For a sufficiently fine subdivision $t_1<\ldots<t_N,$ the augmentation $\varepsilon_{t_i}$ can, moreover, be assumed to coincide with the pull-back of the augmentation $\varepsilon_{t_{i-1}}$ under the DGA-morphism induced by the exact Lagrangian trace of the isotopy $\Lambda_1^t$ for $t \in [t_{i-1},t_i]$.
  \end{lem}


In the below, let $x_{01}$ (resp. $x_{10}$) indicate a mixed chord starting on $\Lambda$ (resp. $\Lambda^t_1$) and ending on $\Lambda^t_1$ (resp. $\Lambda$). All other chords $z_i$ below are pure.
We define the action  of $x_{01}$ to be
$$\mathfrak{a}(x_{01}) \coloneqq  +\int_{x_{01}} \alpha>0$$
and of $x_{10}$ to be
$$\mathfrak{a}(x_{10}) \coloneqq  -\int_{x_{10}} \alpha<0.$$
Fix $a_t,b_t \in \R$ such that
\begin{equation}
\label{eq:parameters}
0< b_t-a_t < l- l(t).
\end{equation}
Denote by
$$C_*([a_t,b_t)) = C_*(\R_{\ge0} \cap [a_t,b_t)) \quad \text{and}\quad C^*([a_t,b_t))=C^*(\R_{\le0} \cap [a_t,b_t))$$
the ``linearized Legendrian contact homology complex'' and the ``dual co-complex,'' respectively, where the former is generated by the mixed chords $x_{01}$ while the latter is generated by the mixed Reeb chords $x_{10}$. In both cases we assume that  the action $\mathfrak{a}$ of the generators  is contained inside the interval $[a_t,b_t)$.

As prescribed below, the differential restricted to $C_*([a_t,b_t))$ (resp.~$C^*([a_t,b_t))$) is the usual differentials of the Legendrian contact homology (co)complex linearized by the augmentation $\varepsilon_t$ that counts trips with one positive (resp. negative) mixed input Reeb chord and one negative (resp. positive) mixed output Reeb chord; see below for the precise formula.

\begin{rem}
For the co-complex, the differential increases the Reeb chord length $\ell$, hence decreases the above action $\mathfrak{a}$. This is why when taking the quotient complex, we consider $C^*([a_t,b_t))$ and not $C^*((a_t,b_t]).$
  \end{rem}

Recall that the linearized Legendrian contact (co)homology complex of a pair of Legendrians can be endowed with a grading that depends on additional choices of the two Legendrians involved; see \cite{Dimitroglou:Cthulhu} as well as Appendix \ref{sec:maslovpotential}. For simplicity we restrict ourself to the case when the first Chern class of $(Y,\alpha)$ vanishes, which means that we can choose a symplectic trivialization of the (square of the) determinant $\C$-line bundle $\det \xi \to Y$. Given the choice of a homotopy class of such a symplectic trivialization, together with choices of Maslov potentials of $\Lambda_i$ as described in Appendix \ref{sec:maslovpotential}, we get a canonically induced $\Z$-grading for which the linearized differential (resp. codifferential) is of degree $-1$ (resp. $1$). In addition, the choice of a Maslov potential can be naturally extended over a Legendrian isotopy. Note that a loop of Legendrians can induce a non-trivial action on its set of Maslov potentials (see Lemma \ref{lem:actionshift} for an example).

For a pair of mixed chords $x,y$ in either complex, and ordered (possibly empty) sets of pure chords $\mathbf{z},\mathbf{z}'$ define
$$\mathbf{m}_{\bar{\Lambda}^t} (x^+, y^\pm) = \sum_{\mathbf{z}, \mathbf{z}'} \sharp \widehat{\mathcal{M}}^0(x^+\mathbf{z}^- y^\pm \mathbf{z}'^-; \bar{\Lambda}^t) \varepsilon_{t}(\mathbf{z} \mathbf{z}').$$
Note that the action of the individual pure chord $z$  in $\mathbf{z}$ or $ \mathbf{z}'$ is less than $b_t-a_t,$ 
and so $\varepsilon_t(z)$ is defined.
We suppress the subscript notation when we define the maps
\begin{gather*}
d_{01}: C_*([a_t,b_t)) \rightarrow C_*([a_t,b_t)),\\
d_{10} : C^*(
{{[a_t,b_t)}}
)
 \rightarrow C^*([a_t,b_t)),\\
B : C_*([a_t,b_t)) \rightarrow C^*([a_t,b_t]),
\end{gather*}
where
\begin{eqnarray*}
d_{01} (x_{01}) & = & \sum_{\{y_{01}| \mathfrak{a}(y_{01}) \in [a_t,b_t) \cap \R_{>0} \}} \mathbf{m}_{\bar{\Lambda}^t} (x_{01}^+, y_{01}^-)y_{01}  \\
d_{10} (x_{10}) & = & \sum_{\{y_{10}| \mathfrak{a}(y_{10}) \in [a_t,b_t) \cap \R_{<0} \}} \mathbf{m}_{\bar{\Lambda}^t} (y_{10}^+, x_{10}^-)y_{10},\\
B (x_{01})  & = & \sum_{\{y_{10}| \mathfrak{a}(y_{10}) \in [a_t,b_t) \cap \R_{<0} \}} \mathbf{m}_{\bar{\Lambda}^t} (x_{01}^+, y_{10}^+)y_{10}.
\end{eqnarray*}

Note that $d_{01}$ and $d_{10}$ is the  linearized  Legendrian contact homology differential and co-differential, respectively.

\begin{prop}
\label{prop:BananaComplex}
The matrix $d_{B} := \begin{pmatrix}
d_{01} & 0 \\
B & d_{10}
\end{pmatrix}$
is a filtered mapping cone differential for the filtered chain map $B.$
\end{prop}

\begin{proof}
Proposition \ref{prop:MasterEQ}($1_\pm$) and \ref{prop:MasterEQ}($2_\pm$) imply the matrix squares to zero. 
For example, suppose $x_{01}\in C_*([a_t,b_t))$ and $z_{10} \in C^*([a_t,b_t))$. Then
\begin{eqnarray*}
\lefteqn{\langle d_B^2 x_{01}, z_{10} \rangle =}\\
& = &\sum_{ y_{01} } \mathbf{m}_{\bar{\Lambda}^t} (x_{01}^+, y_{01}^-) 
\mathbf{m}_{\bar{\Lambda}^t} (y_{01}^+, z_{10}^+) 
+  \sum_{ y_{10}} \mathbf{m}_{\bar{\Lambda}^t} (x_{01}^+, y_{10}^+) \mathbf{m}_{\bar{\Lambda}^t} (y_{10}^-, z_{10}^+)
\end{eqnarray*}
which vanishes by Proposition \ref{prop:MasterEQ}($1_\pm$).
\end{proof}

\begin{defn}
\label{defn:MCC}
Let $(a_t,b_t,\varepsilon_t, l-l(t), \bar{\Lambda}^t)$ denote the auxiliary information used to define this cone complex. We denote this {\bf{mapping cone complex}} as 
$$RFC_*^{[a_t,b_t)}(\Lambda_0,\Lambda^t_1;\varepsilon_t) = (C_*(t)[a_t, b_t) \oplus C^{n-*-2}(t)[a_t, b_t), d_B)$$
which naturally is a filtered chain complex with action window $[a_t,b_t)$ and filtration induced by $\mathfrak{a}$.
\end{defn}

\begin{rem}
\begin{enumerate}
\item
  The sign change and shift of grading of the second summand is needed in order for the summand $B$ of the differential $d_B$ to be of degree $-1$. The relevant index for the disks with two positive punctures whose count defines $B$ was computed in \cite[Lemma 2.5]{Ekholm:Duality}.

Further note that, with this convention,  the degree of a generator is continuous under deformations through transverse chords, even when the length of the chord at some point becomes zero, such that the component of the starting and endpoints become interchanged; this readily follows from the definition of the grading in Appendix \ref{sec:maslovpotential}.

The mapping cone complex also can be defined without a grading. 
While the grading is necessary (and warranted) to prove Theorem \ref{thm:RP}, the other results in Section \ref{sec:Introduction} need only the ungraded complex.

  \item
  We suppress the upper bound notation $l-l(t)$  in an attempt to reduce the overbearing set of decorations.
  To be consistent with similar concepts in other literature, we sometimes call the mapping cone complex the {\bf Rabinowitz Floer complex}, and when the Legendrian is augmentable (so that we can choose $b_t=l = -a_t=\infty$)  denote it by $RFC_*(\Lambda, \Lambda_t).$ 
  \item
   Note that the Legendrian contact cohomology complex $C^*((-\infty,a])$ is the degree-wise dual of a possibly infinite dimensional complex. In such a case the chords $x_{10}$ do not form a basis of $RFC_*(\Lambda,\Lambda_t)$, since one must also allow formal infinite sums of such chords. In other words, in each fixed degree $i \in \Z$, the filtered vector space $RFC_i^{(-\infty,+\infty)}(\Lambda,\Lambda_t)$ is the inverse limit of the directed system
  $$ RFC_i^{[0,+\infty)} \leftarrow  RFC_i^{[-1,+\infty)} \leftarrow RFC_i^{[-2,+\infty)} \leftarrow \ldots$$
  of filtered vector spaces and canonical quotient maps induced by the filtration.
\end{enumerate}
  \end{rem}

  \begin{thm}
    \label{thm:pwc}
    Assume that the following is satisfied:
    \begin{itemize}
    \item $l-l(-)>0$ is smaller than the length of any contractible periodic Reeb orbit $\gamma$ of degree $|\gamma|\le 1$;
    \item $\varepsilon_{-}$ is an augmentation of the sub-DGA $\mathcal{A}_{\OP{pure}}^{l-l(-)}(\Lambda_0\cup \Lambda_1)$ generated by pure Reeb chords; and
    \item $\Lambda^t_1$ is generated by a Legendrian isotopy of oscillation
      $$l(t)=\int^t_{-} \left(\max_{y \in Y} H_\tau(y) - \min_{y \in Y} H_\tau(y) \right) d\tau.$$
    \end{itemize}
    Then, as long as $l-l(t) > b_t-a_t > 0$  and $[a_t,b_t)$ is a finite interval,  there exists a sequence of augmentations $\varepsilon_t$ of the DGAs $\mathcal{A}_{\OP{pure}}^{l-l(t)}(\Lambda_0 \cup \Lambda^t_1)$ that makes $RFC^{[a_t,b_t)}_*(\Lambda_0,\Lambda_1^t;\varepsilon_t)$ into 
 a  well-defined PWC family of complexes with action-window $[a_t,b_t)$. In particular, the complexes undergo the deformations specified by the Barcode Proposition \ref{prop:Barcode} as $t$ varies.

In the case when, $l=+\infty$, $a_t=-\infty$, and $b_t=+\infty$ holds for all $t$, then the homology of the entire complex is invariant under Legendrian isotopy. In addition, in any smooth family of finite action windows, we may assume that we have a PWC family of complexes.
\end{thm}
 The analogous invariance also holds when the first Legendrian is deformed by a Legendrian isotopy, while the second copy is being fixed. The proof is completely analogous and left to the reader.

\begin{rem}
\label{rem:OtherNorms}
\begin{itemize}
  \item Theorem \ref{thm:pwc} holds if we replace $l(t)$ with the oscillation underlying the Shelukhin--Chekanov--Hofer norm or the Usher norm $$
l_1(t) = \int_-^t \max_{y \in Y} H_\tau(y) d\tau
\quad
\mbox{or}
\quad
l_2(t) = l(t) + \max_{y \in Y} |g(y)|.
$$
(Here $g(y)$ is the conformal factor for the time-$t$ contactomorphism defined by $H.$ See \cite[Definition 10.1]{RosenZhang}.)
This follows because $l(t) \le l_1(t), l_2(t).$

\item If $H$ defines a contact-form-preserving contactomorphism, i.e., if the conformal factor is 0, then the length of the pure chords are preserved. It then follows from Proposition \ref{prop:Bifurcation} that the stable-tame isomorphism class of the Chekanov--Eliashberg algebra $\mathcal{A}^l$ does not depend on $t$. In particular $\mathcal{A}^l$ has an augmentation for all $t$, and we can improve Theorem \ref{thm:pwc} by dropping the condition $l-l(t)>0$.
  \end{itemize}
\end{rem}

To prove Theorem \ref{thm:pwc}, we write the isotopy as a concatenation of short isotopies.
Below any action level upper bound, we can assume there is at most one of the following singular moments: a mixed chord that enters or leaves the action window $[a_t,b_t);$  the birth/death of a pair of chords occurs; the action of exactly two mixed chords coincide; the action of a pure chord equals $l - l(t);$ and the action of a mixed chord vanishes.
All of these moments, except the last one, are considered in \cite[Proposition 3.5]{DRS2}, which is the equivalent of Theorem \ref{thm:pwc} in the special case $(Y, \alpha) = (P \times R_z, dz -  \theta).$  The case when the action of a mixed chord vanishes corresponds to when the Legendrians $\Lambda_0$ and $\Lambda_1^t$ intersect. By genericity we can always assume that there are only finitely many such moments in the family. The invariance at these moments is taken care of in Section \ref{sec:onedoublepoint} below.

To recycle the reasoning of \cite[Proposition 3.5]{DRS2}, we will apply the algebraic machinery from Section \ref{sec:Algebra}, in particular Propositions \ref{prop:Barcode} and \ref{prop:PWC}, to our current geometric set-up. 
(In particular, the algebraic interpretation of bifurcation analysis done in \cite{DRS2} is replaced by Proposition \ref{prop:Bifurcation}.)
And to apply the Section \ref{sec:Algebra} results, we need to establish the hypotheses \ref{A1} and \ref{A2} for Lemma \ref{lem:smalldeg}.

 Finally, in the case when $l=+\infty$, $a_t=-\infty$, $b_t=+\infty$, the proof of the invariance of the complex is simpler, since we do not need to care about whether the maps consist of standard bifurcations, i.e.~handle-slides or birth/deaths. (In other words, we do not need Lemma \ref{lem:HomotopyInverse}). Instead, the weaker property of quasi-isomorphism follows by standard invariance properties of the linearized Legendrian contact homology as in \cite{Ekholm:Floer}, combined with the treatment of the double point and hypothesis \ref{A2} in Section \ref{sec:onedoublepoint} below.

    
\subsection{Mapping cone complex invariance during a short Legendrian isotopy}
\label{sec:BananaInvariance}

We now consider a varying $t \in [-,+].$
Assume that the interval $[-,+]$ is small.
As in Section \ref{sec:ModuliSpace},  $L$ denotes the Lagrangian isotopy-trace concordance of $\bar{\Lambda}^t.$
Since the oscillatory norm $\epsilon = l(+) - l(-)$ isotopy from $\bar{\Lambda}_-$ to $\bar{\Lambda}_+$ is arbitrarily small, 
we assume that there are no pure chords of length in the interval $(l-l(+), l-l(-)).$
Thus we can use the same action window $[a_-, b_-) = [a_+,b_+) = [a,b)$ for both complexes.
Moreover, we assume that no mixed chords enter or leave this action window, and that there are no birth/death pairs of chords when $t \in [-,+].$

\subsubsection{One double point}
\label{sec:onedoublepoint}
Suppose $\Lambda \cap  \Lambda_1^t  = \emptyset$ when $t \in [-,+]\setminus\{0\}$ and that there exists a unique intersection point 
$\Lambda \cap \Lambda_1^0  =  \{q\}$
 that is transverse in a one-parameter family. In particular, the double point arises as a transverse family of mixed Reeb chords $c_t$ whose action $\mathfrak{a}(c_t)$ changes sign at $t=0$.

\begin{rem}\label{rem:cNegative}
For the remainder of this discussion, we assume that $c_-$ runs from 
$\Lambda_1^-$ 
to $\Lambda,$ and hence $\mathfrak{a}(c_-) < 0 < \mathfrak{a}(c_+).$ The case when $c_-$ runs from $\Lambda$ to $\Lambda_1^-$   
follows from this case by considering the reverse-time isotopy.
\end{rem}

For some sufficiently small interval $[-,+],$ after a possible $C^0$-small perturbation, we can
choose a contact-form preserving Darboux neighborhood $U \subset Y$ of $q$ such that the following hold.
\begin{itemize}
\item 
 There is a contact-form preserving identification
  $$U \cong (B_1)_x \times (B_1)_y \times [-\epsilon,\epsilon]_z \subset (B_1)_x \times \R^n_y \times \R_z= 
   ( J^1(B_1),  \eta (dz - p_x dx))$$
for some $\eta >0$ that we can make arbitrarily small, and 
where $B_r \subset \R^n$ is the disk centered at ${\mathbf{0}}$ of radius $r;$ 
\item
$\Lambda \cap U$ is the ${\mathbf{0}}$-section $j^10$;
\item
 $ \Lambda_1^t  \cap J^1(B_{1}) = j^1(f_t)$ for some $f_t \colon B_{1} \rightarrow \R$ that smoothly varies with $t$;
\item
For all $x \in B_{1} \setminus B_{2/3},$ $f_t(x)$ is independent of $t$;
\item
For all $x \in B_{1/3},$ $f_t(x) = \|x\|^2  +  t$;
\item
For all $x \in B_{1},$ $(df_t)^{-1}(0) ={\mathbf{0}}$;
\item
The chord $c_t \subset \mathbf{0} \times \mathbf{0} \times \R_z$  is the unique mixed chord of $\Lambda$ and $ \Lambda_1^t $ in $U$.
\end{itemize}

In order to describe the trace cobordism it is useful to utilize the exact symplectomorphism
\begin{gather*}(\R_\tau \times \R^n_x \times \R^n_y \times \R_z,d(e^\tau(dz-ydx))) \xrightarrow{\cong} (\R_{>0})_{q} \times \R^n_x \times \R^n_{p_x} \times \R_{p_q},\\
  (\tau,x,y,z) \mapsto (e^\tau,x,e^\tau y,z),
\end{gather*}
where the target is equipped with the symplectic form $d(q\, dp_q-p_x\,dx)$. For the choice of primitive $d(-p_q\,dq-p_x\,dx)$ one can describe exact Lagrangians via their front projection to the contactization
$$ (((\R_{>0})_{q} \times \R^n_x \times \R^n_{p_x} \times \R_{p_q}) \times \R_Z, dZ-p_q\,dq-p_x\,dx).$$
 The non-cylindrical part of the immersed Lagrangian trace cobordism inside $\R \times U$ is constructed in these coordinates via the front projection shown in Figures \ref{fig:intersection} and \ref{fig:disc}. We can assume that:
\begin{itemize}
\item The Lagrangian trace is cylindrical away from the intersection point, i.e.
  $$L \cap (\R_\tau \times (Y \setminus U)) = \R_\tau \times ( (\Lambda_1^+ \cup \Lambda)  \setminus U) = \R_\tau \times (( \Lambda_1^- \cup \Lambda) \setminus U);$$
holds for some small neighborhood $U$ of $q$;
\item there exists $A<0<B$ so that $L \cap \{A \le \tau \le B\}$ contains the unique intersection point $q$ in the level $\{\tau=0\}$
\item
   In the above identification of $\R\times U$ with a subset of $\R \times J^1B_1$, the non-cylindrical part of $L$ is identified with the exact Lagrangian immersion whose front is as shown in Figures \ref{fig:intersection} and \ref{fig:disc}. (The figures depict $L$ in the coordinates $(q=e^\tau,x,p,p_x,Z)$ on the contactization of the symplectization $\R \times U$ described above.)
   \item The pull-back of the Liouville form $e^\tau\alpha$ to $L$ has a primitive that vanishes outside of a compact subset; this is equivalent to the front projection in Figure \ref{fig:intersection} to consist of sheets that are of the form $e^\tau f(x)$ outside of a compact subset.
 \item In the subset of $\R \times U \hookrightarrow \R \times J^1B_1 $ that lives inside $\R \times J^1B_{1/3}$, the immersed trace $L$ is, moreover, of the form
   \begin{equation}
     \label{eq:reebtrace}
     \R \times \Lambda \:\: \cup \:\: \left\{(\tau,\phi^{\gamma(\tau)}_R(x)); \:\: x \in \Lambda_1^-\right\}
     \end{equation}
   for some smooth $\gamma(\tau) \ge 0$ (but not monotonously increasing) function that vanishes when $\tau \le A$, and which is constant when $\tau \ge B$. Recall that $\phi^t_R$ is the Reeb flow of $\alpha$.
\end{itemize}

\begin{figure}[htp]
\vspace{5mm}
\centering
\labellist
\pinlabel $Z$ at 0 64
\pinlabel $e^\tau$ at 11 164
\pinlabel $L_0=\R_\tau\times\Lambda_0$ at 60 25
\pinlabel $x$ at 190 7
\pinlabel $-1/3$ at 87 -2
\pinlabel $1/3$ at 107 -2
\pinlabel $\color{red}c_-$ at 99 23
\pinlabel $\color{red}c_+$ at 99 135
\pinlabel $\color{red}q$ at 101 64
\pinlabel $\color{blue}L_1$ at 60 143
\endlabellist
\includegraphics[scale=1.5]{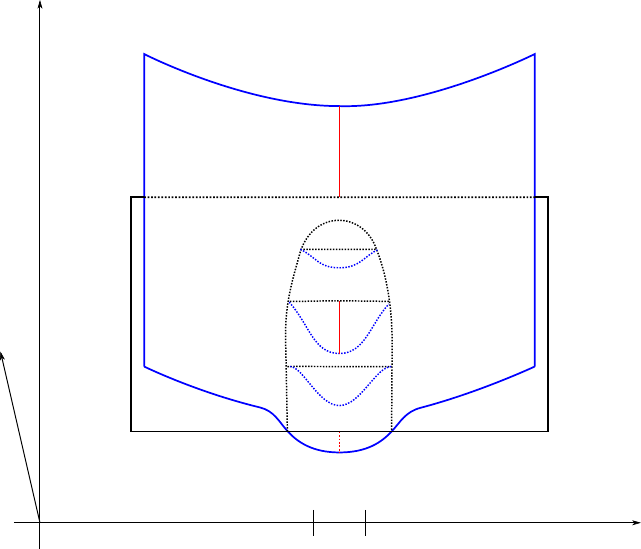}
\vspace{3mm}
\caption{The front projection of the Legendrian lift of the exact Lagrangian immersion $L=L_0  \cup L_1$ inside $\R_\tau \times U$ to the contactization $(\R_\tau \times U) \times \R_Z$ of the symplectization. The Reeb chord $q$ on the Legendrian lift corresponds to the unique double point $\{q\}=L_0  \cap L_1$. 
The cobordism is cylindrical with a vanishing primitive of $e^\tau\alpha$ outside of a compact subset if all sheets of the front are of the form $e^\tau f(x)$ outside of a compact subset. In order for \eqref{eq:reebtrace} to hold inside the subset $\{\|x\| \le 1/3\}$, it suffices to consider a front which  
is the graph of a function of the form $e^{\tau}(\tilde f(x)+g(e^\tau))$ 
above $B_{1/3}^n.$}
\label{fig:intersection}
\end{figure}
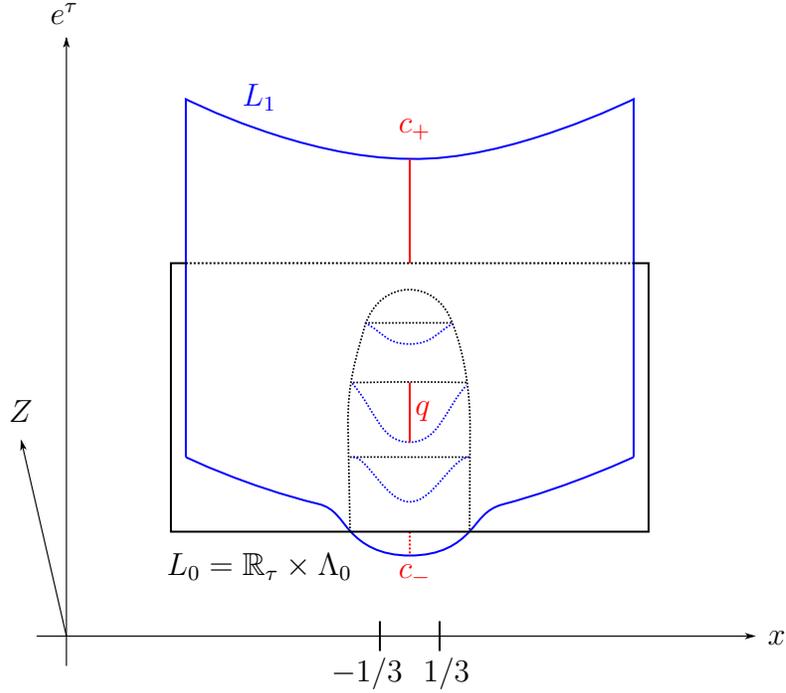

\begin{figure}[htp]
\vspace{5mm}
\centering
\labellist
\pinlabel $Z$ at 7 88
\pinlabel $\color{red}q$ at 89 35
\pinlabel $\color{red}c_+$ at 174 38
\pinlabel $e^\tau$ at 190 27
\pinlabel $\color{red}c_-$ at 32 38
\pinlabel $e^B$ at 144 15
\pinlabel $\color{blue}L_1$ at 144 63
\pinlabel $e^A$ at 66 38
\pinlabel $z$ at 221 86
\pinlabel $\color{red}q$ at 307 33
\pinlabel $\color{red}c_+$ at 390 53
\pinlabel $\tau$ at 400 40
\pinlabel $\color{red}c_-$ at 240 31
\pinlabel $B$ at 360 30
\pinlabel $u^+$ at 340 53
\pinlabel $u^-$ at 260 31
\pinlabel $A$ at 278 52
\pinlabel $\R_\tau\times\Lambda_0$ at 248 51
\pinlabel $\color{blue}L_1$ at 250 10
\endlabellist
\includegraphics[scale=1]{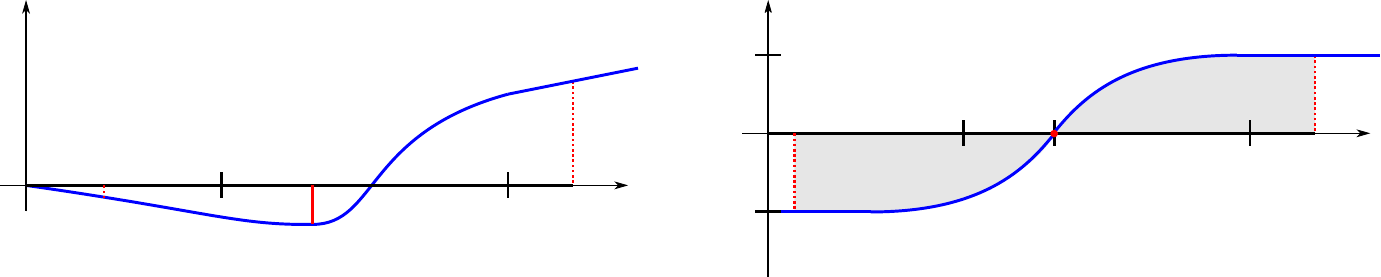}
\vspace{3mm}
\caption{Left: The front projection the Legendrian lift of the exact immersed Lagrangian cobordism $L=L_1 \cup (\R \times \Lambda_0)$ 
 to the contactization $(\R_\tau \times  B_{1/3} ) \times \R_Z$ of the symplectization. Right: the image of small disks $u^\pm$ with one puncture at the Reeb chord $c_\pm$ and a  positivepuncture at the double point $q$ which is contained inside a slice consisting of the image of the line $\R_\tau \times \{\mathbf{0}\}$ under the Reeb flow.}
\label{fig:disc}
\end{figure}

\begin{rem}
Equality \eqref{eq:reebtrace} holds because $\Lambda_1^t$  is assumed to induced by the Reeb flow inside the neighborhood $J^1B_{1/3}$ of the double point.There is one subtle point here: the function $\gamma(\tau)$ must be carefully chosen in order for $e^\tau\alpha$ to admit a compactly supported primitive when pulled-back to $L$,  as postulated by the fourth bullet point above. This condition on the primitive is easier to describe in the coordinates of the front projection: the front  should be the graph of a suitable function of the form $e^\tau(\tilde f(x)+g(e^\tau))$ above $B_{1/3}$, where $g(e^\tau)$ is constant for all $\pm \tau \gg 0$ sufficiently large.
  \end{rem}



\begin{lem}
\label{lem:SmallDisk}
For $[-,+]$ sufficiently small, the following holds.
For a suitable compatible almost complex structure on $\R \times Y$ which is cylindrical outside of a compact subset, there 
exists a \emph{unique} pseudo-holomorphic disk  $u^+$ (resp. $u^-$) with positive (resp. negative) asymptotic puncture at $c_+$ (resp. $c_-$), a  positive  puncture at $q,$ and no other punctures.
There exists a neighborhood of adjusted almost complex structures (see Section \ref{sec:Geometry}) in which a generic choice $J$  makes $u^\pm$ rigid and transversely cut out.
Furthermore,  there are no disks with a negative puncture at $q$, and with all other punctures negative Reeb chord asymptotics of which precisely one is mixed (and thus necessarily going from $\Lambda$ to $\Lambda_1^-$).

%
%

\end{lem}

\begin{proof}
From the above assumptions (including Formula \eqref{eq:reebtrace}),  there exists a local projection
$$\pi: \R \times U \subset \R \times J^1(B_1) \rightarrow T^*(B_1), \quad \mbox{such that}
$$
$$
\pi(L_1 \cap (\R \times U)) \cap T^*B_{1/3}= \{(x,2x)\} \cap T^*B_{1/3},
\quad
\pi(\R \times (\Lambda \cap B_{1/3})) = \{(x,0)\}.
$$
Let $J_{T^*(B_1)}$ be any almost complex structure compatible with the standard symplectic structure $\omega_0$ on $T^*(B_1).$
This lifts to a unique cylindrical  almost complex structure $J_U$ on $\R \times U$  for which  $\pi$ is $(J_U, J_{T^*(B_1)})$-holomorphic.

If $J_{T^*(B_1)}$ is integrable in a neighborhood of the double point $\mathbf{0} \in T^*(B_1)$, then \cite[Lemma 8.3(1)]{GDRLift}
proves the transversality result for the strip $u_\pm \subset \pi^{-1}(\mathbf{0})$ whose image is the trace of the isotopy of the Reeb chord $c_t$ between $q$ and $c_\pm.$ This strip is depicted in Figure \ref{fig:disc}.

Extend $J_U$ to an adjusted almost complex structure $J$ for $\R \times Y.$
Since transversality is an open condition, we can perturb $J$ generically to assume all rigid holomorphic disks are transversely cut out.

We will show uniqueness of $u = u^+$  first using a monotonicity result for Lagrangian surgery cobordisms \cite[Lemma 5.1]{GDRSurgery} (the current setup is similar), and then using a monotonicity result for compact Lagrangians \cite[Proposition 4.7.2]{AudinLafontaine}. 
The $u^-$ case is similar.

Following the \cite{GDRSurgery} notation, let $- =A$ and $+ = B$ so that $L$ is cylindrical outside of  $\{A \le \tau \le B\} \subset \R_\tau \times Y.$
\cite[Section 3.3.1]{GDRSurgery} defines the {\bf{total energy}}
$$E_{[A,B]}(u)=e^{-A}\int_u d(e^{\varphi(\tau)} \alpha) + \sup_{\rho(\tau)}\int_u \rho(\tau)d\tau\wedge \alpha. $$ 
 This total energy is the sum of the $d(e^{\varphi(\tau)} \alpha)$-energy of $u,$ where
$$\varphi(\tau)=\begin{cases}
  A, & \tau \le A,\\
  \tau, & \tau \in [A,B],\\
  B, & \tau \ge B,
\end{cases}$$
and the $(d\tau \wedge \alpha)$-energy of $u \cap \{\tau \notin [A,B]\},$ and where the $\rho(\tau)$ are non-negative bump functions that have compact support  contained in precisely one of the subsets
\begin{itemize}
\item $(-\infty,A]$, in which case $\int_\R \rho(\tau)dt=1$, or in
\item $[B,+\infty)$, in which case $\int_\R \rho(\tau)dt=e^{B-A}.$
\end{itemize}

Let $h(q_+)>h(q_-)=0$ be the the primitives of $e^\tau\alpha|_{TL}$ at the two sheets of the double point of $L$, whose size is controlled by the size of $\ell(c_\pm)$ (see the difference of $Z$--coordinates in Figure \ref{fig:intersection}).
Applying \cite[Proposition 3.11 (2) and (3)]{GDRSurgery} to the first and second term of the total energy, we get 
    \begin{align*}
      & e^{-A}\int_u d(e^{\varphi(\tau)} \alpha) \le e^{B-A}\ell(c_+) + e^{-A}(h(q_+)-h(q_-)),\\
      & \sup_{\rho(\tau)}\int_u \rho(\tau)d\tau\wedge \alpha \le e^{B-A}\ell(c_+)
    \end{align*}
    and thus
    $$  E_{[A,B]}(u') \le e^{B-A}\ell(c_+) + e^{-A}(h(q_+)-h(q_-))+e^{B-A}\ell(c_+),$$
    for any $J$-holomorphic strip $u'$ with its unique positive Reeb chord puncture at $c_+,$ one positive puncture at $q$, and possibly additional negative Reeb chord punctures.  For the computation of the above bound on the energy, we have used that the primitive of $e^\tau\alpha$ pulled back to $L$ can be taken to have compact support.

    In view of the above energy bound,and dependence of $(h(q_+)-h(q_-))$ on $\ell(c_\pm)$, we get the following crucial property: total energy of $u'$ can be assumed to be \emph{arbitrarily small}, after shrinking the interval of the isotopy used in the construction of the cobordism (in order for $\ell(c_\pm)$ to become arbitrarily small).


Following the proof of \cite[Lemma 5.1]{GDRSurgery} based upon the standard monotonicity property
 for pseudo-holomorphic curves in symplectic manifolds (see \cite[Proposition 4.3.1]{AudinLafontaine}), there is a constant $E_0$ such that if $u'$ intersects the subset $\{z=\pm 1\} \subset J^1B_1 $ (which is disjoint from $L$), then
$$E_{[A,B]}(u') \ge E_0.$$
Under the assumption that the cobordism has been constructed so that $\ell(c_\pm)>0$ is sufficiently small (while keeping the above coordinates, almost complex structure, and projection of $L$ to $T^*B_{1/3}$ fixed), we can hence conclude that $u'$ is disjoint from some fixed neighborhood of $\{|z|=1\}.$ Since
$$\partial U=\{ |z|=1 \} \: \cup \: \partial(B_{1} \times B_1 )\times(-1,1),$$
this means that the projected curve $v=\pi \circ u'|_{u'^{-1}(U)}$ has boundary $v|_{u'^{-1}(\partial U)}$ contained inside $\partial(B_{1} \times B_1 ) \subset T^*B_1 .$

In view of the above bound on the $|z|$--coordinate of $u'$, we conclude the following. If $v$ intersects the boundary $\partial(B_{1/3}  \times B_1 )$, then its symplectic area can be bounded from below by Sikorav's original monotonicity result \cite[Proposition 4.7.2]{AudinLafontaine}. More precisely, we apply this monotonicity to $J_{T^*B_1 }$-holomorphic curves inside in $(B_{1/3}  \times B_1 ,\omega_0)$ with boundary on the transversally intersecting Lagrangian planes
$$\pi(L \cap (B _{1/3} \times B _1 \times [-1,1])) \subset B _{1/3} \times B _1$$
(the Lagrangian planar property of the projection of $L$ is a consequence of Formula \eqref{eq:reebtrace}). We conclude that any such $v$ has $\omega_0$-area bounded from below by some constant $C> 0$.

The holomorphicity of the projection $\pi$, together with the fact that the two-forms $d(e^{\varphi(\tau)}\alpha)$ and $\rho(\tau)d\tau \wedge \alpha$ pull back to non-negative two-forms on any pseudo-holomorphic curve for a cylindrical almost complex structure (recall that $\varphi'(\tau),\rho(\tau)\ge 0$), implies the inequalities
$$ 0 \le \int_v e^Ad\alpha \le E_{[A,B]}(u).$$
In particular,
 we conclude that $E_{[A,B]}(u)$ has a lower bound in terms of the symplectic area of $v$ which, in turn, is bounded from below by $C>0$ in view of the aforementioned monotonicity argument. For $\ell(c_\pm)>0$ sufficiently small, the upper bound on the left-hand side implies that the image of $u$ must be contained entirely inside $B _{1/3} \times B _1 \times [-1,1]$.

Since the projected boundary condition $\pi(L \cap (B _{1/3} \times B _1 \times [-1,1]))$ consists of two transversely intersecting Lagrangian planes, Stokes' theorem can be applied to show that $v$ must have vanishing $\omega_0$-area. Hence the projection $v$ is constantly equal to $0 \in B_{1/3}  \times B_1 $. This means that $u'$ is contained inside the Reeb orbit that projects to the origin in the above coordinates. The sought claim concerning the uniqueness of $u^\pm$ follows from this.

The claim about the non-existence of discs with only negative asymptotics is a consequence of Stokes' Theorem.\end{proof}

For a pair of mixed chords $x,y$ and ordered (possibly empty) sets of pure chords $\mathbf{z},\mathbf{z}'$ define
$$\mathbf{m}_L(x^+, y^\pm) = \sum_{\mathbf{z}, \mathbf{z}'} \sharp \mathcal{M}^0(x^+\mathbf{z} y^\pm \mathbf{z}';L) \varepsilon_{-}(\mathbf{z} \mathbf{z}').$$
Lemma \ref{lem:SmallDisk} implies (up to a $q \mapsto \pm q$ basis change)
\begin{equation}
\label{eq:mcp}
\mathbf{m}_L((c_+)^+, q^+) = 1 = \mathbf{m}_L(q^+, (c_-)^-).
\end{equation}
Since a positive (resp. negative) puncture at $q$ means that the boundary of the disk makes a jump from (resp. to) the component $\R \times \Lambda$ at the puncture, we immediately get the vanishing result
$$\mathbf{m}_L(q^-, (c_-)^-) = 0 = \mathbf{m}_L((c_+)^+,q^-).$$

Choose $\delta$ such that
$$\max_{x_{10} \ne c_-} \ell(x_{10}) <  \delta < \ell(c_-) < 0 < \min_{x_{01} } \ell(x_{01}).$$
Let $RFC^{[\delta, b)}_*(\pm)$ and $RFC^{[a, \delta)}_*(\pm)$ be the two Rabinowitz Floer complexes for $\bar{\Lambda}_\pm $ with maps $d_{01}^\pm, d_{10}^\pm, B^\pm$ from Definition \ref{defn:MCC}.

Define the linear maps 
$${\phi}_{01}: RFC^{[\delta, b)}_*(+) \rightarrow RFC^{[\delta, b)}_*(-)$$ and
$${\phi}_{10}: RFC^{[a, \delta)}_*(-) \rightarrow RFC^{[a, \delta)}_*(+)$$
via the generators as follows.

\begin{eqnarray}
\label{eq:phi}
\phi_{01}(x_{01}) & = & \sum_{y_{01}} \mathbf{m}_L(x_{01}^+, y_{01}^-)y_{01}+ \mathbf{m}_L(x_{01}^+, q^{+}) c_{-} \\
\notag
\phi_{10}(x_{10})& = & \sum_{y_{10}} \mathbf{m}_L(x_{10}^-, y_{10}^+)y_{10}. 
\end{eqnarray}
Note that $\phi_{10}(c_-) =c_-$ and $c_-$ is not in the domain of either map.

Fix a small $\epsilon >0.$
By Lemma \ref{lem:ActionBound}, we can assume that the time-interval of the isotopy, $[-,+],$
is small enough such that the following holds.
For any pair of distinct chords $\{x,y\} \ne \{c_-, c_+\}$ and any chord $z \notin \{c_-, c_+\}$
$$   |\ell(x^-) - \ell(y^+)| > e^\epsilon \epsilon, \quad |\ell(z^-) - \ell(z^+)| < e^{-\epsilon} \epsilon, 
\quad |\ell(c_-) -\ell(c_+)| < e^{-\epsilon} \epsilon.$$

\begin{lem}
\label{lem:HomotopyInverse}
For $\epsilon$ sufficiently small and for all  $x_{01}, x_{10} \notin \{c_-, c_+\},$ 
 $\langle \phi_{01} x_{01}, x_{01} \rangle = k_{01}$ and  $\langle \phi_{10} x_{10}, x_{10} \rangle = k_{10}$ are units.
\end{lem}

\begin{proof}
The exact Lagrangian concordance $L_1$ shown in Figure \ref{fig:intersection} has an inverse cobordism $L_2$ constructed in the same manner, so that the concatenation $L_2 \odot L_1$ is (compactly supported) Hamiltonian isotopic to $\R \times \Lambda^+$. (One can either construct it by the front projection, or use the fact that the non-cylindrical component of $L_1$ can be realized as a Lagrangian trace cobordism as constructed by Proposition \ref{prop:trace}.)  Let $L^{0} = \R \times \bar{\Lambda}^+$ and let $L^{1}=L_2 \odot L_1$, such that $L^{1} \cap \{|\tau|> \tau_0\}$ agrees with $L^{0}.$
  
For $0 \le s \le 1,$ let $L^{s}$ be a  smooth family of exact Lagrangian concordances with only double points of arbitrarily small action,  such that $L^{s} \cap \{|\tau|> \tau_0\}$ agrees with $L^{0}.$ As in the last statement of  Lemma \ref{lem:SmallDisk} (proven by an application of Stokes' Theorem and the positivity of the total energy) we conclude that $\mathcal{M}^d(\Gamma, L^s) = \emptyset$ whenever $\Gamma$ has  a double point of $L^{s}$ (thus, of very small action), no positive Reeb chords, and no negative Reeb  at $c_+$.

 $\mathcal{M}^d(\Gamma, L^s) = \emptyset$ implies that if we ignore disks with punctures at double points of the concordance,  the trace concordance and inverse-trace concordance induce maps $\Phi, \Psi$ which satisfy (\ref{eq:homotopy}).
 In particular, if we set $x = x_{01}$ (resp. $x_{10}$) then we get $k_\Phi$ is a unit, arguing as in Case (1) in the proof of Proposition \ref{prop:Bifurcation} in Section \ref{sec:Geometry}.
 But $k_\Phi$ is also a count of the disks contributing to $\langle \phi_{01} x_{01}, x_{01} \rangle$
 (resp. $\langle \phi_{10} x_{10}, x_{10} \rangle$).

\end{proof}

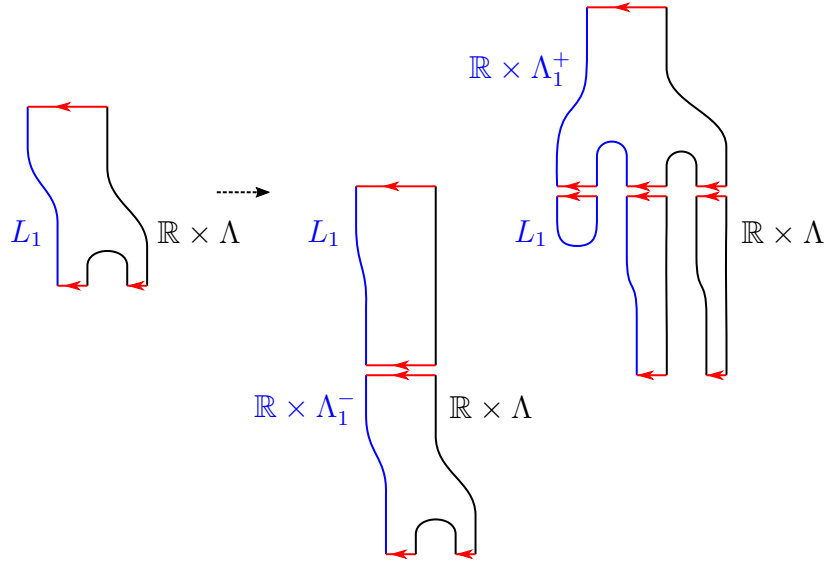
\begin{figure}[htp]
\vspace{5mm}
\centering
\labellist
\pinlabel $\color{blue}L_1$ at 0 123
\pinlabel $\color{blue}L_1$ at 112 123
\pinlabel $\color{blue}L_1$ at 190 123
\pinlabel $\color{blue}\R \times \Lambda_1^+$ at 186 185
\pinlabel $\R \times \Lambda$ at 285 123
\pinlabel $\R \times \Lambda$ at 65 123
\pinlabel $\color{blue}\R \times \Lambda_1^-$ at 105 57
\pinlabel $\R \times \Lambda$ at 175 57
\endlabellist
\includegraphics[scale=1]{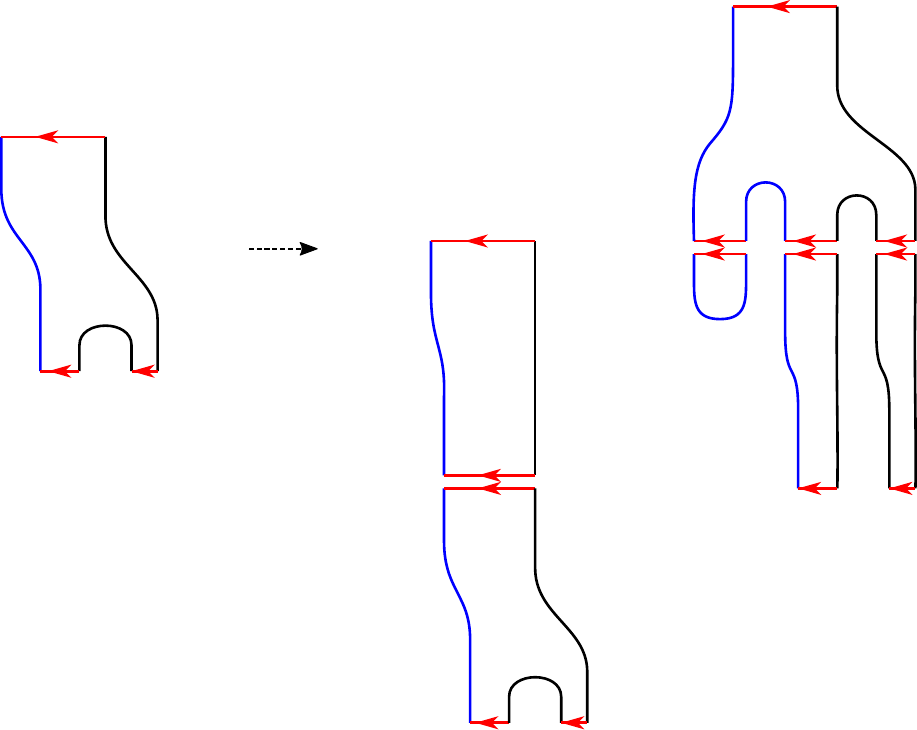}
\caption{Further examples of boundary points of type $(5 \times 2_\pm)$ for the compactification of the one-dimensional moduli space of pseudo-holomorphic disks with precisely one positive and precisely one negative mixed Reeb chord asymptotic, and boundary on $L_1 \cup (\R \times \Lambda).$ See Figure \ref{fig:breaking5}. The components in the middle level that have only pure Reeb chord asymptotics are responsible for pulling back the augmentations under the DGA-morphism induced by the concordance.}
\label{fig:breaking5.2}
\vspace{3mm}
\end{figure}

\begin{prop}
\label{prop:AxiomA1}
The maps $\phi_{01}, \phi_{10}$  satisfy Axiom \ref{A1} with  degree $\epsilon.$
\end{prop}

\begin{proof}
We first verify $\phi_{01} d_{01}^+ =  d_{01}^- \phi_{01}.$ 
\begin{eqnarray*}
\lefteqn{\langle (\phi_{01} d_{01}^+ -  d_{01}^- \phi_{01}) x_{01}, c_- \rangle = }\\
&=&\sum_{y_{01}}  \langle d_{01}^+ x_{01}, y_{01} \rangle\langle \phi_{01} y_{01}, c_- \rangle    -
 \langle \phi_{01} x_{01}, y_{01} \rangle \langle d_{01}^- y_{01}, c_- \rangle \\
&=&
 \sum_{y_{01}}  \mathbf{m}_{\bar{\Lambda}_+} (x_{01}^+, y_{01}^-)  \mathbf{m}_L(y_{01}^+, q^{+}) -   \mathbf{m}_L(x_{01}^+, y_{01}^- ) \mathbf{m}_L(q^{+}, (c_-)^- ) \mathbf{m}_{\bar{\Lambda}^-} (y_{01}^+, c_-^+) 
\end{eqnarray*}
The term $\mathbf{m}_L(q^{+}, (c_-)^- )=1$ is added to illustrate how the first and second sums are of type $(2_+ \times 6)$ and  $(2 \times 7 \times 1_-),$ respectively. 

The breaking of curves that involve only pure chords can be disregarded when the strips are counted with augmentations; see the analysis of the so-called ``$\delta$-breakings'' in \cite{Dimitroglou:Cthulhu}. This is due to two different mechanisms. First, a broken configuration such as $(5 \times 3_-)$ shown in Figure \ref{fig:breaking5}, i.e.~with a disk with only pure punctures in the bottom level (these discs define the differential of the Chekanov--Eliashberg algebra of $\overline{\Lambda}^-)$, are canceled algebraically when the count is weighted by the value of the augmentation. (Recall that augmentations vanish on boundaries of the Chekanov--Eliashberg differential by definition). Second, a disk with only pure punctures in the middle level, shown e.g.~in Figure \ref{fig:breaking5.2}, plays the role of pulling back the augmentation $\varepsilon_-$ under the DGA-morphism induced by the concordance $L$. In view of Lemma \ref{lem:PullBackAug} this pull-back is equal to $\varepsilon_+$, as sought.
For more details, see \cite[p.416, I and II]{Dimitroglou:Cthulhu}.

Since $(6 \times 3_-)$ corresponds to such an augmentation-related breaking,  Proposition \ref{prop:MasterEQ}(6) then implies the right hand side is 0.  The relevant boundary of the moduli space is shown in Figure \ref{fig:breaking6}.

For $z_{01} \ne c_-,$
\begin{eqnarray*}
  \lefteqn{\langle (\phi_{01} d_{01}^+ -  d_{01}^- \phi_{01}) x_{01}, z_{01} \rangle  =}\\
  & = & \sum_{y_{01}}  \mathbf{m}_{\bar{\Lambda}^+} (x_{01}^+, y_{01}^-)  \mathbf{m}_L(y_{01}^+, z_{01}^-)- \mathbf{m}_L (x_{01}^+, y_{01}^-)  \mathbf{m}_{\bar{\Lambda}^-}(y_{01}^+, z_{01}^-)
        \end{eqnarray*}
        which are the $(2_+ \times 5)$ and $(5 \times 2_-)$ terms in Proposition \ref{prop:MasterEQ}(5). See Figures \ref{fig:breaking5} and \ref{fig:breaking5.2} for illustrations.  There is no $(6 \times 7)$ term  since the component with a negative mixed Reeb chord asymptotic to $z_{01}^-$ must have a negative puncture at $q$, while $\mathbf{m}_L(q^{-}, z_{01}^-) = 0$ holds by Lemma \ref{lem:SmallDisk}. The   $(5 \times 3_-)$ terms are the augmented terms. Finally, 
$$
(\phi_{01} d_{01}^+ -  d_{01}^- \phi_{01})c_+ = \phi_{01}(0) - d_{01}^-c_- = 0 - 0.
$$

Verifying $\phi_{10} d_{10}^- =  d_{10}^+ \phi_{10}$ has only one computation:
\begin{eqnarray*}
  \lefteqn{\langle (\phi_{10} d_{10}^- -  d_{10}^+ \phi_{10}) x_{10}, z_{10} \rangle  = }\\
  & = &  \sum_{y_{10}} \mathbf{m}_L (y_{10}^+,x_{10}^-)  \mathbf{m}_{\bar{\Lambda}^-}(z_{10}^+,y_{10}^-) -  \mathbf{m}_{\bar{\Lambda}^+} (y_{10}^+,x_{10}^-)  \mathbf{m}_L(z_{10}^+,y_{10}^-)\end{eqnarray*}
which are all the terms in Proposition \ref{prop:MasterEQ}(5). Again the boundary is depicted in Figure \ref{fig:breaking5}, and we use Lemma \ref{lem:SmallDisk} to conclude that $\mathbf{m}_L(q^+,x_{10}^-)=0$ since $x_{10}^- \neq c_-$ holds by assumption.

For $x_{01} \ne c_+,$ Lemma \ref{lem:HomotopyInverse} implies $\langle \phi_{01} x_{01}, x_{01} \rangle  = k_{01}$ and $\langle \phi_{10} x_{10}, x_{10} \rangle = k_{10}$ are units. 
So to construct (strict, not just homotopy) inverses $\psi_{01}, \psi_{10}$ for $\phi_{01}, \phi_{10},$ it suffices to set
\begin{eqnarray*}
\psi_{01}(x_{01}) &= & k_{01}^{-1} x_{01} - \sum_{y_{01} \ne x_{01}} \langle \phi_{01}(x_{01}), y_{01} \rangle y_{01} -  \langle \phi_{01}(x_{01}), c_- \rangle c_+, \\
 \psi_{01}(c_-) & = & c_+,\\
\psi_{10}(x_{10}) & = & k_{10}^{-1} x_{10} - \sum_{y_{10} \ne x_{10}} \langle \phi_{10}(x_{10}), y_{10} \rangle y_{10}.
\end{eqnarray*}
It is easy to check that these are chain maps.

Stokes' theorem finally bounds all the maps' degrees by $\epsilon>0.$

%
\end{proof}

\begin{prop}
\label{prop:AxiomA2}
The maps $B^-$ and $B^+$ are chain maps which satisfy Axiom \ref{A2}. That is, $\phi_{10} B^+ \phi_{01}$ is homotopic to $B^-,$ where $B^\pm$ are of degree $\epsilon.$   
\end{prop}

\begin{proof}
To prove Axiom \ref{A2}, define (on generators, extend linearly) the map
$$H:  RFC^{[\delta, b)}_*(+)\rightarrow RFC^{[a, \delta)}_{*+1}(+),$$
which is defined by

\begin{equation*}
H \coloneqq H_\alpha + H_\beta,
  \end{equation*}
where
\begin{align}
\nonumber  & \langle H_\alpha w_{01}, v_{10} \rangle \coloneqq \mathbf{m}_L(w_{01}^+,  v_{10}^+)\\
\label{eq:Hbeta}   & \langle H_\beta w_{01}, v_{10} \rangle \coloneqq \mathbf{m}_L(v_{10}^+,  (c_-)^-)\langle \phi_{01}w_{01},c_{-}\rangle.
  \end{align}
The degree of $H$ is tautologically bounded from above by 0.
It suffices to show that
\begin{equation}
  \label{eq:htpy}
\langle (\phi_{01} B^- \phi_{01}  - B^+) x_{01}, y_{10} \rangle = 
\langle (d_{10}^+ H + H d_{01}^+)x_{01}, y_{10} \rangle.
\end{equation}
Apply Proposition \ref{prop:MasterEQ}(4) to analyze $\partial \mathcal{M}^1(x_{01}^+  \mathbf{z}  y_{10}^+ \mathbf{z'} ; L)$
for possibly empty words of pure cords  $\mathbf{z} ,\mathbf{z'}.$  The broken configurations that constitutes this boundary are shown in Figure \ref{fig:breaking4}. If the count these boundary strata weighted by augmentations, then the strata of type $(4\times 3_-)$ all cancel; for this reason, these configurations can be ignored.

The right-hand side of \eqref{eq:htpy} corresponds to the boundary strata of the following types:
\begin{itemize}
\item[($2_+ \times 4$):] These boundary points correspond to all terms in $d_{10}^+\circ H_\alpha-H_\alpha \circ d_{01}^+$;
 \item[$(5 \times 5 \times 1_-)$:] With the additional requirement that the component in $(1_-)$ has one positive mixed Reeb chord asymptotic to $c_-$, these boundary points correspond to the terms in $H_\beta \circ d_{01}^+$ (here we have used the chain map property of $\phi_{01}$ established in Proposition \ref{prop:AxiomA1} together with Equation \eqref{eq:Hbeta}).
 \end{itemize}
 However, we also have:
 \begin{itemize}
 \item[($\dagger$)] The broken configurations that correspond to $d_{10}^+\circ H_\beta$, which are of type ($2_+ \times 5$) (which, thus, are not boundary points of ($4$)).
 \end{itemize}
 We continue by analyzing the left-hand side of Equation \eqref{eq:htpy}.
 \begin{itemize}
 \item [($1_+$)]: These  components equal the $\langle  B^+ x_{01}, y_{10} \rangle$ term on the left hand side.
 \end{itemize}
By unraveling the definitions, the remaining term on the left-hand side is computed to be equal to
\begin{eqnarray*}
\phi_{10} B^- \phi_{01}(x_{01}) & = & \phi_{10} B^- \left(  \mathbf{m}_L(x_{01}^+, q^{+}) c_{-} +
\sum_{y_{01}} \mathbf{m}_L(x_{01}^+, y_{01}^-)y_{01}\right) \\
& =& 
\phi_{01} \left(\sum_{z_{10}}\left[
     \mathbf{m}_{\bar{\Lambda}^-}( z_{10}^+, (c_-)^-) \mathbf{m}_L(x_{01}^+, q^{+})  + \right.\right. \\
  & + & \left.
 \sum_{y_{01}} \left. \mathbf{m}_{\bar{\Lambda}^-}(y_{01}^+, z_{10}^+) \mathbf{m}_L(x_{01}^+, y_{01}^-)
\right] 
 z_{10} \right) \\
 & = &  \sum_{z_{10}, w_{10}} \left[
 \mathbf{m}_L(w_{10}^+, z_{10}^-) \mathbf{m}_{\bar{\Lambda}^-}( z_{10}^+, (c_-)^-) \mathbf{m}_L(x_{01}^+, q^{+}) \right.
 \\
 &&  +
 \sum_{y_{01}}  \left. \mathbf{m}_L(w_{10}^+, z_{10}^-) \mathbf{m}_{\bar{\Lambda}^-}(y_{01}^+, z_{10}^+) \mathbf{m}_L(x_{01}^+, y_{01}^-)
 \right]w_{10}.
\end{eqnarray*}
Here we find the remaining types of broken configurations that arise in the boundary $\partial \mathcal{M}^1(x_{01}^+  \mathbf{z}  y_{10}^+ \mathbf{z'} ; L)$:
\begin{itemize}
\item [($5 \times 5 \times 1_-$):] With the additional constraint that no Reeb chord asymptotic of the component $1_-$ is asymptotic to $c_-$, these types correspond to the second term on the right hand side.
\end{itemize}
What remains is the term ($6\times6$). For this we need to analyze the term
$$\mathbf{m}_L(w_{10}^+, z_{10}^-) \mathbf{m}_{\bar{\Lambda}^-}( z_{10}^+, (c_-)^-) \mathbf{m}_L(x_{01}^+, q^{+})$$
in the expression $\phi_{10} B^- \phi_{01}(x_{01})$ further. We start by using Proposition \ref{prop:MasterEQ}(5) to rewrite this as
$$(\mathbf{m}_{\bar{\Lambda}^+}(w_{10}^+, z_{10}^-) \mathbf{m}_{L}( z_{10}^+, (c_-)^-)+\mathbf{m}_{L}(w_{10}^+,q^-)\mathbf{m}_{L}(q^+,(c_-)^-) \mathbf{m}_L(x_{01}^+, q^{+}).$$
These broken strata are shown in Figure \ref{fig:breaking5}. Since $\mathbf{m}_{L}(q^+,(c_-)^-)=1$ holds by Lemma \ref{lem:SmallDisk}, the previous expression can be simplified even further. We have thus found
\begin{itemize}
\item[($6 \times 6$):] This corresponds to the second term in the latter expression, i.e.~the term $\mathbf{m}_{L}(w_{10}^+,q^-)\mathbf{m}_L(x_{01}^+, q^{+})$.
\end{itemize}
The first term $\mathbf{m}_{\bar{\Lambda}^+}(w_{10}^+, z_{10}^-) \mathbf{m}_{L}( z_{10}^+, (c_-)^-)\mathbf{m}_L(x_{01}^+, q^{+})$ of the latter expression remains to be taken care of, since it is not inside boundary $\partial \mathcal{M}^1(x_{01}^+  \mathbf{z}  y_{10}^+ \mathbf{z'} ; L)$.  However, these terms cancel with the remaining contribution on the right-hand side of Equation \eqref{eq:htpy}, since
\begin{itemize}
\item[($\dagger$)] The first term
  $$\mathbf{m}_{\bar{\Lambda}^+}(w_{10}^+, z_{10}^-) \mathbf{m}_{L}( z_{10}^+, (c_-)^-)\mathbf{m}_L(x_{01}^+, q^{+})$$
  is equal to $d_{10}^+\circ H_\beta$ by Formula \eqref{eq:Hbeta} (it is not a part of the boundary of the moduli space ($4$)).
  \end{itemize}
We have thus established Equation \eqref{eq:htpy} as sought.
\end{proof}

\subsubsection{No double points}
\label{sec:nodoublepoints}

Simply consider the summands $C_*(\pm)=RFC_*^{[0, b)}(\pm)$ and $C^*(\pm)=RFC_*^{(a, 0]}(\pm).$ We define $\phi_{01}, \phi_{10}$ as in equation (\ref{eq:phi}), noting that $\mathbf{m}_L(x_{01}^+, q^-) = 0$ since there is no double point $q.$ 
The rest of the arguments are  essentially  simplifications of the ones above, since curves of type ($6$) and ($7$) do not exist.
 Note that bifurcations can occur that were not considered in Section \ref{sec:onedoublepoint}. Notably, a pair of mixed chords may be born or die, making one of the above pair of $RFC$-complexes non-isomorphic. 
But even in the event of a birth (death is a reverse-time birth), the statement and proof of
Lemma \ref{lem:HomotopyInverse} still apply.
To recap our argument: since there are no double points, we can apply Proposition \ref{prop:Bifurcation} to replace the DGA-morphism induced by the Lagrangian concordance with a stable-tame isomorphism (STI) of DGAs.
Lemma \ref{lem:PullBackAug} equates how the pull-back of the augmentation induced by the concordance is the same as the change in augmentation induced by the STI in \cite[Lemma 3.4]{DRS2}. 
Thus, as outlined shortly after stating Theorem \ref{thm:pwc}, we can apply the techniques of  \cite[Section 3]{DRS2}, to both complexes $C_*(\pm)$ and $C^*(\pm),$ to prove Theorem \ref{thm:pwc}.

%
%
%

\section{Proofs of Theorem  \ref{thm:Main}, Theorem \ref{thm:NonDegen} and Theorem \ref{thm:Interlink}}
\label{sec:Proofs}

We need the following variation of \cite[Lemma 3.1]{DRS2}, which allows us to estimate the oscillation of a contact Hamiltonian based on the change of lengths of a pair of Reeb chords.

\begin{lem}
  \label{lem:main}
  
Consider a smooth one-parameter family $c(t) \subset (Y,\alpha)$ of Reeb chords with boundary on $a(t) \in \Lambda_0,  b(t) \in \Lambda_1(t),$  where $\Lambda_0 \subset Y$ is a fixed Legendrian submanifold and $\Lambda_1(t) \subset Y$ is a Legendrian isotopy. (Here $a$ is either the endpoint or the starting point, and $c(t)$ is allowed to be of zero length, i.e.~a double point $\Lambda_0 \cap \Lambda_1(t)$.) Then
\begin{equation}
\label{eq:contham}
\frac{d}{dt}\ell (c(t)) = \alpha(X_{b(t)}(t))=H_t(b(t))
\end{equation}
where $X_{b(t)} \in TY$ is the contact vector field that generates $\Lambda_1(t)$. In particular, if $c(t)$ and $d(t)$ are two Reeb chords as above, then
$$|(\ell (c(0)) - \ell(d(0)) - (\ell (c(1))-\ell(d(1)))| \le \|H_t\|_{\OP{osc}}$$
is satisfied.

Consider a smooth one-parameter family $c_1(t)$ of pure Reeb chords with initial and terminal endpoints at $a_1(t), b_1(t) \in  \Lambda_1(t).$
Then
$$ \frac{d}{dt}\ell(c_1(t))=\alpha(X_t(b_1(t)))-\alpha(X_t(a_1(t)))=H_t(b_1(t))-H_t(a_1(t)).$$
Hence, the inequality $|\ell(c_1(1))-\ell(c_1(0))| \le \|H_t\|_{\OP{osc}}$ holds.

\end{lem}

\begin{proof}
  The calculation for the pure chords $c_1$ was proven in \cite[Lemma 3.1]{DRS2} for general contact manifolds, so we only need to verify the computation for $c.$

 When $c(t)$ has positive length, we may perform the computation for a contact vector field $X$ that can be taken to vanish along all of $\Lambda_0$ after cutting off the contact Hamiltonian in some neighborhood. The computation of $\frac{d}{dt}\ell (c(t))$ is then a direct application of 
  \cite[Lemma 3.1]{DRS2}.

 In the case when $c(t)$ is of length zero, we replace $\Lambda_0$ by its image $\phi^\epsilon_{R_\alpha}(\Lambda_0)$ under the time-$\epsilon$ Reeb flow, and instead compute $\frac{d}{dt}\ell (\tilde{c}(t))$ for chord between $\phi^\epsilon_{R_\alpha}(\Lambda_0)$ and $\Lambda_1(t)$. 
Here $\tilde{c}(t)$ corresponds to $c(t)$ under the natural 
bijective correspondence of Reeb chords between, on one hand, $\Lambda_0$ and $\Lambda_1(t)$ and, on the other hand, $\phi^\epsilon_{R_\alpha}(\Lambda_0)$ and $\Lambda_1(t)$. Under this bijection, 
gotten by ``removing" the first time-$\epsilon$ portion, 
the chord lengths differ by the constant $\epsilon$.

%
%
\end{proof}

\subsection{Proof of Main Theorem  \ref{thm:Main}}

Recall the notation from Theorem  \ref{thm:Main}.
A Legendrian $\Lambda \subset (Y, \alpha)$ is moved by a contact isotopy $\phi^t_{\alpha,H},$ with $0 \le t \le 1.$ The Hamiltonian $H_t \colon Y \to \R$ satisfies  $\|H_t\|_{\OP{osc}} < \min\left\{l, \ell(c_k)\right\}.$ 
Here $0 < l \le \infty$ is chosen such that there exists an augmentation $\varepsilon: \mathcal{A}^l(\Lambda) \rightarrow \kk,$ and $c_k$ is the Reeb chord with the $k$-th shortest length.

Let $1\gg\epsilon>0$ be a constant to be determined.
Let $\Lambda_0 = \Lambda_0^t = \Lambda$ and $\Lambda_1^t=\phi^t_{\alpha,H}(\Lambda_0')$ where $\Lambda_0'$ is a perturbation of a push-off (of flow $\epsilon$ in the positive Reeb direction) of $\Lambda_0.$ Let $\bar{\Lambda}^t = \Lambda_0^t \cup \Lambda_1^t.$

First set $-a_0 = 2\epsilon = b_0$ and $l(t)=\int_0^t \max H_\tau-\min H_\tau d\tau.$
We claim that $RFC_*^{[a_t,b_t)}(\Lambda_0,\Lambda_1^0;\varepsilon_0)$ (notation as in Definition \ref{defn:MCC}) is quasi-isomorphic to the Morse complex of $\Lambda_0.$
This can be seen since $C^*(a_0, 0]=0$ and $d_{01}^0$ counts only holomorphic strips which approximate the gradient flows the function which represents $\Lambda_1^0$ graphical in a small $J^1\Lambda_0^0$ neighborhood of $\Lambda_0^0$ for a suitable cylindrical almost complex structure; see e.g.~\cite{GDRLift}.

We wish to adapt our situation in order to replicate as much as possible \cite[Section 3.4]{DRS2}, which proves the special case of Main Theorem \ref{thm:Main} when $(Y = P \times\R_z, \alpha = dz-\theta).$
In that case, due to the global $\partial_z$ Reeb flow, we could choose a large $N\gg1$ push off instead of a small $\epsilon \ll1$ one.
The complex considered in the previous set-up is somewhat simpler since there never are any $x_{10}$ chords throughout the isotopy (or chords of 0-length).
The chords at $t=0$ in the action window $[N-2\epsilon, N+2\epsilon]$ form a subcomplex which is quasi-isomorphic to the Morse complex of $\Lambda_0.$

The action of chords ($\approx N$ versus $\approx 0$) which generate a Morse complex, and the generators of the complex (type $x_{01}$ versus type $x_{01}, x_{10}$) are the only distinctions between the existing argument for Main Theorem \ref{thm:Main} when $(Y = P \times\R_z, \alpha = dz-\theta)$ presented in  \cite[Section 3.4]{DRS2}, and the needed argument for Main Theorem \ref{thm:Main} when $(Y,\alpha)$ is arbitrary. 

We briefly sketch this argument, deferring full details to \cite[Section 3.4]{DRS2}.
Label the Morse generators $x_1, \ldots, x_m.$
For each family $x_i(t)$ consider families of two barcodes which are approximately  based on an action windows
$[\ell(x_i(t)) - (l-l(t))+\epsilon, {{\ell(x_i(t))+\epsilon}}]$ 
and
$[\ell(x_i(t)) - \epsilon, \ell(x_i(t))+ (l-l(t)) -\epsilon]$
{{as specified in \cite[Equations (3.6) and (3.7)]{DRS2}.
\cite[Lemma 3.7]{DRS2} implies that in
}}
each action window we see 
{{an infinite}}
 bar that starts at $\ell(x_i(t)).$
Theorem \ref{thm:pwc} implies we can use the Barcode Proposition \ref{prop:Barcode}
{{
for this isotopy.
By looking at both action windows simultaneously during this isotopy,
}}
we verify that the bar can only disappear in the amount of time it takes $\ell(x_i(t))$ to coincide with one of the values of some other mixed chord that was not part of the original Morse complex. 
The rate of change of the difference of these action values is controlled by $\frac{d}{dt} l(t)$ as calculated in
{{Lemma \ref{lem:main}, which brings in the $\|H_t\|_{\OP{osc}}$ term as stated in the Main Theorem \ref{thm:Main}.
}}
This sketch glosses over the difference between the endpoints of the bars in a barcode, and their representations as Reeb chords. However, this important distinction has been addressed 
{{in case (ii) in \cite[Section 3.4]{DRS2}, and the argument there is
}} independent of the contact manifold and form.

\subsection{Proof of Theorem \ref{thm:NonDegen}}

Rosen and Zhang prove that, for any submanifold $N \subset Y$ of $\dim(N) =n,$ the distance $\delta_\alpha$ is either non-degenerate or $\delta_\alpha \equiv 0$ \cite[Theorem 1.9]{RosenZhang}.
So it suffices to find $\Lambda_1, \Lambda_2$ which are the images of contact isotopies of $\Lambda,$ and for which $\delta_\alpha(\Lambda_1, \Lambda_2) > 0.$ 

Consider any closed Legendrian $ \Lambda$ and fix a contact form $\alpha$ on $Y$. Let $\Lambda_i$, $i=1,2$, be given by $j^1(\epsilon^2\cdot i\cdot f)$ in a standard contact-form preserving jet-neighborhood of the form
$$D_{\le r}T^*\Lambda \times [-5\epsilon,5\epsilon]_z \subset (J^1\Lambda,dz-pdq)$$
inside $Y$ which identifies $\Lambda$ with $j^10$, where $f \colon \Lambda \to [0,1]$ is a Morse function, and $0<\epsilon \ll 1$. Let
$$0<\min_{f}=m_1<\ldots<m_k=\max_{f}$$
be an enumeration of the critical values of $f$ that we, moreover, assume correspond to distinct critical points. After post-composing the Morse function with a suitable change of coordinates, we may assume 
\begin{equation}
  \label{eq:spectrum}
  2\min(f) > \max(f)>0.
\end{equation}

Using the notation of Definition \ref{defn:MCC} for the Rabinowitz Floer complex with action window
$$ RFC_*^{[a_t,b_t)}(\Lambda,\Lambda_t;\varepsilon_t)$$
we choose the following: the action window is constantly equal to $[a_t, b_t) = [-2\epsilon, +2\epsilon);$ the initial action threshold, $l=6\epsilon,$ is sufficiently smaller than the length of any contractible periodic Reeb orbit $\gamma$ of degree $|\gamma|\le 1$; 
by the existence of the above standard neighborhood, all Reeb chords of $\Lambda$ and $\Lambda_i$ that start and end on the same component have length at least $7\epsilon$, so $\varepsilon_0$ is necessarily trivial;
and $\Lambda_t$ is an isotopy from $\Lambda_1$ to $\Lambda_2$ with oscillation $l(1) \le \epsilon.$ 

If we had studied a Legendrian isotopy from $\Lambda_1$ to $\Lambda_2$ generated by a Hamiltonian $H_t$ such that $l(1)$ no longer satisfies $l-l(1) > b_1-a_1$ (so that our technology breaks down), we automatically get a desired lower bound
$$
\int_0^1 \max |H_t| dt \ge \frac{1}{2}\|H_t\|_{osc}  > \frac{1}{2}l(1) 
{{
\ge \frac{1}{2}
(l-(b_1 -a_1)) = \frac{1}{2}(6\epsilon-4\epsilon)=\epsilon.
}}
$$
So assume 
 we can apply Theorem \ref{thm:pwc}, and thus 
 the isotopy deforms the barcode of $RFC_*^{[-2\epsilon, +2\epsilon)}(\Lambda,\Lambda_1;\varepsilon_1)$ to that of $RFC_*^{[-2\epsilon, +2\epsilon)}(\Lambda,\Lambda_2;\varepsilon_2)$ via the bifurcations specified by the Barcode Proposition \ref{prop:Barcode}.

 During this isotopy, one of the following scenarios occur. Case (1): some starting/endpoint of a bar in the barcode of $RFC_*^{[-2\epsilon, +2\epsilon)}(\Lambda,\Lambda_1;\varepsilon_1)$ survives, and by {{inequality}}
\eqref{eq:spectrum}, moves at least a distance $\epsilon^2(2\min(f)-\max(f)) > 0$.  Case (2): some bar in the barcode dies or escapes the action window. Note that the concerned bars all are of length at least $\min\{\epsilon^2(m_2-m_1),\ldots,\epsilon^2(m_k-m_{k-1})\}$ by the assumption made on distinct critical values. (In fact, $RFC_*^{[-2\epsilon, +2\epsilon)}(\Lambda,\Lambda_i;\varepsilon_i)$ for $i=1,2$ are the Morse complexes that compute the singular homology $H_*(\Lambda;\Z_2).$ So even more can be said about the barcode, but we do not use this.) 
Theorem \ref{thm:pwc} together with the Barcode Proposition \ref{prop:Barcode} then implies that
$$\int_0^1 \max |H_t| \ge C$$
holds for this Hamiltonian that generates a Legendrian isotopy from $\Lambda_1$ to $\Lambda_2$, where
$$ C \coloneqq \min\left\{\epsilon^2 \left(2\min(f)-\max(f)\right),\epsilon^2(m_2-m_1),\ldots,\epsilon^2(m_k-m_{k-1})\right\}>0 $$
is a constant that only depends on the contact form $\alpha$ and the Legendrians $\Lambda, \Lambda_1,\Lambda_2$. Here we use  Lemma \ref{lem:main}  to relate change in Reeb chord length and the value of the contact Hamiltonian $H_t$ at the endpoint of the Reeb chord of $\Lambda_t$ that corresponds to the moving bar.

We conclude that $\delta_\alpha({{\Lambda_1, \Lambda_2}}) > C$
holds as sought.

 \qed

\subsection{Proof of Theorem \ref{thm:Interlink}}
Let $\Lambda_0^t = \phi_{\alpha, H_t}^t(\Lambda_0),$ with $H_t \ge c >0$, $\Lambda_1^t = \Lambda_1$ and 
$\bar{\Lambda}^t = \Lambda_0^t \cup \Lambda_1^t.$
Recall the assumptions that $\Lambda_0, \Lambda_1$ are augmented
 and that no pseudo-holomorphic plane appears in the SFT-compactifications from \cite{BEHWZ}  (see Remarks \ref{rem:ClosedOrbits} and \ref{rem:ClosedOrbits2}). Thus we can set $l = +\infty$
 and consider the sequence of complexes $RFC_*(\Lambda_0^t,\Lambda_1)$, whose barcode in any finite action window is well-defined for all $t \ge 0$  and varies continuously as in Theorem \ref{thm:pwc}. 
 Since each  finite  bar is a pairing of two mixed chords of relative grading 1,  the hypothesis implies the barcode at  $t=0$  either  has
 an infinite bar with a starting point of positive action,  or a finite bar with starting point of negative action, and endpoint of positive action.
 The interlinkedness property then follows by the continuity of the barcode, together with Lemma \ref{lem:main}.  Namely, according to the latter, all endpoints of bars moves in the direction of decreasing action, with a speed bounded from below by $c>0.$ (Here we use a large enough action window for the given Hamiltonian.)

\section{Computations for the standard Legendrian $\RP^n$ (Proof of Theorem \ref{thm:RP})}
\label{sec:RP}

This section concerns computations of the Rabinowitz Floer complex for the standard Legendrian $\RP^n \subset \RP^{2n+1}$, with the goal of proving Theorem \ref{thm:RP}.

Recall the definition
$$ \Lambda_0 \coloneqq S^{2n+1} \cap \Re \C^{n+1} \subset (S^{2n+1},\xi_{st})$$
of the standard Legendrian sphere inside the standard contact sphere, as defined in Section \ref{sec:Introduction}. By taking the quotient under the antipodal map we obtain the standard Legendrian embedding
$$\tilde\Lambda_0 \coloneqq \Lambda_0/\Z_2 \subset S^{2n+1}/\Z_2 = \RP^{2n+1}$$
of $\RP^{n}$ into the standard contact projective space.

\begin{rem} It is possible to also pass to further quotients $S^{2n+1}/\Z_k$ with $k \ge 2$ in order to obtain Legendrian embeddings in the standard contact (higher dimensional analogues of) ``Lens spaces".  Most of the analysis carried out in this section should be possible to apply with only minor modifications in order to derive similar results also for general $k \ge 2$. However, we do not pursue this direction further.
  \end{rem}

Similarly, any linear Lagrangian subspace $V^{n+1} \subset \C^{n+1}$ gives rise to a Legendrian embedding
$$ \Lambda_V \coloneqq S^{2n+1} \cap V \subset (S^{2n+1},\xi_{st})$$
of $S^{n}$, and hence a corresponding Legendrian embedding
$$\tilde\Lambda_V\coloneqq \Lambda_V / \Z_2 \subset (\RP^{2n+1},\xi_{st})$$
of $\RP^{n}$. (Note that $\tilde\Lambda_0=\tilde\Lambda_{\Re \C^{n+1}}$.) All of these different Legendrian embeddings are clearly Legendrian isotopic via an ambient contact isotopy that preserves the round contact form. In particular, we obtain a $C^\infty$-small Legendrian push-off of $\tilde\Lambda_0$ by considering $\tilde\Lambda_V$ for a Lagrangian plane $V \subset \C^{n+1}$ which is sufficiently close to $\Re\C^{n+1}$.

Recall that the round contact $S^{2n+1}$ is equipped with the contact form $\alpha_{st}=\frac{1}{2}\sum_i (x_idy_i-y_idx_i)$, and a time-$t$ Reeb flow given by complex scalar multiplication by $e^{i2t}$. Hence, $S^{2n+1}$  is foliated by simple closed Reeb orbits, all of whose periods are equal to $\pi$. It follows that the round $\RP^{2n+1}$ is foliated by simple closed Reeb orbits of period $k\pi/2$ for each $k=1,2,3,\ldots$. The orbits of period $k\pi/2$  form a manifold $\Gamma_{k} \cong \CP^n$ of Reeb chords which are non-degenerate in the Bott sense; see the work \cite{Bourgeois} by Bourgeois. Note that these Reeb orbits are contractible if and only if $k=2m$.

The Reeb chords on $\tilde\Lambda_V$ all come in connected families $\mathcal{Q}(\tilde\Lambda_V)_k^{\OP{Bott}} \cong \RP^{n}$ labeled by the Reeb chord lengths $k\pi/2$, $k=1,2,3,\ldots$. These families are also smooth manifolds which are non-degenerate in the Bott sense.

In Sections \ref{sec:cz} and \ref{sec:relcz} below we will compute the Conley--Zehnder indices of these orbits and chords. In addition, we compute the Conley--Zehnder indices of the chords and orbits after a generic perturbation by a Morse function on the Bott manifolds as described in \cite{Bourgeois}, which makes the contact form non-degenerate. The conclusion of Proposition \ref{prop:planedegree} in Section \ref{sec:cz} is that the Chekanov--Eliashberg DGAs of the standard Legendrian $\RP^n \subset \RP^{2n+1}$ is well-defined for the aforementioned non-degenerate perturbations of the round contact form. In Proposition \ref{prop:halfplanedegree} from Section \ref{sec:relcz} the degree of the Reeb chords on $\RP^n$ are computed. This is useful for showing that $\RP^n$ admits augmentations, which of course is crucial for defining the Rabinowitz Floer complex in Proposition \ref{prop:RCF}.

\subsection{Conley--Zehnder index of a periodic Reeb orbit}
\label{sec:cz}

Here we perform Conley--Zehnder index computations for the periodic Reeb orbits on the round $\RP^{2n+1}$, as well as its non-degenerate perturbations. See Appendix \ref{sec:gradingperiodic} for the definition of the Conley--Zehnder index.

Recall that precisely the even covers of the simple orbits are contractible. In addition, since the universal cover of $\RP^{2n+1}$ is  $S^{2n+1}$, it follows that $\pi_2(\RP^{2n+1})=0$, and any two planes with the same asymptotic Reeb orbit are homotopic through planes of the same kind. 
 Denote by $\pi \colon \RP^{2n+1} \to \CP^{n}$ the prequantum bundle projection, for which $D\pi|_\xi$ is injective.
For any contractible Reeb orbit $\gamma \in \Gamma_{2m},$ take the trivialization of $\xi$ along the Reeb orbit $\gamma$ that is induced by pulling back a symplectic frame at the point $\pi(\gamma) \in \CP^n$ under the bundle projection $\pi.$

We will apply the index formula \eqref{eq:index} to prove the following.

\begin{lem}
\label{lem:relchern}
Consider a plane $u \colon \C \to \R \times \RP^{2n+1}$ which is asymptotic to $\gamma \in \Gamma_{2m}.$
  The relative first Chern class with respect to the above choice of trivialization satisfies
  $$c^\xi_{1,\OP{rel}}[u]=c^{\CP^{n}}_1[\pi \circ u]=m(n+1), \:\: m \ge 1,$$
  where $c^{\CP^{n}}_1$ is the usual first Chern class and the detailed definition of $c^\xi_{1,\OP{rel}}$ is in Appendix \ref{sec:gradingperiodic} .

  Here we identify $\pi \circ u$ with a sphere that is homologous to $m\cdot L$, where $L \in H_2(\CP^n)$ denotes the homology class of a line.
\end{lem}
\begin{proof}
By Stokes' theorem the symplectic area of the chain $\pi\circ u$ in $\CP^n$ is equal to the length of the periodic Reeb orbit.
Here we have endowed $\CP^n=\RP^{2n+1}/S^1$ with the symplectic form induced by the curvature $d\alpha$ of the prequantization form $\alpha$ on $\RP^{2n+1}$ via symplectic reduction. (With these conventions, the area of a line in $\CP^n$ is equal to $\pi$.)  Since $\CP^n$ is monotone, that is the area and index are proportional,  the second equality follows. The first equality then follows since
  $$D\pi|_\xi \colon \xi \to T\CP^n$$
  is a symplectic bundle morphism which is an isomorphism on the fibers.
  \end{proof}

  \begin{lem}
    The Conley--Zehnder index of $\gamma \in \Gamma_{2m}$ with respect to the above choice of trivialization is equal to
  \begin{equation}
\label{eq:cz}
  \mu_{\OP{CZ}}(\mathbf{A}_\gamma-\delta\cdot\id)=n\end{equation}
independently of the multiplicity $2m$.

After a perturbation by a Morse function on the Bott manifold as in \cite{Bourgeois}, the non-degenerate orbit that corresponds to a critical point has Conley--Zehnder index
$$ \mu_{\OP{CZ}}(\mathbf{A}_\gamma-\delta\cdot\id)+i-2n=i-n$$
where $i \in \left\{0,1,\ldots,\dim \Gamma_{2m}=2n\right\}$ is the Morse index of the critical point.
\end{lem}
\begin{proof}The computation in the Bott setting was performed in \cite[(3.10)]{Wendl10}. Alternatively, one can argue as follows. The linearized Reeb flow in this trivialization is constantly equal to the identity map. Hence, the aforementioned perturbation of the Reeb flow by a small positive rotation $s \mapsto e^{i\delta s}\id$ in the contact planes, perturbs the non-degenerate return map to a non-degenerate one. We then compute the classical Conley--Zehnder index of this path, as defined in \cite{RobbinSalamon}, which gives $\frac{1}{2}(2n)=n$ as sought.

The Conley--Zehnder index after the perturbation by a Morse function on the Bott manifold follows from Formula \eqref{eq:czmorse}. Recall that $\dim \Gamma_{2m}=2n$ is independent of $m$.  
\end{proof}
In conclusion we have shown the following.
\begin{prop}
  \label{prop:planedegree}
For a plane $u \colon \C \to \R \times \RP^{2n+1}$ which is asymptotic to a Reeb orbit in the family $\Gamma_{2m}$, $m \ge 1$, for the round contact form, we have the identity
$$\OP{index}(u)=((n+1)-3)+n+2m(n+1)=(2m+2)(n+1)-4, \:\: m \ge 1,$$
for the expected dimension of the moduli space of unparameterized pseudo-holomorphic planes of the same type (with asymptotics that are free to vary in the Bott family $\Gamma_{2m}$).

Moreover, after a small perturbation of the contact form by a Morse function defined on the Bott manifolds, as constructed in \cite{Bourgeois}, all periodic Reeb orbits may be assumed to be non-degenerate, and to satisfy the bound
$$|\gamma|\ge 4(n+1)-4-2n=2n$$
on their degrees.
\end{prop}
\begin{proof}
The result follows from a direct application of the index Formula \eqref{eq:index} combined with the above computations of the relative first Chern class and Conley--Zehnder indices. 
\end{proof}
  
\subsection{Conley--Zehnder index of a pure Reeb chord}
\label{sec:relcz}

The next step is to compute the Conley--Zehnder indices for the pure Reeb chords $\mathcal{Q}(\tilde\Lambda_V)_k$ on $\Lambda_V$, both for the round contact form and after a non-degenerate perturbation. The definition of the Conley--Zehnder index for Reeb chords will be recalled in Appendix \ref{sec:puregrading}.

First, note that the Reeb chords on $\RP^n \subset \RP^{2n+1}$ are all null-homotopic when considered as elements in $\pi_1(\RP^{2n+1},\RP^n)$ when $n \ge 1$. This follows since any Reeb chord lifts to a Reeb chord on the standard Legendrian sphere under the universal cover
$$(S^{2n+1},S^{2n+1} \cap \mathfrak{Re} \C^{n+1}) \to (\RP^{2n+1},\RP^{n})$$
where both the map and the restriction $S^{2n+1} \cap \mathfrak{Re} \C^{n+1} \cong S^{n} \to \RP^{n}$ are universal covers.
Moreover, any element in $\pi_2(\RP^{2n+1},\RP^n)$ lifts to an element in $\pi_2(S^{2n+1},S^n)$ under this map, where the latter group vanishes whenever $n>1$. Hence the Maslov class automatically vanishes on $\pi_2(\RP^{2n+1},\RP^n)$ when $n\ge 2$. This is also the case for $n=1$ by a standard calculation.

Contractibility of a Reeb chord is equivalent to the existence of a continuous half-plane
$$u \colon (\HH,\partial \HH) \to (\R \times \RP^{2n+1},\R \times \RP^{n})$$
with boundary condition on the cylinder over the Legendrian $\RP^{n}$ and puncture asymptotic to the Reeb chord. It follows similarly to the case of planes discussed above that two half-planes that share the same asymptotics are homotopic through half-planes of the same type when $n\ge2$; indeed, $\pi_2(S^{2n+1},S^{n})\cong\pi_1(S^{n})=0$ whenever $n \ge 2$. In the case $n=1$, $\pi_2(S^{3},S^{1})\cong \pi_1(S^1)=\Z$ and there are infinitely many homotopy classes of planes asymptotic to any given family of Reeb orbits. The implication is the following.

\begin{lem}
Any Reeb chord $c$ is the asymptotic of a half-plane, and for any two pseudo-holomorphic half-planes $u_1,u_2$ asymptotic to $c,$ $\OP{index}(u_1) = \OP{index}(u_2).$
Here $\OP{index}(u)$ denotes the Fredholm index of a linearization of the $\overline{\partial_J}$-at $u,$ viewed as an element of a standard Banach space of maps.
\end{lem}

The Reeb chords on the Legendrians $\tilde\Lambda_V$ for the round contact form come in Bott families $\mathcal{Q}(\tilde\Lambda_V)^{\OP{Bott}}_k \cong \RP^{n}$, where the images of these chords coincide with $k$-fold multiples of the simply covered periodic Reeb orbits for $k=1,2,3,\ldots$. The index formula for a pseudo-holomorphic half-plane inside $\R \times \RP^{2n+1}$ with boundary on $\tilde\Lambda_V$ and asymptotics to the chord $c \in \mathcal{Q}(\tilde\Lambda_V)^{\OP{Bott}}_k$ (without a constraint at a fixed asymptotic orbit) is given by Formula \eqref{eq:relindex} in Appendix \ref{sec:puregrading}. Namely,
$$ \OP{index}(u)=(\OP{CZ}(c)-1)+\mu_{\R \times \RP^{n}}[u]$$
where $\OP{CZ}(c)$ is the Conley-Zehnder index of the path of Legendrian planes along $c$
(see Section \ref{sec:relcz}).

In order to compute the Conley--Zehnder index and the relative Maslov class we need to make the choice of a capping path (up to homotopy) as described in Appendix \ref{sec:puregrading}. Since the Reeb flow is totally periodic, we will simply choose the constant capping path.

\begin{lem} For the constant capping path, we have the identity
$$\mu_{\R \times \RP^{n}}[u]=\mu_{\RP^{n}}(\pi \circ u)=k(n+1), \:\: k \ge 1,$$
where $\mu_{\RP^{n}}$ is the classical Maslov index for a disc with boundary on the Lagrangian $\RP^n \subset \CP^n$.

Here we identify  $\pi \circ u$ with a disk in $(\CP^n,\RP^n)$ that is homologous to $kD$, where $D \in H_2(\CP^n,\RP^n)$ is the homology class represented by either of the two hemispheres in a complex line $\CP^1 \hookrightarrow \CP^n$ that intersects $\RP^n$ in an equator.
\end{lem}
\begin{proof}
The claim about the homology class of $\pi \circ u$ follows an area consideration similar to the proof of Lemma \ref{lem:relchern}.
  
The first equality is immediate from the choice of capping path, together with the fact that $D\pi|_\xi \colon \xi \to T\CP^N$ is a bundle morphism that is a symplectic isomorphism on the fibers.
  
 The Maslov index computation
$$\mu_{\RP^{n}}(\pi \circ u)=k(n+1)$$
is well-known. (Note that the symplectic area of the projected disc in $(\CP^n,\RP^n)$ is equal to $k\pi/2$, where $\pi$ is the symplectic area of a line.)
\end{proof}

\begin{lem}
  With the above choice of constant capping paths, the Conley--Zehnder index satisfies
\begin{equation}
  \label{eq:relcz}
  \OP{CZ}(c)=n
\end{equation}
for any $c \in \mathcal{Q}(\tilde\Lambda_V)_k^{\OP{Bott}}$ independently of $k=1,2,3,\ldots$.

After a perturbation by a Morse function on the Bott manifold as in \cite{Bourgeois}, the non-degenerate Reeb chord that corresponds to a critical point has Conley--Zehnder index
$$ \OP{CZ}(c)+i-n=i$$
where $i \in \left\{0,1,\ldots,\dim  \mathcal{Q}(\tilde\Lambda_V)_k^{\OP{Bott}}=n\right\}$ is the Morse index of the critical point.
\end{lem}
\begin{proof}
  This is similar to the computation of Equation \eqref{eq:cz} in the periodic case. More precisely, the Reeb flow on $\RP^{2n+1}$ is totally periodic, and the starting point and end point of all Reeb chords on $\RP^n$ coincide. The return map of the Reeb flow is the identity, and we make it non-degenerate by performing a small positive rotation $s \mapsto e^{i\delta s}\id$ in the contact planes. Finally, the result is obtained by computing the standard Conley--Zehnder index of this non-degenerate path.
    \end{proof}

\begin{prop}
  \label{prop:halfplanedegree}
For a half-plane $u$ as above with boundary on $\R \times \tilde\Lambda_V$, and boundary puncture asymptotic to a Reeb orbit in the family $\mathcal{Q}(\tilde\Lambda_V)_k^{\OP{Bott}}$, $k \ge 1$, for the round contact form, we have the identity
$$\OP{index}(u)=(n-1)+k(n+1)=(1+k)(n+1)-2, \:\: k \ge 1,$$
for the expected dimension of the moduli space of unparameterized pseudo-holomorphic half-planes of the same type (with asymptotics that are free to vary in the Bott family $Q(\tilde\Lambda_V)_k^{\OP{Bott}}$).

Moreover, after a small perturbation of the contact form by a Morse function defined on the Bott manifolds as in \cite{Bourgeois}, all Reeb chords with boundary on $\tilde\Lambda_V$ may be assumed to be non-degenerate, and to satisfy the bound
$$|c|\ge 2(n+1)-2-n=n$$
on their degrees.
\end{prop}
\begin{proof}    
It suffices to use the index Formula \eqref{eq:relindex} with the computations from the above lemmas.
\end{proof}

\subsection{Computing the DGA and the Rabinowitz complex}

First we show that the Chekanov--Eliashberg algebra of $\RP^n \subset \RP^{2n+1}$ has a (canonically defined) augmentation. Of course, here it is important that the Chekanov--Eliashberg algebra is well-defined in the first place, i.e., no-Gromov bubbling not captured in the algebra. This non-bubbling follows from Proposition \ref{prop:planedegree} above, which bounds the degree of a contractible periodic Reeb orbits from below by $2n$.
\begin{prop}
\label{prop:aug}
Consider a small Morse perturbation of the round contact form to a non-degenerate one as described above. The DGA of $\tilde\Lambda_V$ has an augmentation in $\Z_2$ that sends all Reeb chord generators to $0$.
\end{prop}
\begin{proof}
In the case 
$n>{1}$ 
this follows directly from the degree computation in Proposition \ref{prop:halfplanedegree}; since all chords have degree strictly greater than one, the differential of the DGA has no constant terms.

In the case $n=1$ we cannot argue merely by considerations of the degree. In this case there is a single Reeb chord $c$ in degree $|c|=1$, while all other chords have degree at least $2$. There are thus two possibilities when coefficients in $\Z_2$ are used: either 
$\partial c=1,$
and there are no augmentations, or 
$\partial c=0,$ 
and there is a unique ``trivial" augmentation that sends all chords to zero. We will see that the latter case holds.

Since $\RP^3=UT^*S^2
=\{(q,p)\in T^*S^2\,\,|\,\,\|p\|=1\}
$ 
and $\tilde\Lambda_V$ is Legendrian isotopic to the conormal lift 
$\{(q_0,p) \,\,|\,\,\|p\|=1\}$ 
of a point 
$
q_0 \in
S^2$, 
the Legendrian $\tilde\Lambda_V$ 
admits an exact Lagrangian filling isotopic to
$\{(q_0, p)\} \subset T^*S^2$
inside the exact symplectic filling 
$T^*S^2$  
of the contact manifold $UT^*S^2 = \RP^{2n+1}$
when $n=1.$ 
Hence, in this case, the Chekanov--Eliashberg algebra admits augmentations that are geometrically induced by the fillings by the functorial properties of SFT proven in \cite{Ekholm}. 
(Recall that by functoriality, a Lagrangian filling induces a DGA-morphism from the Chekanov--Eliashberg algebra of the Legendrian at the positive end to the empty-set Legendrian, whose DGA is the ground ring with trivial differential. Such a morphism is by definition an augmentation.)
In view of the above degree consideration, the existence of the augmentation implies that $\partial$ has no constant terms. Hence, there is a (trivial) augmentation.
\end{proof}

The mixed Reeb chords that start on $\tilde\Lambda_0$ and end on $\tilde\Lambda_V$ and which are of length $t \ge 0$ can be parametrized by their starting points
$$(\Lambda_0 \cap e^{-i2t}\cdot V)/ \Z_2 \subset \tilde\Lambda_0.$$
The mixed chords from $\tilde\Lambda_0$ to $\tilde\Lambda_V$ are non-degenerate if and only if all K\"{a}hler angles between the two subspaces are pairwise distinct. 
In any case, as for the pure chords, any mixed chord $\gamma$ can be concatenated by any closed Reeb orbit of period $k\pi/2$, in order to form a new mixed Reeb chord of length $\ell(\gamma)+k\pi/2$.

We continue to direct our attention to the following particular family of Lagrangian subspaces
$$V_s \coloneqq \langle e^{is\pi/(n+1)-s\epsilon}\cdot\partial_{x_1},e^{is2\pi/(n+1)-s\epsilon}\cdot\partial_{x_2},\ldots,e^{is(n+1)\pi/(n+1)-s\epsilon}\cdot\partial_{x_{n+1}}\rangle_\R \subset \C^{n+1},$$
for $\epsilon>0$ sufficiently small, and the induced one-parameter family $\tilde\Lambda_s \coloneqq \tilde\Lambda_{V_s} \subset \RP^{2n+1}$ of Legendrians.

\begin{lem}
\label{lem:mixeddegrees}
For a suitable choice of Maslov potentials, the complex $RFC_*(\tilde\Lambda_0,\tilde\Lambda_1)$ is generated by the mixed chords $c^k_j$ with $k \in \Z$ and $j \in \{1,2,\ldots,n,n+1\}$ that all are non-degenerate, and whose gradings are given by $|c^k_j|=j+k(n+1)-1$. In addition
\begin{itemize}
\item  The chords $c^k_j$ with $k=0,1,2,\ldots$ start on $\tilde\Lambda_0$, end on $\tilde\Lambda_1$, and are of the form
\begin{gather*}
c^k_j(t) \coloneqq e^{i2t}\cdot \partial_{x_j},\:\:t \in [0,(1/2)(\pi(j/(n+1)+k)-\epsilon)],\\
j=1,\ldots,n+1, \: k=0,1,2,3,\ldots.
\end{gather*}
\item  The chords $c^k_j$ with $k=-1,-2,-3,\ldots$ start on $\tilde\Lambda_1$, end on $\tilde\Lambda_0$, and are of the form
\begin{gather*}
c^k_j(t) \coloneqq e^{i2t}\cdot e^{i(j\pi/(n+1)-\epsilon)}\partial_{x_j},\:\:t \in [0,(1/2)(\pi(-j/(n+1)-k)+\epsilon)],\\
j=1,\ldots,n+1, \: k=-1,-2,-3,\ldots.
\end{gather*}
\item Their actions satisfy
$$\mathfrak{a}(c^{k_1}_{j_1}) < \mathfrak{a}(c^{k_2}_{j_2}) $$
if and only if $(k_1,j_1) < (k_2,j_2)$ with respect to the lexicographic order.
\end{itemize}
\end{lem}
\begin{proof}
Consider the family $\tilde\Lambda_s$ of Legendrian push-offs. It can be seen that $\tilde\Lambda_\delta$
for $\delta>0$ sufficiently small is obtained by the perturbation of the image of $\tilde\Lambda_0$ under a small positive Reeb flow by a perfect Morse function $\RP^{n} \to \R$. Recall that such a Morse function has precisely $n+1$ non-degenerate critical points, one in each degree $0,1,2,\ldots,n$.

We proceed to perform the computation of the grading of the mixed chords, as defined in Appendix \ref{sec:maslovpotential}. First, note that the chords $c^k_j$ correspond to the critical point of the above Morse function with Morse index $j-1$. For a suitable choice of Maslov potentials, a standard computation thus gives us $|c^0_j|=j-1$; see e.g.~\cite[Lemma 2.9(2)]{GDRFamilies}.

The chords $c^k_j$ with $k > 0$ have the same start and endpoints as $c^0_j$. By Section \ref{sec:maslovpotential} the difference in Maslov potential can be computed by the relative Chern number $c^{\xi}_{1,\OP{rel}}[u]$ of a plane $u \colon \C \to \R \times \RP^{2n+1}$ whose asymptotic is the $k$-fold cover of the simple periodic Reeb orbit that contains $c^0_j$, relative the trivialization that is constant under the periodic Reeb flow. We compute this index to be
$$c^{\xi}_{1,\OP{rel}}[u]=c^{\CP^{n}}_1[\pi \circ u]=k(n+1),$$
where the right hand side is the first Chern class in $\CP^n$ of the $k$-fold multiple of the generator of $H_1(\CP^n)$. The computation of $|c^k_j|$ for all $k\ge 0$ follows from this.

We leave the the computation of the degree of the chords $c^k_j$ with $k< 0$ to the reader, since it follows by analogous computations (although the order of the Legendrians have been switched).

To deduce the statement for $\tilde\Lambda_1$ from the statement for $\tilde\Lambda_\delta$, it suffices to use the continuity of the degrees of the chords. This, in turn, holds since all chords remain transverse as the parameter $s  \in [\delta,1]$ varies.
\end{proof}

\begin{lem}
  \label{lem:actionshift}
There exists a \underline{contact-form preserving} isotopy $\phi^t$ of the round projective space $(\RP^{2n+1},\alpha_{st})$ for which
\begin{itemize}
\item $\phi^t$ acts on $\tilde\Lambda_0$ by the Reeb flow and a reparameterization, more precisely $\phi^t(\tilde\Lambda_0)=e^{t\pi/(n+1)}\tilde\Lambda_0$ for all $t \in [0,1]$;
\item in addition, $\phi^1$ fixes $\tilde\Lambda_1$ set-wise, i.e.~$\phi^1(\tilde\Lambda_1)=\tilde\Lambda_1$;
\item the mixed Reeb chords with one endpoint on $\tilde\Lambda_0$ and one endpoint on $\phi^t(\tilde\Lambda_1)$ remain non-degenerate for all $t \in [0,1]$, in particular there are no birth/deaths of Reeb chords during this isotopy (note that mixed chords can shrink to a zero-length chord and change direction in the path, which corresponds to the moments when $\tilde\Lambda_0 \cap \tilde\Lambda_1 \neq \emptyset$); and
\item the path of non-degenerate mixed Reeb chords parametrized by $t \in [0,1]$ that is induced by the above isotopy connects the chord $c^k_j$ (at $t=0$) with the chord $c^k_{j+1}$ if $j<n+1$ and $c^{k+1}_1$ if $j=n+1$ (at $t=1$).
\end{itemize}
\end{lem}
\begin{proof}
It suffices to consider the action on $\C^n$ of a path of real matrices in $SO(n+1) \subset SU(n+1) \subset Gl(n+1,\C)$ that starts with the identity matrix, and ends at a matrix in $SO(n+1)$ that represents a linear map of the form
$$ \partial_{x_j} \mapsto \pm \partial_{x_{j+1}}, \:\: j=1,2,\ldots,n+1,$$
on the standard basis of $\C^{n+1}$ (here we write $\partial_{x_{n+2}}=\partial_{x_1}$). In other words, the latter matrix performs as permutation of the real coordinate lines.
The new Legendrian now corresponds to the subspace
\begin{eqnarray*}
  \lefteqn{\C^{n+1} \supset V' \coloneqq} \\
& &
    \langle \pm e^{i(n+1)\pi/(n+1)-s\epsilon}\cdot\partial_{x_1},\pm e^{i\pi/(n+1)-s\epsilon}\cdot\partial_{x_2},\ldots,\pm e^{in\pi/(n+1)-s\epsilon}\cdot\partial_{x_{n+1}}\rangle_\R
    \end{eqnarray*}
Note that this contact isotopy fixes $\tilde\Lambda_0$ set-wise for all $t \in [0,1]$, but that the time-1 map does not fix $\tilde\Lambda_1$. It is finally a simple matter to apply the time-$\pi/(2(n+1))$ Reeb flow in order for the image of $\tilde\Lambda_1$ to become fixed set-wise under the time-1 map.

This proves the first two bullet points, while the third and fourth follow from the first two.
\end{proof}

\begin{prop}
\label{prop:RCF}
There exist arbitrarily small perturbations of  the round contact form on $\RP^{2n+1}$, $n\ge 1$, to a non-degenerate contact form, for which the minimal degree of a contractible Reeb orbit is $|\gamma| \ge 2n$. Moreover,
\begin{enumerate}
\item the Chekanov--Eliashberg algebra the standard Legendrian $\RP^n$, which is well-defined and invariant under Legendrian isotopy by the above, admits an augmentation; and
\item the Rabinowitz Floer complex $RFC_*(\tilde\Lambda_0,\tilde\Lambda_1)$  for the above pair of standard $\RP^n$'s, which thus also is well-defined and invariant under Legendrian isotopy, can be assumed to have underlying graded vector space of the form
$$RFC_*(\tilde\Lambda_0,\tilde\Lambda_1)=\bigoplus_{i \in \Z} \Z_2 \cdot c_i=\bigoplus_{i \in \Z} \Z_2[i],$$
where $|c_i|=i \in \Z$ and $\mathfrak{a}(c_i) < \mathfrak{a}(c_{i+1})$, and with a differential that vanishes.
\end{enumerate}
\end{prop}
\begin{proof}
We perturb the contact form by a Morse function defined on the Bott manifold of Reeb orbits as in \cite{Bourgeois}. Proposition \ref{prop:planedegree} implies the sought bound on the degree of the periodic Reeb orbits. The well-definedness claimed in Part (1) is a consequence of this bound. The augmentation of the Chekanov--Eliashberg algebra then follows from Proposition \ref{prop:aug}.

For Part (2), the form of the graded vector space that underlies the complex can be seen by considering Lemma \ref{lem:mixeddegrees}. Recall that the mixed chords described in that lemma already are transversely cut out, and may be assumed to be unaffected by the small perturbation of the round contact form by the Morse functions. The same is also true for the path of mixed Reeb chords in the isotopy produced by Lemma \ref{lem:actionshift}.

Consider the barcode of the filtered complex
$$RFC_*(\tilde\Lambda_0,\tilde\Lambda_1)$$
with coefficients in $\Z_2$. The claim that the differential vanishes is equivalent to the claim that there are no finite bars in this barcode for any finite action window.

We argue by contradiction and assume that there exists a finite bar. By degree considerations, this finite bar must start at some Reeb chord $c_i$ end at $c_{i+1}$ for some $i \in \Z$. 

Consider the family of barcodes
that corresponds to the family of filtered complexes
$$RFC_*(\tilde\Lambda_0,\phi^t(\tilde\Lambda_1)), \:\: t \in [0,1],$$
obtained from the isotopy produced by Lemma \ref{lem:actionshift}.
Here we can impose a finite action window that includes $\ell(c_i), \ell(c_{i+1}),$ and is much larger than the oscillation of the Hamiltonian which generates $\phi^t.$ 
This enables us to apply Theorem \ref{thm:pwc} to get a PWC.
For this family of pairs of Legendrians,  no births/deaths of mixed Reeb occur; the Reeb chords  remain transverse, even if their lengths can shrink to zero and change direction. Since the barcode varies continuously, we deduce that no bar can (dis)appear during the deformation.

Lemma \ref{lem:actionshift} moreover implies that the generators of $RFC_*(\tilde\Lambda_0,\phi^t(\tilde\Lambda_1))$ vary continuously for $t \in [0,1]$, and connect $c_i$ to $c_{i+1}$. In particular, since $c_i$ is the start of a finite bar, we conclude that there must be a finite bar in the barcode of
$$RFC_*(\tilde\Lambda_0,\phi^1(\tilde\Lambda_1))=RFC_*(\tilde\Lambda_0,\tilde\Lambda_1)$$
that starts at $c_{i+1}$ as well. It is however impossible for a single chord to correspond to both an endpoint and a starting point of a bar, which is the sought contradiction. 
\end{proof}

\appendix

\section{Gradings and indices for Reeb chords and orbits}
\label{sec:GradingsAndIndices}
Here we recall some generalities about index formulas of pseudo-holomorphic curves and Conley--Zehnder indices for a $2n+1$-dimensional contact manifold $(Y,\alpha)$ with a choice of contact form which is non-degenerate in the Bott sense. We consider both cases of periodic Reeb orbits and Reeb chords on Legendrians. We also consider Maslov potentials for pairs of Legendrians and induced gradings of mixed Reeb chords. None of these results are new, but since they can be difficult to extract from the literature, we give a brief but systematic treatment of relevant definitions and basic results.

For simplicity we assume that the first Chern class of the contact distribution $\xi \to Y$ vanishes. The Reeb chords and orbits of $\alpha$ are assumed to be non-degenerate in the Bott sense; see \cite{Bourgeois}. (In particular, this also includes the case when the Reeb chords and orbits are non-degenerate in the usual sense.)

\subsection{Grading and Conley--Zehnder index for periodic Reeb orbits}
\label{sec:gradingperiodic}

The grading $|\gamma|$ of a contractible period Reeb orbit $\gamma \in C^\infty(S^1,Y)$ that lives in a Bott manifold $\Gamma$ is defined as the expected dimension
$$ |\gamma|=\OP{vdim}\mathcal{M}(\Gamma) $$
of the moduli space that consists of unparameterized finite-energy pseudo-holomorphic planes
$$ u\colon \C \to \R_\tau \times Y$$
that are
\begin{itemize}
\item asymptotic to some Reeb orbit in the (possibly zero-dimensional) Bott family $\Gamma \ni \gamma$ (which is free to vary); and
\item pseudo-holomorphic for an almost complex structure $J$ that is compatible with $d(e^\tau\alpha)$ and cylindrical outside of a compact subset.
  \end{itemize}
  Note that, here we do not identify two solutions that differ by a translation of the symplectization coordinate $\tau \in \R$. 
\begin{rem}
Both the contractibility of $\gamma$ and the vanishing of the first Chern class on spherical classes are crucial for the well-definedness of the grading.
\end{rem}
One can compute the expected dimension from topological data of the map $u$. Namely,
$$ |\gamma| \coloneqq \OP{index}(u),$$
This Fredholm index $\OP{index}(u)$ has been computed; see \cite{Bourgeois},  \cite[Proposition 3.7]{Wendl10} or \cite{DiogoLisi}. Here we use the formulation from \cite[Proposition 3.7]{Wendl10} applied to the special case of a pseudo-holomorphic plane.
\begin{equation}
  \label{eq:index}
  \OP{index}(u)=((n+1)-3)+ \mu_{\OP{CZ}}(\mathbf{A}_\gamma-\delta\cdot\id)+2c^\xi_{1,\OP{rel}}[u],
\end{equation}
where $\delta>0$ is a sufficiently small number.
Strictly speaking, we will not define the individual terms in the expression $\mathbf{A}_\gamma-\delta\cdot\id,$ although see Remark \ref{rem:CZnotation}.
\begin{rem} The quantity $n+1$ in our formula corresponds to $n$ in \cite{Wendl10}. In addition, unlike \cite[Proposition 3.7]{Wendl10}, we consider the moduli space of \emph{unparameterized} planes; one thus has to add $\dim_\R \OP{Aut}(\C)=4$ to \eqref{eq:index} in order to get the cited formula.
  \end{rem}
What remains is now to explain the quantities in the formula. First we need to choose a symplectic trivialization of the contact planes $\xi$ along $\gamma$ (up to homotopy).
\begin{itemize}
\item \emph{The relative first Chern class $c^\xi_{1,\OP{rel}}[u]$} is the algebraic number of zeros of a section of the determinant line $u^*\det\xi$ line that is constant in the trivialization of $\gamma^*\det{\xi} \to S^1$ chosen along the Reeb orbit.
\item \emph{The Conley--Zehnder index $\mu_{\OP{CZ}}(\mathbf{A}_\gamma-\delta\cdot\id)$ in the non-degenerate case} can be computed as a Maslov index of a suitable (non-closed) path of Lagrangians, as shown in \cite[Remark 5.4]{RobbinSalamon}. The path of Lagrangians is obtained by taking the graphs of the symplectic path induced by the linearized Reeb flow (in the chosen trivialization of $\gamma^*\xi\to S^1$); this is a path of symplectic matrices that start with the identity matrix and ends with a symplectic matrix whose eigenvalues are all different from one (by the non-degeneracy assumption);
\item \emph{The Conley--Zehnder index $\mu_{\OP{CZ}}(\mathbf{A}_\gamma-\delta\cdot\id)$ in the degenerate case}, i.e.~when the linearized time-one map of the Reeb flow along $\gamma$ has an eigenvalue equal to one, is computed as follows. Deform the path of symplectic matrices induced by the linearized Reeb flow by adding a small positive rotation $s \mapsto e^{i\delta s}\id$, $\delta>0$, in the contact plane. Then compute the above Conley--Zehnder index for this deformed path.
\end{itemize}
\begin{rem}
\label{rem:CZnotation}
  The notation $\mu_{\OP{CZ}}(\mathbf{A}_\gamma-\delta\cdot\id)$ is motivated by the non-degenerate case (second bullet point above), where the Conley--Zehnder index also can be computed by first perturbing the asymptotic operator $\mathbf{A}_\gamma$ for the linearized $\overline{\partial}_J$-equation by adding a small negative multiple of the identity, and then computing the corresponding classical Conley--Zehnder index.  We refer to \cite[Section 3.2]{Wendl10} for a description of the asymptotic operator $\mathbf{A}_\gamma$, which has a discrete spectrum, and which is injective if and only if the periodic Reeb orbit $\gamma$ is non-degenerate.
  \end{rem}

  As described in \cite{Bourgeois}, for a contact manifold $(Y,\alpha)$ with a Reeb flow that is non-degenerate in the Bott sense, one can use a small Morse function $f \colon \Gamma \to \R$ defined on the Bott manifolds in order to perturb the contact form to one whose Reeb flow is non-degenerate, while still keeping control of the periodic orbits. More precisely, the Reeb orbits obtained by such a perturbation coincides with the critical points of the Morse function $f$, where the Reeb orbit corresponding to the critical point $p$ has Conley--Zehnder index
  \begin{equation}
    \label{eq:czmorse}\mu_{\OP{CZ}}(\mathbf{A}_\gamma-\delta\cdot\id)+\OP{index}(p)-\dim \Gamma,
    \end{equation}
where $\OP{index}(p)$ denotes the Morse index of the critical point $p$ of $f$.
\begin{rem} In particular, the Conley--Zehnder in the degenerate case coincides with the Conley--Zehnder index of the non-degenerate orbit at the maximum of the Morse function that arises from the aforementioned perturbation.
  \end{rem}

  \subsection{Grading and Conley--Zehnder index for pure Reeb chords which bound planes}
  \label{sec:puregrading}
In addition to the vanishing of the first Chern class of $\xi \to Y$, we now also assume that the Maslov class of $\R \times \Lambda$ vanishes. In this case, the grading $|c|$ of a contractible pure Reeb chord $c \in C^\infty([0,\ell],\{0,\ell\}),(Y,\Lambda))$ that lives in a Bott manifold $\mathcal{Q}$, can be defined as the expected dimension
$$ |c|=\OP{vdim}\mathcal{M}(c).$$
Here $\mathcal{M}(c)$ denotes the moduli space that consists of unparameterized finite-energy pseudo-holomorphic half-planes
$$ u\colon (\C \cap \{y \ge 0\},\{y=0\}) \to (\R_\tau \times Y,\R \times \Lambda)$$
that are
\begin{itemize}
\item asymptotic to some Reeb chord in the (possibly zero-dimensional) Bott family $\mathcal{Q} \ni c$ (which is free to vary); and
\item pseudo-holomorphic for an almost complex structure $J$ that is compatible with $d(e^\tau\alpha)$ and cylindrical outside of a compact subset.
  \end{itemize}
  Note that, here we do not identify two solutions that differ by a translation of the symplectization coordinate $\tau \in \R$. 
\begin{rem}
Again, for the well-definedness of the above grading, both the contractibility of $\gamma$ as well as the vanishing of the Maslov class for disks are crucial properties.
\end{rem}
As before we can compute the expected dimension from topological data of the map $u$ by
$$ |c| \coloneqq \OP{index}(u),$$
where the Fredholm index can be expressed as
\begin{equation}
  \label{eq:relindex}
  \OP{index}(u)=(\OP{CZ}(c)-1)+\mu_{\R \times \Lambda}[u]
  \end{equation}
We proceed to describe the quantities in the above formula. First we need to choose a continuous capping path of Lagrangian tangent planes along $c(t)$ that connects $T_{c(0)}\Lambda \subset \xi_{c(0)}$ to $T_{c(\ell)}\Lambda \subset \xi_{c(\ell)}$ (up to homotopy).
\begin{itemize}
\item \emph{The relative Maslov class $\mu_{\R \times  \Lambda}[u]$} of the half-plane $u$ which is defined by concatenating the path of Lagrangian tangent planes along $u|_{\{y=0\}}$ with the capping path at the puncture to obtain a closed path of Lagrangian tangent planes, and then computing the usual Maslov index for this loop of Lagrangian tangent planes in the trivialization induced by $u$.
\item \emph{The Conley--Zehnder index $\OP{CZ}(c)$ in the non-degenerate case} is defined as follows. First, we use the Reeb flow to identify the contact planes along the Reeb chord. Construct a closed loop of Lagrangian tangent planes by rotating $T_{c(0)}\Lambda$ to $T_{c(1)}\Lambda$ in the contact plane by using the minimal positive K\"{a}hler angles (these are non-zero by the non-degeneracy assumption). Then concatenate this path with the capping path to obtain a loop of Lagrangian tangent planes. The Maslov index of this loop is the Conley--Zehnder index.
  \item \emph{The Conley--Zehnder index $\OP{CZ}(c)$ in the degenerate case} is computed as above, but where we first perturb the Lagrangian tangent plane $T_{c(0)}\Lambda$ at the starting point by a small positive rotation $e^{i\delta}\id$ in the contact plane.
  \end{itemize}

  Similarly to the perturbation of Bott manifolds of periodic Reeb orbits by Morse functions as constructed in \cite{Bourgeois}, one can also perform a perturbation of the Bott manifolds of Reeb chords. We again obtain an analogous formula
  \begin{equation}
    \label{eq:relczmorse}
    \OP{CZ}(c)+\OP{index}(p)-\dim \mathcal{Q}
    \end{equation}
for the Conley-Zehnder index of the non-degenerate Reeb orbit that corresponds to the critical point $p$ of the Morse function $f \colon \mathcal{Q} \to \R$. Here $\OP{index}(p)$ denotes the Morse index of the critical point $p$ of $f$.
\begin{rem} As in the periodic orbit case, the Conley--Zehnder in the degenerate case coincides with the Conley--Zehnder index of the non-degenerate orbit at the maximum after such a perturbation.
  \end{rem}
  
\subsection{Grading and Conley--Zehnder index for mixed Reeb chords}
\label{sec:maslovpotential}

In this subsection we assume that the mixed Reeb chords are all non-degenerate in the strong sense.

The grading of a mixed Reeb chord with endpoints on two different Legendrians $\Lambda_0$ and $\Lambda_1$ depends on several additional choices. First, we need to choose a symplectic trivialization of the square of the determinant $\C$-line bundle $\left(\det_\C \xi\right)^{\otimes_\C 2} \to Y$ (up to homotopy). This is possible since the first Chern class vanishes (in fact one only needs it to be two-torsion). Note that, since $T\Lambda_i \subset \xi$ is Lagrangian, its image $[T_{c(0)}\Lambda_i] \subset \left(\det_\C \xi_{c(0)}\right)^{\otimes_\C 2} \cong \C$ inside the determinant line is a one-dimensional real subspace. Second, we need to make continuous choices of lifts to $\R$ of the angular phase in $\R/\pi\Z$, i.e.~the argument, of the images
$$[T\Lambda_i] \coloneqq \left(\det_\R\Lambda_i\right)^{\otimes_\R 2} \subset \left(\det_\C\xi\right)^{\otimes_\C 2} \to Y$$
of these real subspaces in the latter $\C$-line bundle. (Passing to the square means that this operation is well-defined on unoriented Legendrian tangent spaces.) The choice of such a lift is called a {\bf Maslov potential}, and it exists if and only if the Maslov classes of $\Lambda_i$ vanish. See \cite[Section 2.5]{GDRFamilies} or \cite{Dimitroglou:Cthulhu} for more details about Maslov potentials in the Legendrian setting.

We can define the grading of a mixed Reeb chord $c \colon [0,\ell] \to Y$ from $\Lambda_i$ to $\Lambda_j$ as follows. Let $\overline\phi_i \in \R$ be the lifts of phases of $[T\Lambda_i] \subset \left(\det_\C \xi\right)^{\otimes_\C 2}$ at the endpoints of the Reeb chord $c$ as defined by the choice of Maslov potential. Use the Reeb flow $\phi^t_R \colon (Y,\alpha) \to (Y,\alpha)$ to identify $T_{c(0)}\Lambda_i$ with a Lagrangian tangent plane $\phi^\ell_R(T_{c(0)}\Lambda_i)\subset \xi_{c(\ell)}$. By continuity (of course using the triviality of $\left(\det_\C\xi_{c(1)}\right)^{\otimes_\C 2} \to Y$) we obtain a lift $\overline\phi_0'$ of the phase of $[\phi^\ell_R(T_{c(0)}\Lambda_i)] \subset \left(\det_\C \xi_{c(\ell)}\right)^{\otimes_\C 2}.$ Perform the smallest positive rotation $e^{i\theta_0}$, $\theta_0 \in (0,\pi)$ that makes the real line
$$[\phi^\ell_R(T_{c(0)}\Lambda_i)] \subset \left(\det_\C \xi_{c(\ell)}\right)^{\otimes_\C 2}\cong\C$$
coincide with $[\phi^\ell_R(T_{c(\ell)}\Lambda_i)]$. (The non-degeneracy implies that this angle is non-zero). The Conley--Zehnder index is then the difference
$$ \OP{CZ}(c) = \frac{1}{\pi}\left(\overline\phi_0'+\theta_0-\overline\phi_1\right) \in \Z$$
of lifts of phases, and we define the grading via
$$ |c| \coloneqq \OP{CZ}(c)-1 \in \Z.$$
Recall that this grading depends on the chosen symplectic trivialization of the $\C$-line bundle $\left(\det_\C\xi\right)^{\otimes_\C 2} \to Y$ as well as the choice of Maslov potentials of the Legendrians involved.

A feature of this grading is that an unparameterized pseudo-holomorphic strip in $\R_\tau \times Y$ that has a positive asymptotic to the mixed Reeb chord $c^+$ and a negative Reeb chord asymptotic to the mixed Reeb chord $c^-$, lives in a moduli space of expected dimension
$$ \OP{vdim}\mathcal{M}(c^+,c^-)=|c^+|-|c^-|.$$
Note that, for this dimension, we again do not identified solutions that differ by a translation of the $\tau$-coordinate.

We end by noting that Legendrian isotopies naturally induce continuous extension of the Maslov potential. In the case of a loop $\Lambda^t$ of Legendrians, the effect on the Maslov potential at a point $p \in \Lambda^0=\Lambda^1$ can be seen to be as follows. Take a smooth path $\gamma(t) \in \Lambda^t$ so that $\gamma(0)=p=\gamma(1)$, and consider the trivialization of $\gamma^*\left(\det_\C\xi\right)^{\otimes_\C 2}$ induced by the real lines $[T_{\gamma(t)}\Lambda^t]$. The action on this isotopy of the Maslov potential at $p$ can then be seen to be equal to
$$c_{1,\OP{rel}}^\xi[u] \in \Z.$$
Here the relative first Chern class computes the algebraic number of zeros of a smooth extension of the section of $u^*\left(\det_\C\xi\right)^{\otimes_\C 2}$ along a smooth orientable compact surface $u \colon \Sigma \to Y$ whose boundary parametrizes $\gamma(t)$, where we require the section to be non-zero and constant with respect to the aforementioned trivialization along the boundary $\gamma(t)$.

\section{Length of trace cobordisms and conformal factors}

\label{sec:trace}

It is well-known that the trace of a Legendrian isotopy can be deformed to an exact Lagrangian concordance in the symplectization; see \cite{Chantraine, ChantraineColinDR, EkholmHondaKalman} for different versions of this construction. Here we revisit the version from \cite[Theorem 1.2]{Chantraine} and show that it fits our purposes as far as control on the length is concerned. The length of a Lagrangian cobordism was defined in \cite{SabloffTraynor15} by Sabloff--Traynor; see Section \ref{sec:Geometry}.

Let $\Lambda \subset (Y,\alpha)$ be a Legendrian submanifold of a contact manifold with a choice of contact form $\alpha$. Let $\phi^t_{\alpha, H} \colon (Y,\ker \alpha) \to (Y,\ker\alpha)$ be a contact isotopy with $\phi^0_{\alpha, H}=\id$, which thus is generated by a contact Hamiltonian $H_t \colon Y \to \R$ defined by $H_t \circ \phi^t_{\alpha, H}=\alpha(\dot\phi^t_{\alpha, H})$. Furthermore, let $f_t \colon Y \to \R$ be the smooth function for which $(\phi^t)_{\alpha, H}^*\alpha=e^{f_t}\alpha.$ The function $e^{f_t}$ is called the {\bf conformal factor} of the contact isotopy and was introduced in Section \ref{sec:Introduction}.
In particular, $(\tau,y) \mapsto (\tau-f_t(y),\phi^t_{\alpha, H}(y))$, is a Hamiltonian isotopy of the symplectization that is generated by the $t$-dependent Hamiltonian $e^\tau H_t \colon \R_\tau \times Y \to \R$. Note that this symplectic isotopy preserves the primitive $e^\tau\alpha$ of the symplectic form.

\begin{prop}
  \label{prop:trace}
For a contact isotopy as above, the following holds for any arbitrary choice of $\epsilon>0$:
\begin{enumerate}
\item There exists a Lagrangian trace cobordism $L_{01} \subset (\R_\tau \times Y,e^\tau\alpha)$ from $\Lambda$ to $\phi^1_{\alpha, H}(\Lambda)$ of length equal to $-e^{1+\epsilon}\min_{x \in Y, t \in [0,1]}f_t(x)\ge 0$
\item There exists a Lagrangian trace cobordism $L_{10} \subset (\R_\tau \times Y,e^\tau\alpha)$ from $\phi^1_{\alpha, H}(\Lambda)$ to $\Lambda$ of length equal to $e^{1+\epsilon}\max_{x \in Y, t \in [0,1]}f_t(x) \ge 0$.
\item One can assume that the two concatenations $L_{01} \odot L_{10}$ and $L_{10} \odot L_{01}$ of traces, which thus are Lagrangian cobordism from $\Lambda$ to itself and $\phi^1_{\alpha, H}(\Lambda)$ to itself, respectively, are of length
  $$c_0 \coloneqq e^{1+\epsilon}\left(\max_{x \in Y, t \in [0,1]}f_t(x)-\min_{x \in Y, t \in [0,1]}f_t(x)\right).$$
  Moreover, these concatenations are Hamiltonian isotopic to the trivial Lagrangian cylinders $\R \times \Lambda$ and $\R \times \phi^1_{\alpha, H}(\Lambda)$, respectively, by isotopies supported in a subset of the form $[0,c_0] \times Y$ (after a suitable translation of the $\tau$-coordinate).
\item All Lagrangian cobordisms constructed above have primitives of the pull-back of $e^\tau\alpha$ that can be taken to be globally constant on each cylindrical end $\pm \tau \gg 0.$
\end{enumerate}
\end{prop}
\begin{proof}
  (1): It suffices to consider the image of the trivial Lagrangian cylinder $\R \times \Lambda$ under the time-one map of the Hamiltonian isotopy defined by a Hamiltonian of the form $\rho(\tau)e^\tau H_t.$ Recall that $\tau$ is the symplectization coordinate here, while $t$ is the time-coordinate. We take $\rho(\tau) \colon \R_\tau \to 0$ to be a smooth function that vanishes on the subset $(-\infty,0]$ and is equal to one on the subset $[\delta,+\infty)$ for some $\delta>0$.

  So far we have merely repeated the argument from the proof of \cite[Theorem 1.2]{Chantraine}. To achieve the bound on the length, it suffices to choose $\delta>0$ sufficiently small, so that the inequality
  $$e^{1+\epsilon}\left(-\min_{x \in Y, t \in [0,1]}f_t(x) \right) \ge \delta-\min_{x \in Y, t \in [0,1]}f_t(x)$$
is satisfied. 
  To show the result follows from this inequality, we use the fact that $e^\tau H_t$ generates a Hamiltonian isotopy $(\tau,y) \mapsto (\tau-f_t(y),\phi^t_{\alpha, H}(y))$ of the symplectization, which means that $\max_{y \in Y} f_t(y)$ 
 is the maximal translation in the negative symplectization direction at time $t$.

  (2): This is similar to (1), using the fact that $(\phi^t_{\alpha, H})^{-1}$ again is a contact isotopy (thus it is generated by a contact Hamiltonian) whose conformal factor is equal to $-f_t$. We then apply the construction from (1) to the cylinder $\R \times \phi^1_{\alpha, H}(\Lambda)$ instead of $\R \times \Lambda$, while using the contact isotopy $(\phi^t_{\alpha, H})^{-1}$ instead of $\phi^t_{\alpha, H}$.

    (3): The constructions in (1) and (2) can be taken to depend smoothly on the family of contact isotopies $[0,1] \ni t \mapsto (\phi^t_{\alpha, H})_r \coloneqq  \phi^{rt}_{\alpha, H}$, where the family is parametrized by $r \in [0,1]$. Writing $e^{f_{r,t}}$ for the conformal factor of $(\phi^t_{\alpha, H})_r$ we immediately note that
    $$ \max_{x \in Y, t \in [0,1]}f_t(x) \ge \max_{x \in Y, t \in [0,1]}f_{r,t}(x)$$
    and
    $$ \min_{x \in Y, t \in [0,1]}f_{r,t}(x) \ge \min_{x \in Y, t \in [0,1]}f_t(x)$$
    holds for any $r \in [0,1]$, from which the sought length properties follow.

    In this manner we produce families of Lagrangian cobordisms whose concatenation $L_{01}^r \odot L_{10}^r$ (resp.~$L_{10}^r \odot L_{01}^r$) interpolate between $\R \times \Lambda$ (resp.~$\R \times \phi^1_{\alpha, H}(\Lambda)$) at $r=0$ and $L_{01} \odot L_{10}$ (resp.~$L_{10} \odot L_{01}$) at  $r=1$. The corresponding isotopy may be assumed to be supported inside $[0,c_0] \times Y$. Since this is an isotopy through exact Lagrangians, a standard fact implies that it is generated by a global Hamiltonian isotopy.

    (4):  As follows by Cartan's formula, a Hamiltonian isotopy $\phi^t_{\rho(\tau)e^\tau H_t} \colon \R \times Y \to \R \times Y$, with $\rho'(\tau)$ being of compact support, pulls back the primitive of the symplectic form $e^\tau\alpha$ to a one form $e^\tau\alpha+dG$. Since $\phi^t_{\rho(\tau)e^\tau H_t}$ preserves the primitive $e^\tau\alpha$ outside of a compact subset, we conclude that $G \colon Y \to \R$ is locally constant outside of a compact subset or, equivalently, $dG$ is compactly supported. The sought statement follows from this.
\end{proof}





\end{document}